\documentclass[11pt]{amsart}
\usepackage{amsmath,amssymb,amsthm,mathrsfs,enumerate,bm,xcolor,multirow,pbox}
\usepackage{graphicx,color,framed,tikz,caption,subcaption}
\usepackage{enumitem}
\setlist{leftmargin=5mm}
\allowdisplaybreaks[4]
\numberwithin{equation}{section}
\newcommand{\N}{\mathbb{N}}
\newcommand{\R}{\mathbb{R}}

\newcommand{\E}{\mathbb{E}}
\newcommand{\Prob}{\mathbb{P}}
\newcommand{\G}{\mathbb{G}}

\newcommand{\pnorm}[2]{\lVert#1\rVert_{#2}}

\newcommand{\abs}[1]{\lvert#1\rvert}
\newcommand{\bigabs}[1]{\big\lvert#1\big\rvert}
\newcommand{\biggabs}[1]{\bigg\lvert#1\bigg\rvert}

\renewcommand{\epsilon}{\varepsilon}

\renewcommand{\d}[1]{\mathrm{d}#1}

\newcommand{\floor}[1]{\left\lfloor #1 \right\rfloor}

\newcommand{\beq}{\begin{equation}}
\newcommand{\eeq}{\end{equation}}
\newcommand{\beqa}{\begin{equation} \begin{aligned}}
\newcommand{\eeqa}{\end{aligned} \end{equation}}
\newcommand{\beqas}{\begin{equation*} \begin{aligned}}
\newcommand{\eeqas}{\end{aligned} \end{equation*}}

\newcommand{\bit}{\begin{itemize}}
	\newcommand{\eit}{\end{itemize}}
\newcommand{\bmat}{\begin{bmatrix}}
	\newcommand{\emat}{\end{bmatrix}}

\theoremstyle{definition}
\theoremstyle{remark}\newtheorem{assumption}{Assumption}

\theoremstyle{remark}\newtheorem{remark}{Remark}
\theoremstyle{definition}\newtheorem{example}{Example}
\theoremstyle{plain}
\theoremstyle{plain}\newtheorem{theorem}{Theorem}
\theoremstyle{plain}\newtheorem{lemma}{Lemma}
\theoremstyle{plain}\newtheorem{proposition}{Proposition}
\theoremstyle{plain}\newtheorem{corollary}{Corollary}
\theoremstyle{plain}

\AtBeginDocument{%
	\def\MR#1{}
}

\begin{document}

\title[Limit distribution theory]{Limit distribution theory for block estimators in multiple isotonic regression}
\thanks{The research of Q. Han is partially supported by DMS-1916221. The research of C.-H. Zhang is partially supported  by DMS-1513378, DMS-1721495 and IIS-1741390. }

\author[Q. Han]{Qiyang Han}

\address[Q. Han]{
Department of Statistics, Rutgers University, Piscataway, NJ 08854, USA.
}
\email{qh85@stat.rutgers.edu}

\author[C.-H. Zhang]{Cun-Hui Zhang}

\address[C.-H. Zhang]{
Department of Statistics, Rutgers University, Piscataway, NJ 08854, USA.
}
\email{czhang@stat.rutgers.edu}

\date{\today}

\begin{abstract}
We study limit distributions for the tuning-free max-min block estimator originally proposed in \cite{fokianos2017integrated} in the problem of multiple isotonic regression, under both fixed lattice design and random design settings. We show that, if the regression function $f_0$ admits vanishing derivatives up to order $\alpha_k$ along the $k$-th dimension ($k=1,\ldots,d$) at a fixed point $x_0 \in (0,1)^d$, and the errors have variance $\sigma^2$, then the max-min block estimator $\hat{f}_n$ satisfies
\begin{align*}
(n_\ast/\sigma^2)^{\frac{1}{2+\sum_{k \in \mathcal{D}_\ast} \alpha_k^{-1}}}\big(\hat{f}_n(x_0)-f_0(x_0)\big)\rightsquigarrow \mathbb{C}(f_0,x_0).
\end{align*}
Here $\mathcal{D}_\ast, n_\ast$, depending on $\{\alpha_k\}$ and the design points, are the set of all `effective dimensions' and the size of `effective samples' that drive the asymptotic limiting distribution, respectively. If furthermore either $\{\alpha_k\}$ are relative primes to each other or all mixed derivatives of $f_0$ of certain critical order vanish at $x_0$, then the limiting distribution can be represented as $\mathbb{C}(f_0,x_0) =_d K(f_0,x_0) \cdot \mathbb{D}_{\bm{\alpha}}$, where $K(f_0,x_0)$ is a constant depending on the local structure of the regression function $f_0$ at $x_0$, and $\mathbb{D}_{\bm{\alpha}}$ is a non-standard limiting distribution generalizing the well-known Chernoff distribution in univariate problems. The above limit theorem is also shown to be optimal both in terms of the local rate of convergence and the dependence on the unknown regression function whenever such dependence is explicit (i.e. $K(f_0,x_0)$), for the full range of $\{\alpha_k\}$ in a local asymptotic minimax sense. 

There are two interesting features in our local theory. First, the max-min block estimator automatically adapts to the local smoothness and the intrinsic dimension of the isotonic regression function at the optimal rate. Second, the optimally adaptive local rates are in general not the same in fixed lattice and random designs. In fact, the local rate in the fixed lattice design case is no slower than that in the random design case, and can be much faster when the local smoothness levels of the isotonic regression function or the sizes of the lattice differ substantially along different dimensions.
\end{abstract}

\maketitle

\section{Introduction}

\subsection{Overview}

Limit distribution theory for shape-restricted estimators is of fundamental importance in the area of statistical inference under shape restrictions. There are two main types of limit distribution theories so far available in the literature.

One line starts from the seminal contribution of \cite{rao1969estimation}, who showed that the limiting distribution of the maximum likelihood estimator (MLE) of a decreasing density on $[0,\infty)$ (known as Grenander estimator) at a fixed point is given by the following: Suppose the true density $f_0$ is decreasing on $[0,\infty)$ and continuously differentiable around $x_0 \in (0,\infty)$ with $f_0'(x_0)<0$. Then the MLE $\hat{f}_n$ satisfies
\begin{align}\label{eqn:limit_dist_monotone}
n^{1/3}\big(\hat{f}_n(x_0)-f_0(x_0)\big)\rightsquigarrow \abs{f'_0(x_0) f_0(x_0)/2}^{1/3} \mathbb{Z}.
\end{align}
Here $\mathbb{Z}$, known as the Chernoff distribution, is the slope at zero of the least concave majorant of the process $t\mapsto \mathbb{B}(t)-t^2$, where $\mathbb{B}$ is the standard Brownian motion starting at $0$. Later on, \cite{groeneboom1989brownian} gives an exact analytic characterization of the limiting Chernoff distribution, whereas \cite{groeneboom1985estimating} suggests the `switching relation' that quickly becomes popular as a powerful proof technique in univariate problems with monotonicity shape restrictions. The limiting Chernoff distribution arises in a number of different problems with univariate monotonicity shape restrictions, e.g., (1) estimation of a regression function \cite{brunk1970estimation,vaneeden1958testing,wright1981asymptotic}, (2) estimation of a monotone failure rate \cite{huang1995estimation,huang1994estimating,rao1970estimation}, (3) estimation in interval censoring models \cite{groeneboom1996lectures,groeneboom1992information}, etc. We refer the reader to the recent survey \cite{durot2018limit} for extensive references in this direction.

Another line of limit theorems for shape restricted estimators is pioneered by \cite{groeneboom2001canonical,groeneboom2001estimation}, who studied limit distribution for the MLE of a convex decreasing density on $[0,\infty)$ and the least squares estimator (LSE) of a convex regression function at a fixed point. In the density setting, if the true density $f_0$ is convex decreasing on $[0,\infty)$ and twice continuously differentiable in a neighborhood of $x_0$ with $f_0''(x_0)>0$, \cite{groeneboom2001estimation} showed that the MLE $\hat{f}_n$ satisfies
\begin{align}\label{eqn:limit_dist_convex}
n^{2/5}\big(\hat{f}_n(x_0)-f_0(x_0)\big)\rightsquigarrow \big(f_0''(x_0) f_0^2(x_0)/24\big)^{1/5} \mathbb{H}''(0),
\end{align}
where $\mathbb{H}$ is a particular upper invelope of an integrated two-sided Brownian motion plus $t^4$, cf. \cite{groeneboom2001canonical}. The process $\mathbb{H}$ appears in several other problems involving univariate convexity shape restrictions, e.g.  for the MLE of a log-concave density on $\R$, cf. \cite{balabdaoui2009limit}, for the MLE of a convex bathtub-shaped hazard function, cf. \cite{jankowski2009nonparametric}, for the R\'enyi-divergence estimators for $s$-concave densities on $\R$, cf. \cite{han2015approximation}, etc. A generalized version of $\mathbb{H}$ appears in \cite{balabdaoui2007estimation} in the context of $k$-monotone density estimation.

Limit theorems of types (\ref{eqn:limit_dist_monotone})-(\ref{eqn:limit_dist_convex}) are not only interesting from a statistical point of view, but are of theoretical value in their own rights. Indeed, these limit theorems and the proof techniques used therein serve as fundamental building blocks for numerous further developments, including likelihood based inferential methods \cite{banerjee2007likelihood,banerjee2001likelihood,doss2016inference,groeneboom2015nonparametric}, bootstrap in non-standard problems \cite{kosorok2008bootstrapping,sen2010inconsistency}, estimation and inference with dependence structures \cite{anevski2006general,anevski2011monotone,bagchi2016inference}, limit theory for global loss functions and functionals \cite{durot2007error,durot2012limit,groeneboom1999asymptotic,jankowski2014convergence,kulikov2005asymptotic}, limit distribution theory for shape-restricted estimators of discrete functions \cite{balabdaoui2017bernoulli,balabdaoui2016maximum,balabdaoui2013asymptotics,jankowski2009estimation}, limit distribution theory for split points in decision trees \cite{banerjee2007confidence,buhlmann2002analyzing}, cube-root asymptotics in more general settings \cite{anevski2011monotone,kim1990cube}, just to name a few.

Despite the wealth of limit distribution theories for univariate shape-restricted problems, much less is known in multi-dimensional settings. The only exception we are aware of is the recent work \cite{anevski2018asymptotic}, in which asymptotic distributions for isotonized estimators are derived in the settings of multi-dimensional discrete isotonic regression and probability mass function estimation. The goal of this paper is to study limit theorems for shape-restricted estimators, with a focus on the problem of multiple isotonic regression in a continuous setting. 

Here is our setup. Consider the regression model 
\begin{align}\label{model}
Y_i = f_0(X_i)+\xi_i,\quad i=1,\ldots,n,
\end{align}
where $X_1,\ldots,X_n$ are design points which can be either fixed or random, and $\xi_1,\ldots,\xi_n$ are random errors. By multiple isotonic regression we assume that the regression function $f_0 \in \mathcal{F}_d$, where $\mathcal{F}_d$ denotes the class of coordinate-wise nondecreasing functions on $[0,1]^d$:
\begin{align*}
\mathcal{F}_d\equiv \{f:[0,1]^d \to \R, f(x)\leq f(y) \textrm{ if }x_i\leq y_i\textrm{ for all }i=1,\ldots,d\}.
\end{align*}

In addition to the aforementioned importance of having a limit distribution theory for shape restricted estimators beyond univariate settings, there is one further consideration for such a theory, related to one distinct attractive feature of shape-constrained models: the MLE/LSE often exists and automatically adapts to certain structures of the underlying truth without the need of any tuning. This automatic adaptation property has attracted a lot of recent attention, mostly from a global perspective. Indeed, adaptation of shape constrained MLEs/LSEs to piecewise simple structures in global metrics is confirmed extensively in various univariate models, cf. \cite{bellec2018sharp,chatterjee2015risk,guntuboyina2013global,kim2016adaptation,zhang2002risk}. Typically, these tuning-free estimators adapt to piecewise constant/linear signals in univariate models with monotonicity/convexity shape constraints. For the multiple isotonic regression model (\ref{model}), global adaptation of the LSE to piecewise constant signals is proved for the bivariate case $d=2$  in \cite{chatterjee2018matrix}, and for the case of general dimensions in \cite{han2017isotonic}. See also \cite{han2019}. Despite these very positive adaptation results, there remain two important drawbacks for considering global adaptation of the natural LSE in the multiple isotonic regression model:
\begin{enumerate}
	\item[(D1)] The isotonic regression function is of global smoothness level $1$ or $\infty$, so the LSE can adapt, if at all possible, to limited global structures. In fact, piecewise constancy is the only known global structure to which the LSE is confirmed to adapt, cf. \cite{chatterjee2018matrix,han2017isotonic}.
	\item[(D2)] The LSE does not adapt at the optimal rate to constant signals when $d\geq 3$: it is shown in \cite{han2017isotonic} that the LSE adapts to the global constant structures at a \emph{strictly sub-minimax rate} $n^{-1/d}$ (up to logarithmic factors) in $L_2$-type losses. 
\end{enumerate}
The reasons for these limitations, however, lie in very different places: the drawback in (D1) is due to the perspective of considering global adaptation, while the drawback in (D2) is due to the use of the LSE.

In view of these limitations, in this paper we consider the local behavior of the following alternative max-min block estimator originally proposed by \cite{fokianos2017integrated}: for any $x_0 \in [0,1]^d$,
\begin{align}\label{def:max_min_estimator}
\hat{f}_n(x_0) &\equiv \max_{x_u\leq x_0} \min_{ \substack{x_v\geq x_0\\ [x_u,x_v]\cap \{X_i\}\neq \emptyset}} \frac{\sum_{{i: x_u\leq X_i\leq x_v}} Y_i}{\abs{ \{i: x_u\leq X_i\leq x_v\}}}\\
&= \max_{x_u\leq x_0} \min_{ \substack{x_v\geq x_0\\ [x_u,x_v]\cap \{X_i\}\neq \emptyset}} \bar{Y}|_{[x_u,x_v]} \nonumber\\
&= \max_{x_u\leq x_0} \min_{ \substack{x_v\geq x_0\\ [x_u,x_v]\cap \{X_i\}\neq \emptyset}} \bar{Y}|_{[x_u,1]\cap [0,x_v]},\nonumber
\end{align}
Here $[x_u,x_v]=\{x\in \R^d: x_u\le x\le x_v\}$, 
$\bar{Y}|_A$ is the average of $\{Y_i: X_i\in A\}$ as in (\ref{def:average}), 
and for any $x, y \in \R^d$, $x\leq y$ if and only if $x_j\leq y_j$ for all $1\leq j\leq d$, and the similar definition applies to $\geq$. It is easy to see that $\hat{f}_n \in \mathcal{F}_d$ and is tuning-free. The computation for (\ref{def:max_min_estimator}) is exact and requires at most $\mathcal{O}(n^2)$ for each design point, so the total computational complexity is at most $\mathcal{O}(n^3)$, independent of the dimension $d$. 

The max-min block estimator (\ref{def:max_min_estimator}) above is closely related to the LSE studied in \cite{han2017isotonic}, in the sense that the LSE also admits a max-min representation \cite{robertson1988order}, but with the rectangles $[x_u,1], [0,x_v]$ replaced by all \emph{upper sets} and \emph{lower sets} containing $x_0$. Since the upper and lower sets reduce to intervals in dimension one, (\ref{def:max_min_estimator}) coincides with the standard univariate isotonic LSE in $d=1$. The representation through upper and lower sets is also observed in a related monotone density estimation problem in $d=2$, cf. \cite{polonik1998silhouette}.

The max-min representation gives one heuristic explanation for the difficulty of the LSE in the sense of
(D2): the class of upper and lower sets is too large for the partial sum process to remain tight in the large sample limit as soon as $d\geq 3$ (cf, \cite{dudley1999uniform}). On the other hand, using the smaller class of rectangles as in (\ref{def:max_min_estimator}), it is shown in \cite{deng2018isotonic} that (\ref{def:max_min_estimator}) does adapt to constant signals at a nearly optimal parametric rate in all dimensions, as opposed to the slower rate $n^{-1/d}$ for the LSE, cf. (D2). For the same reason, it is hard to expect a limiting distribution theory for the LSE.

The main contribution of this paper is to develop a limit distribution theory for the max-min block estimator (\ref{def:max_min_estimator}). We show that, the limiting distribution of $\hat{f}_n$, depending on the local structure of $f_0$ at $x_0$, takes the following general form: Suppose $f_0$ admits vanishing derivatives up to order $\alpha_k$ along the $k$-th dimension ($k=1,\ldots,d$) at a fixed point $x_0 \in (0,1)^d$, and the errors $\{\xi_i\}$ have variance $\sigma^2$. Then
\begin{align}\label{eqn:limit_distribution}
(n_\ast/\sigma^2)^{\frac{1}{2+\sum_{k \in \mathcal{D}_\ast} \alpha_k^{-1}}}\big(\hat{f}_n(x_0)-f_0(x_0)\big)\rightsquigarrow \mathbb{C}(f_0,x_0).
\end{align}
Here $\mathcal{D}_\ast$ and $n_\ast$, determined by the value of $(\alpha_1,\ldots,\alpha_d)$ and the design of $\{X_i\}$, are the set of all `effective dimensions' and the size of `effective samples' that drive the asymptotic limiting distribution, the exact meaning of which will be clarified in Section \ref{section:limit_distribution}. The dependence of the limiting distribution $\mathbb{C}(f_0,x_0)$ on the local properties of $f_0$ at $x_0$ cannot be in general expressed by a simple factor, due to possible existence of non-zero mixed derivatives of critical order $\bm{j}=(j_1,\ldots,j_d)$ satisfying $\sum_{k=1}^d j_k/\alpha_k=1$ and $\pnorm{\bm{j}}{0}>1$. However, in situations where $\{\alpha_k\}$ are relative primes to each other (so that any such index vector $\bm{j}$ must have $\pnorm{\bm{j}}{0}=1$), or all mixed derivatives of $f_0$ of the critical order vanish at $x_0$, the limiting distribution $\mathbb{C}(f_0,x_0)$ can be represented in a similar form as in (\ref{eqn:limit_dist_monotone})-(\ref{eqn:limit_dist_convex}), namely
\begin{align*}
\mathbb{C}(f_0,x_0) =_d K(f_0,x_0)\cdot \mathbb{D}_{\bm{\alpha}}.
\end{align*}
Here $K(f_0,x_0)$ is a constant depending on the local structure of the regression function $f_0$ at $x_0$ to be specified in Section \ref{section:limit_distribution}, and $\mathbb{D}_{\bm{\alpha}}$ is the non-standard limiting distribution playing the similar role as the Chernoff distribution $\mathbb{Z}$ in univariate problems. 

One important and canonical setting for (\ref{eqn:limit_distribution}) is the following: Suppose (i) $f_0$ depends only through its first $s$ coordinates ($0\leq s\leq d$), and all non-trivial first-order partial derivatives of $f_0$ are non-vanishing at $x_0$: $\partial_k f_0(x_0)>0,1\leq k\leq s$, and (ii) the design points $\{X_i\}$ are either of a balanced fixed lattice design (see Section \ref{section:limit_distribution} for a precise definition) or a random design with uniform distribution on $[0,1]^d$. In this setting,  (\ref{eqn:limit_distribution}) reduces to:
\begin{align*}
(n/\sigma^2)^{\frac{1}{2+s}}\big(\hat{f}_n(x_0)-f_0(x_0)\big)\rightsquigarrow \bigg\{\prod_{k=1}^s \big(\partial_k f_0(x_0)/2\big)\bigg\}^{\frac{1}{2+s}}\cdot \mathbb{D}_{(\underbrace{1,\ldots,1}_{s\textrm{ many }1\textrm{'s}},\infty,\ldots,\infty)}.
\end{align*}
When $s=d=1$, we recover the familiar limit distribution theory for univariate isotonic least squares estimator.

The limit theory in (\ref{eqn:limit_distribution}), as we will see in Section \ref{section:limit_distribution}, implies that \emph{the max-min block estimator (\ref{def:max_min_estimator}) automatically adapts to the local smoothness structures and the intrinsic dimension of $f_0$}. The local adaptation is in similar spirit to \cite{balabdaoui2009limit,chen2014convex,wright1981asymptotic}, who showed that univariate shape-restricted MLEs/LSEs adapt to local smoothness of the truth. It should be emphasized here that local smoothness to which adaptation occurs specifically refers to the number of \emph{vanishing} (partial) derivatives. A distinct feature for the max-min block estimator (\ref{def:max_min_estimator}) here is that both (i) the local rate of convergence, i.e. $n_\ast^{\frac{1}{2+\sum_{k \in \mathcal{D}_\ast} \alpha_k^{-1}}}$, and (ii) the dependence on $\{f_0,x_0\}$ whenever explicit, i.e. the constant $K(f_0,x_0)$ in the limit distribution, are \emph{optimal in a local asymptotic minimax sense for all possible local smoothness levels}. So in this sense the limit distribution theory for the max-min block estimator (\ref{def:max_min_estimator}) in the form of (\ref{eqn:limit_distribution}) is the best one can hope for in the problem of multiple isotonic regression.

Another interesting consequence of (\ref{eqn:limit_distribution}) and its local asymptotic minimaxity is that \emph{the optimal local rates of convergence are in general not the same in fixed lattice and random designs}. In fact, the local rate in the fixed lattice design case is \emph{no slower} than that in the random design case, and can be much faster when (a) the local smoothness levels of the isotonic regression function, or (b) the sizes of the lattice, differ substantially along different dimensions. The reason for the discrepancy in the local rates can be attributed to the fact that significant imbalance in (a) or (b) screens out dimensions with `low regularity'  that do not contribute to the asymptotics in the fixed lattice design case. Here dimensions with `low regularity', loosely speaking, refer to those with low smoothness levels in (a), and to those with sparsely spaced design points in (b).

The proof of the limit theory (\ref{eqn:limit_distribution}) in general dimensions is significantly more challenging than its univariate counterpart. Indeed, thanks to the `switching relation' put forward by \cite{groeneboom1985estimating}, it is now well understood that limiting distributions for various univariate monotonicity shape constrained estimators can be obtained via the argmax continuous mapping theory, upon a proper \emph{one-sided} localization (typically) on the order of a cubic-root rate (cf. \cite{van1996weak}). In contrast, in the multiple isotonic regression problem we consider here, the key step in the proof is a \emph{two-sided} localization technique, where the stochastic orders of the length of all sides of the rectangle over which the max-min block estimator (\ref{def:max_min_estimator}) takes average, need to be estimated sharply \emph{both from above and below}. These estimates bring about substantial technical challenges as opposed to univariate problems.

Finally we mention the work of \cite{deng2018isotonic}, in which global risk bounds in $L_q$ norms for the max-min block estimator (\ref{def:max_min_estimator}) are thoroughly studied. Risk bounds in global metrics, as already mentioned in (D1), have a limited scope of structures for adaptation due to the strict global smoothness of the isotonic functions. Our local limit distribution theory (\ref{eqn:limit_distribution}) can therefore also be viewed as a further step in understanding the adaptive behavior of the max-min block estimator (\ref{def:max_min_estimator}) to a rich class of structures that are exhibited only through local properties of the isotonic regression function.

The rest of the paper is organized as follows. In Section \ref{section:limit_distribution}, we present the limit distribution theory (\ref{eqn:limit_distribution}) for the max-min block estimator (\ref{def:max_min_estimator}), and discuss its many implications. In Section \ref{section:local_minimax}, we establish a local asymptotic minimax lower bound, showing the information-theoretic optimality of the limit theorem (\ref{eqn:limit_distribution}). Due to the highly technical nature of the proofs, Section \ref{section:proof_outline} is devoted to an outline of the main ideas in the proofs. Section \ref{section:discussion} concludes the paper with a brief discussion. All the proof details are presented in Sections \ref{section:proof_limit_dist}-\ref{section:proof_auxiliary}.

\subsection{Notation}\label{section:notation}

For a real-valued measurable function $f$ defined on $(\mathcal{X},\mathcal{A},P)$, $\pnorm{f}{L_p(P)}\equiv \pnorm{f}{P,p}\equiv \big(P\abs{f}^p)^{1/p}$ denotes the usual $L_p$-norm under $P$, and $\pnorm{f}{\infty}\equiv \sup_{x \in \mathcal{X}} \abs{f(x)}$. Let $(\mathcal{F},\pnorm{\cdot}{})$ be a subset of the normed space of real functions $f:\mathcal{X}\to \R$. For $\epsilon>0$ let $\mathcal{N}(\epsilon,\mathcal{F},\pnorm{\cdot}{})$ be the $\epsilon$-covering number of $\mathcal{F}$; see page 83 of \cite{van1996weak} for more details.

For the regression model (\ref{model}), for any $A \subset [0,1]^d$, define
\begin{align}\label{def:average}
\bar{Y}|_A \equiv \frac{1}{n_A }\sum_{i: X_i \in A} Y_i, \bar{f_0}|_A \equiv \frac{1}{ n_A }\sum_{i: X_i \in A} f_0(X_i), \bar{\xi}|_A \equiv \frac{1}{ n_A }\sum_{i: X_i \in A} \xi_i
\end{align}
where $n_A\equiv \abs{\{i: X_i \in A\}}$.

For two real numbers $a,b$, $a\vee b\equiv \max\{a,b\}$ and $a\wedge b\equiv\min\{a,b\}$. For $x \in \R^d$, let $\pnorm{x}{p}$ denote its $p$-norm $(0\leq p\leq \infty)$. For any $x, y \in \R^d$, let $[x,y] \equiv \prod_{k=1}^d [x_k\wedge y_k,x_k\vee y_k]$, $xy\equiv (x_ky_k)_{k=1}^d$, and $x\wedge (\vee) y \equiv (x_k \wedge (\vee) y_k)_{k=1}^d$. For $\ell_1, \ell_2 \in \{1,\ldots,d\}$, we let $\bm{1}_{[\ell_1:\ell_2]} \in \R^d$ be such that $(\bm{1}_{[\ell_1:\ell_2]})_k = \bm{1}_{\ell_1\leq k\leq \ell_2}$, and  $\bm{1}\equiv \bm{1}_{[1:d]}$ for simplicity. $C_{x}$ will denote a generic finite constant that depends only on a generic quantity $x$, whose numeric value may change from line to line unless otherwise specified. $a\lesssim_{x} b$ and $a\gtrsim_x b$ mean $a\leq C_x b$ and $a\geq C_x b$ respectively, and $a\asymp_x b$ means $a\lesssim_{x} b$ and $a\gtrsim_x b$ [$a\lesssim b$ means $a\leq Cb$ for some absolute constant $C$]. $\mathcal{O}_{\mathbf{P}}$ and $\mathfrak{o}_{\mathbf{P}}$ denote the usual big and small O notation in probability. $\rightsquigarrow$ is reserved for weak convergence. For two integers $k_1>k_2$, we interpret $\sum_{k=k_1}^{k_2}\equiv 0, \prod_{k=k_1}^{k_2}\equiv 1$. We also interpret $(\infty)^{-1}\equiv 0, 0/0\equiv 0$. 

For $f: \R^d \to \R$, and $k \in \{1,\ldots,d\}$, $\alpha_k \in \mathbb{Z}_{\geq 1}$, let $\partial_k^{\alpha_k} f(x)\equiv \frac{d^{\alpha_k}}{d x_k^{\alpha_k}}f(x)$. For a multi-index $\bm{j}=(j_1,\ldots,j_d) \in \mathbb{Z}_{\geq 0}^d$, let $\partial^{\bm{j}} \equiv \partial_{1}^{j_1}\cdots \partial_d^{j_d}$, and $\bm{j}! \equiv j_1!\cdots j_d!$ and $x^{\bm{j}} \equiv x_1^{j_1}\ldots x_d^{j_d}$ for $x \in \R^d$. For $\bm{\alpha}=(\alpha_1,\ldots,\alpha_d) \in \mathbb{Z}_{\geq 1}^d$ in Assumption \ref{assump:smoothness} below, i.e. for some $0\leq s\leq d$, $1\leq \alpha_1,\ldots,\alpha_s<\infty = \alpha_{s+1}=\ldots=\alpha_d$, let $J(\bm{\alpha})$ (resp. $J_\ast(\bm{\alpha})$) be the set of all $\bm{j}=(j_1,\ldots,j_d) \in \mathbb{Z}_{\geq 0}^d$ satisfying $0<\sum_{k=1}^s j_k/\alpha_k \leq 1$ (resp. $\sum_{k=1}^s j_k/\alpha_k = 1$) and $j_k = 0$ for $s+1\leq k\leq d$, and let $J_0(\bm{\alpha})\equiv J(\bm{\alpha})\cup \{\bm{0}\}$. We often write $J=J(\bm{\alpha})$, $J_\ast = J_\ast(\bm{\alpha})$ and $J_0 = J_0(\bm{\alpha})$ if no confusion arises. The set $J,J_\ast$ will play a crucial role below in determining $\bm{j}$'s for which $\partial^{\bm{j}}f_0(x_0)$ can be non-zero under Assumption \ref{assump:smoothness}; cf. Lemma \ref{lem:mixed_derivative_vanish}.

\section{Limit distribution theory}\label{section:limit_distribution}

\subsection{Assumptions}
We first state the assumptions on the local smoothness of $f_0$ at the point of interest $x_0 \in (0,1)^d$ and the intrinsic dimension of $f_0$.

\begin{assumption}\label{assump:smoothness}
$f_0$ is coordinate-wise nondecreasing (i.e. $f_0 \in \mathcal{F}_d$), and is $\bm{\alpha}$-smooth at $x_0$ with intrinsic dimension $s$, $\bm{\alpha}=(\alpha_1,\ldots,\alpha_d)$ with integers $1\leq \alpha_1,\ldots,\alpha_s<\infty = \alpha_{s+1}=\ldots=\alpha_d$, $0\leq s\leq d$, in the sense that $\partial_k^{j_k} f_0(x_0)=0$ for $1\leq j_k\leq \alpha_k-1$ and $\partial_k^{\alpha_k} f_0(x_0)\neq 0$, $1\leq k\leq s$, and in rectangles of the form $\cap_{k=1}^d \{\abs{(x-x_0)_k}\leq L_0\cdot(r_n)_k\}$, $r_n = (\omega_n^{1/\alpha_1},\ldots,\omega_n^{1/\alpha_d})$ with $\omega_n>0$, the Taylor expansion of $f_0$ satisfies for all $L_0>0$,
	\begin{align*}
	\lim_{\omega_n \searrow 0} \omega_n^{-1} \sup_{ \substack{x \in[0,1]^d,\\ \abs{(x-x_0)_k}\leq L_0 \cdot (r_n)_k, \\1\leq k\leq d} } \biggabs{f_0(x)- \sum_{\bm{j} \in J_0} \frac{\partial^{\bm{j}}f_0(x_0)}{\bm{j}!}(x-x_0)^{\bm{j}} }= 0.
	\end{align*}
\end{assumption}
Assumption \ref{assump:smoothness} concerns the local smoothness of $f_0$ at a fixed point $x_0$, allowing for potentially different local smoothness levels along different coordinates $\{1,\ldots,s\}$. The Taylor expansion, which includes all terms of order $\omega_n$ or larger in a small hyper-rectangle, 
interestingly features different rates $\omega_n^{1/\alpha_k}$ in different dimensions in the $x$-domain. 
This is quite different from the Taylor expansion in Euclidean balls which includes all terms with 
$\|\bm{j}\|_1\le\alpha$ for a certain smoothness index $\alpha$.  
Our expansion has the prescribed convergence rate if $f_0$ is locally $C^{\max_{1\leq k\leq s} \alpha_k}$ at $x_0$ and depends only through its first $s$ coordinates. Note that Assumption \ref{assump:smoothness} is interesting mostly from a local perspective. Indeed, if this condition holds for all $x_0 \in (0,1)^d$ with some $1\leq \alpha_1, \ldots, \alpha_s<\infty$, then we must have $\alpha_1=\ldots=\alpha_s=1$.

Now we consider a few examples that satisfy Assumption \ref{assump:smoothness} with different values of $\bm{\alpha}$'s. In the following examples we consider $d=2$ and $x_0 = (1/2,1/2)$ unless otherwise specified.
\begin{example}
Let $f_0^{(1)}(x_1,x_2) = x_1+x_2$. Then $\alpha_1=\alpha_2=1$.
\end{example}
\begin{example}
Let $f_0^{(2)}(x_1,x_2)=x_1$. Then $s=1$ with $\alpha_1=1,\alpha_2=\infty$.
\end{example}
\begin{example}
Let $f_0^{(3)}(x_1,x_2)= (x_1+x_2)\bm{1}_{0\leq x_1\leq 1/4}+8x_1\cdot\bm{1}_{1/4<x_1<3/4}+8(x_1+x_2)\bm{1}_{3/4\leq x_1\leq 1}$. Then $s=1$ with $\alpha_1=1,\alpha_2=\infty$ for $x_0 \in (1/4,3/4)\times (0,1)$, and $s=2$ with $\alpha_1=\alpha_2=1$ for $x_0 \in (0,1/4)\times (0,1) \cup (3/4,1) \times  (0,1)$.
\end{example}
Example 2 is a canonical example for which the regression function is globally of intrinsic dimension 1, while in Example 3 the function can be locally of intrinsic dimension 1 in the strip $(1/4,3/4)\times (0,1)$.
\begin{example}
Let $f_0^{(4)}(x_1,x_2) = (x_1-1/2)^3+(x_2-1/2)^3$. Then $\alpha_1 = 3,\alpha_2 =3$.
\end{example}
\begin{example}
Let $f_0^{(5)}(x_1,x_2) =(x_1-1/2)^3+(x_1-1/2)^2(x_2-1/2)+(x_1-1/2)(x_2-1/2)^2+(x_2-1/2)^3$. Then $\alpha_1=3,\alpha_2=3$.
\end{example}
Example 4 and Example 5 both share the same local smoothness level $\bm{\alpha}=(3,3)$, but are quite different in that for $f_0^{(4)}$ all mixed derivatives vanish, while for $f_0^{(5)}$ certain mixed derivatives do no vanish: $\partial^{\bm{j}} f_0^{(5)}(x_0)\neq 0$ for $\bm{j} \in \{(1,2),(2,1)\}$.

\begin{lemma}\label{lem:mixed_derivative_vanish}
The following statements hold:
\begin{enumerate}
	\item Suppose Assumption \ref{assump:smoothness} holds. $\alpha_k$ must be odd and $\partial_k^{\alpha_k} f_0(x_0)>0$ for $1\leq k\leq s$;
	\item Suppose Assumption \ref{assump:smoothness} holds. Any mixed derivative of the form $\partial^{\bm{j}}f_0(x_0)$, $\bm{0}\neq \bm{j} \in J\setminus J_\ast$, vanishes at $x_0$, and thus for some $\epsilon_1>0$, $L_1>0$ depending only on $f_0,x_0$, 
	\begin{align*}
	\sup_{0<\omega_n\leq \epsilon_1}\sup_{ \substack{x \in[0,1]^d,\\ \abs{(x-x_0)_k}\leq (r_n)_k, \\1\leq k\leq d}} \omega_n^{-1} \abs{f_0(x)-f_0(x_0)}\leq L_1;
	\end{align*}
	\item Let $J_1 \equiv \{\bm{j} \in J_\ast: \pnorm{\bm{j}}{0}>1 \}$. Then $J_1=\emptyset$ if and only if $\abs{J_\ast}=s$, if and only if $\{\alpha_k\}_{k=1}^s$ is a set of relative primes, i.e. the greatest common divisor of $\{\alpha_{k_1},\alpha_{k_2}\}$ is 1 for all $1\leq k_1<k_2\leq s$;
	\item When $\bm{\alpha}$ is such that $J_1 \neq \emptyset$, there exists some $f \in \mathcal{F}_d$ for which $f$ satisfies Assumption \ref{assump:smoothness} with $\bm{\alpha}$, but $\partial^{\bm{j}}f(x_0)\neq 0$ for some $\bm{j} \in J_1$.
\end{enumerate}
\end{lemma}
\begin{proof}
See Section \ref{section:proof_lemmas}.
\end{proof}

Lemma \ref{lem:mixed_derivative_vanish} reveals an important  and unique feature of multiple isotonic functions compared with smooth functions: If $f_0$ satisfies the `marginal smoothness' Assumption \ref{assump:smoothness} with $\bm{\alpha}=(\alpha_1,\ldots,\alpha_d)$ at $x_0$, then the only possible non-zero mixed derivatives $\partial^{\bm{j}}f_0(x_0)$ in the Taylor expansion must have critical order $\bm{j} \in J_\ast$ satisfying $\sum_{k=1}^s j_k/\alpha_k =1$. Such possible non-zero mixed derivatives cannot be ruled out under Assumption \ref{assump:smoothness} as soon as certain pair of $\{\alpha_k\}$ has a non-trivial common divisor. The importance of such a feature lies in the fact that these mixed derivatives $\{\partial^{\bm{j}}f_0(x_0): \bm{j} \in J_\ast\}$ contribute to the convergence rate of the same order as the marginal derivatives $\{\partial_k^{\alpha_k} f_0(x_0)\}$, in rectangles of the form $\cap_{k=1}^d \{\abs{(x-x_0)_k}\leq L_0\cdot(r_n)_k\}$, $r_n = (\omega_n^{1/\alpha_1},\ldots,\omega_n^{1/\alpha_d})$ with $\omega_n\searrow 0$. Hence adaptation of the max-min estimator (\ref{def:max_min_estimator}) to marginal smoothness levels---which only uses marginal information in rectangles---becomes possible. 

Next we state the assumptions on the design of the covariates.

\begin{assumption}\label{assump:design}
The design points $\{X_i\}_{i=1}^n$ satisfy either of the following:
\begin{itemize}
	\item (\emph{Fixed design}) $\{X_i\}$'s follow a $\bm{\beta}$-fixed lattice design: there exist some $\{\beta_1,\ldots,\beta_d\}\subset (0,1)$ with $\sum_{k=1}^d \beta_k=1$ such that $x_0 \in \{X_i\}_{i=1}^n = \prod_{k=1}^d \{x_{1,k},\ldots,x_{n_k,k}\}$, where $\{x_{1,k},\ldots,x_{n_k,k}\}$ are equally spaced in $[0,1]$ (i.e. $\abs{x_{j,k}-x_{j+1,k}}=1/n_k$ for all $j=1,\ldots,n_k-1$) and $n_k = \floor{n^{\beta_k}}$. 
	\item (\emph{Random design}) $\{X_i\}$'s follow i.i.d. random design with law $P$ independent of $\{\xi_i\}$'s. The Lebesgue density $\pi$ of $P$ is bounded away from $0$ and $\infty$ on $[0,1]^d$ and is continuous over an open set containing the region $\big\{\big((x_0)_1,\ldots,(x_0)_s,x_{s+1},\ldots,x_d\big): 0\leq x_k\leq 1, s+1\leq k\leq d\big\}$.
\end{itemize}
\end{assumption}

In the fixed lattice design case, we use $\beta_k$ to control the size of the lattice in dimension $k$. A balanced fixed lattice design refers to the special case with $\beta_k=1/d$ for all $k=1,\ldots,d$. In the random design case, the continuity of the density $\pi$ is imposed over the region where asymptotics take place.

We choose without loss of generality the index $\{1,\ldots,d\}$ such that
\begin{align}\label{def:ordering}
0\leq \alpha_1 \beta_1 \leq \ldots \leq \alpha_s\beta_s\leq \ldots \leq \alpha_d \beta_d \leq \infty.
\end{align}
This requirement facilitates the statement of our main Theorem \ref{thm:limit_distribution_pointwise} below. Otherwise we may find some permutation $\tau$ of $\{1,\ldots,d\}$ for which $\alpha_{\tau(1)}\beta_{\tau(1)}\leq \ldots \leq \alpha_{\tau(d)}\beta_{\tau(d)}$, and consider the coordinate-wise nondecreasing function $\tilde{f}_0(x_1,\ldots,\ldots,x_d)\equiv f_0(x_{\tau(1)},\ldots,x_{\tau(d)})$. Such a reparametrization is compatible with Assumption \ref{assump:smoothness} since $\alpha_k\beta_k=\infty$ if and only if $\alpha_k = \infty$.

\subsection{Limit distribution theory}
Let $x_0 \in (0,1)^d$. Let $\pi(x)\equiv 1$ in the $\bm{\beta}$-fixed lattice design case and $\pi(x)\equiv \d{P}/(\d{x_1}\cdots \d{x_d})$ be the Lebesgue density of $P$ in the random design case. For any $0\leq s\leq d$, and $h_1,h_2 \in \R_{\geq 0}^d$ such that $(h_1)_k\leq (x_0)_k, (h_2)_k\leq (1-x_0)_k$ for all $s+1\leq k\leq d$, let
\begin{align*}
&\mathcal{I}^{[s+1:d]}_\pi (h_1,h_2)\\
&\equiv \int_{ \substack{(x_0-h_1)_k\leq x_k \leq (x_0+h_2)_k\\ s+1\leq k\leq d}} \pi\big((x_0)_1,\ldots,(x_0)_s,x_{s+1},\ldots,x_d\big)\ \d{x_{s+1}}\cdots\d{x_d},
\end{align*}
and $\mathcal{I}^{[d+1:d]}_\pi \equiv \pi(x_0)$. The integration above is carried out over the region $\{(x_0-h_1)_k\leq x_k\leq (x_0+h_2)_k: s+1\leq k\leq d\} \subset [0,1]^{d-s}$ with the integrand given by the Lebesgue density $\pi$ of the design distribution $P$. In the $\bm{\beta}$-fixed lattice design case, $\mathcal{I}^{[s+1:d]}_\pi (h_1,h_2)=\prod_{k=s+1}^d(h_1+h_2)_k$.

Let $\kappa_\ast,n_\ast$ be defined by
\vspace{0.5ex}
\setlength{\tabcolsep}{8pt} 
\renewcommand{\arraystretch}{1.2} 
\begin{center}
	\begin{tabular}{|c||c|c|}
		\hline 
		& $\bm{\beta}$-fixed lattice design & random design\\
		\hline\hline
		$\kappa_\ast$ & $\arg\max\limits_{1\leq \ell \leq d} \frac{\sum_{k=\ell}^d \beta_k}{2+\sum_{k=\ell}^s \alpha_k^{-1}}$ & $1$\\
		\hline
		$n_\ast$ & $n^{\sum_{k=\kappa_\ast}^d \beta_k}$ & $n$ \\
		\hline
	\end{tabular}
\vspace{1ex}
\captionof{table}{Definitions of $\kappa_\ast,n_\ast$.}\label{table:k_ast_n_ast}
\end{center}
Let the limit process $\mathbb{C}(f_0,x_0)$ be defined by
\begin{align}\label{def:C_f}
&\mathbb{C}(f_0,x_0)\equiv \sup_{ \substack{h_1>0,\\ (h_1)_k\leq (x_0)_k, \\ s+1\leq k\leq d}}\inf_{\substack{h_2>0,\\ (h_2)_k\leq (1-x_0)_k, \\ s+1\leq k\leq d}} \\
&\qquad \qquad \bigg[\frac{\G(h_1,h_2)}{\prod_{k=\kappa_\ast}^s \big((h_1)_k+(h_2)_k\big) \mathcal{I}^{[s+1:d]}_\pi(h_1,h_2)}+\bar{f_0}(h_1,h_2;x_0)\bigg],\nonumber
\end{align}
where $\G$ is a centered Gaussian process defined on $\R^{d}_{\geq 0}\times \R^{d}_{\geq 0}$ with the following covariance structure: for any $(h_1,h_2), (h_1',h_2')$,
\begin{align*}
&\mathrm{Cov}\big(\G(h_1,h_2),\G(h_1',h_2')\big)\\
&= \prod_{k=\kappa_\ast}^s \big( (h_1)_k\wedge (h_1')_k+ (h_2)_k\wedge (h_2')_k\big)\cdot \mathcal{I}^{[s+1:d]}_\pi\big(h_1\wedge h_1', h_2\wedge h_2'\big),
\end{align*}
and 
\begin{align*}
\bar{f_0}(h_1,h_2;x_0) 
\equiv \sum_{ \substack{\bm{j} \in J_\ast,\\ j_k =0,1\leq k\leq \kappa_\ast-1}} \frac{ \partial^{\bm{j}} f_0(x_0)}{(\bm{j}+\bm{1})!} \prod_{k=\kappa_\ast}^s \frac{(h_2)_k^{j_k+1}-(-h_1)_k^{j_k+1}}{(h_2)_k+(h_1)_k}.
\end{align*}	

Furthermore, let $\mathbb{D}_{\bm{\alpha}}$ be defined by
\begin{align}\label{def:D_alpha}
\mathbb{D}_{\bm{\alpha}} &\equiv \sup_{ \substack{h_1>0,\\ (h_1)_k\leq (x_0)_k, \\s+1\leq k\leq d}}\inf_{\substack{h_2>0,\\ (h_2)_k\leq (1-x_0)_k, \\s+1\leq k\leq d}} \bigg[\frac{ \G(h_1,h_2)}{\prod_{k=\kappa_\ast}^s \big((h_1)_k+(h_2)_k\big) \mathcal{I}^{[s+1:d]}_\pi(h_1,h_2)} \\
&\qquad\qquad +\sum_{k=\kappa_\ast}^s \frac{ (h_2)_k^{\alpha_k+1}-(h_1)_k^{\alpha_k+1}}{(h_2)_k+(h_1)_k}\bigg].\nonumber
\end{align}

With these definitions, we are now in position to state the main result of this paper. 

\begin{theorem}\label{thm:limit_distribution_pointwise}
	Suppose Assumptions \ref{assump:smoothness}-\ref{assump:design} hold, and the errors $\{\xi_i\}$ are i.i.d. mean-zero with finite variance $\E \xi_1^2=\sigma^2<\infty$ (and are independent of $\{X_i\}$ in the random design case). With $\kappa_\ast,n_\ast$ defined in Table \ref{table:k_ast_n_ast}, we have the following local rate of convergence:
	\begin{align*}
	(n_\ast/\sigma^2)^{\frac{1}{2+\sum_{k=\kappa_\ast}^s \alpha_k^{-1}}}\big(\hat{f}_n(x_0)-f_0(x_0)\big) = \mathcal{O}_{\mathbf{P}}(1).
	\end{align*}
	If $\kappa_\ast$ is uniquely defined, with $\mathbb{C}(f_0,x_0)$ defined in (\ref{def:C_f}), the following limit theory holds:
		\begin{align*}
	(n_\ast/\sigma^2)^{\frac{1}{2+\sum_{k=\kappa_\ast}^s \alpha_k^{-1}}}\big(\hat{f}_n(x_0)-f_0(x_0)\big) \rightsquigarrow \mathbb{C}(f_0,x_0).
	\end{align*}
	Furthermore, if either $\{\alpha_k\}$ is a set of relative primes or all mixed derivatives of $f_0$ vanish at $x_0$ in $J_\ast$, then 
	\begin{align*}
	\mathbb{C}(f_0,x_0) =_d K(f_0,x_0)\cdot \mathbb{D}_{\bm{\alpha}},
	\end{align*} 
	where $K(f_0,x_0) = \big\{\prod_{k=\kappa_\ast}^s\big({\partial_k^{\alpha_k} f_0(x_0)}/{(\alpha_k+1)!}\big)^{1/\alpha_k}\big\}^{\frac{1}{2+\sum_{k=\kappa_\ast}^s \alpha_k^{-1}}}$ and $\mathbb{D}_{\bm{\alpha}}$ is defined in (\ref{def:D_alpha}).
\end{theorem}
\begin{proof}
See Section \ref{section:proof_limit_dist}.
\end{proof}

\begin{remark}\label{rmk:dependence_C_alpha}
A few technical remarks:
\begin{enumerate}
	  \item In the $\bm{\beta}$-fixed lattice design case where $\pi(x)=1$ is used to calculate $\mathcal{I}^{[s+1:d]}_\pi (h_1,h_2)$, 
	  the covariance structure of $\G$ is simpler:
	  	\begin{align}\label{eqn:cov_G_simple}
	  	&\mathrm{Cov}\big(\G(h_1,h_2),\G(h_1',h_2')\big)= \prod_{k=\kappa_\ast}^d \big( (h_1)_k\wedge (h_1')_k+ (h_2)_k\wedge (h_2')_k\big),
	  	\end{align}
	  	and can be represented as follows: let $d_\ast\equiv d-\kappa_\ast+1$, and let $\{\mathbb{B}_i: i \in \{1,2\}^{d_\ast}\}$ be independent Brownian sheets on $\R_{\geq 0}^{d_\ast}$. For any $h_1,h_2 \in \R_{\geq 0}^d$, let $\bar{\G}(h_1,h_2)\equiv \sum_{ i \in \{1,2\}^{d_\ast} } \mathbb{B}_i\big((h_{i_1})_{\kappa_\ast}, \ldots, (h_{i_{d_\ast}})_d\big)$. Then $\G(\cdot,\cdot)=_d\bar{\G}(\cdot,\cdot)$. A similar representation holds in the random design case for $s=d$.
	\item $\mathbb{C}(f_0,x_0)$ has at least a sub-gaussian tail and hence admits moments of all orders (cf. Lemma \ref{lem:finite_expectation_W}).
	\item When either $\{\alpha_k\}$ is a set of relative primes or all mixed derivatives of $f_0$ vanish at $x_0$ in $J_\ast$, the following self-similarity property of the process $\G(\cdot,\cdot)$ is essential for the representation $\mathbb{C}(f_0,x_0) =_d K(f_0,x_0)\cdot \mathbb{D}_{\bm{\alpha}}$: for $\gamma \in \R_{\geq 0}^d$ with $\gamma_1=\ldots=\gamma_{\kappa_\ast-1}=0, \gamma_{\kappa_\ast}\ldots,\gamma_s,\gamma_{s+1}=\ldots=\gamma_d=1$, $
	\G(\gamma \cdot,\gamma \cdot) =_d \big(\prod_{k=\kappa_\ast}^s \gamma_k\big)^{1/2}\cdot \G(\cdot,\cdot)$. 
    \item $\mathbb{D}_{\bm{\alpha}}$ (and $\mathbb{C}(f_0,x_0)$) can be represented by sup-inf over the summation of a stochastic term plus a non-random drift term, similar to that of the Chernoff distribution; see (\ref{eqn:chernoff}) below for an explicit derivation of $\mathbb{D}_1$ being the Chernoff distribution.
    \item Although implicit in notation, $\mathbb{D}_{\bm{\alpha}}$ can depend on $x_0$ through $\mathcal{I}_\pi^{[s+1:d]}$. However, such dependence disappears under (i) the fixed lattice design and (ii) the random design with uniform distribution. For general distributions $P$ in the random design case, if $s=d$ and $\mathbb{C}(f_0,x_0)=_d K(f_0,x_0)\cdot \mathbb{D}_{\bm{\alpha}}$, the dependence of $x_0$ within $\mathbb{D}_{\bm{\alpha}}$ can be assimilated into the constant. In fact, by taking $\sigma$ to be $\sigma/\sqrt{\pi(x_0)}$ in (\ref{ineq:limit_dist_rescale}), we have:
    	\begin{align*}
    	&\big(\pi(x_0)n/\sigma^2\big)^{\frac{1}{2+\sum_{k=1}^d \alpha_k^{-1}}}\big(\hat{f}_n(x_0)-f_0(x_0)\big)\rightsquigarrow K(f_0,x_0)\\
    	&\quad\times \sup_{h_1>0}\inf_{h_2>0}\bigg[\frac{\G(h_1,h_2)}{\prod_{k=1}^d \big((h_1)_k+(h_2)_k\big)}  +\prod_{k=1}^d \frac{(h_2)_k^{\alpha_k+1}-(h_1)_k^{\alpha_k+1}}{(h_2)_k+(h_1)_k}\bigg],
    	\end{align*}
     where $\G$ is the Gaussian process with the covariance structure 
	(\ref{eqn:cov_G_simple}).
\end{enumerate}
 
\end{remark}

Theorem \ref{thm:limit_distribution_pointwise} shows that the max-min block estimator (\ref{def:max_min_estimator}) adapts to the local smoothness levels $\{\alpha_k\}$ and the intrinsic dimension $s$ of the isotonic regression function $f_0$, in both the fixed lattice and random design settings. One particularly interesting consequence of the above theorem is that \emph{the adaptive local rates for the fixed lattice and random design cases are in general not the same}. Indeed,
\begin{align*}
\omega_n^{\mathrm{fixed}}
&\equiv n^{-\frac{\sum_{k=\kappa_\ast}^d \beta_k}{2+\sum_{k=\kappa_\ast}^s \alpha_k^{-1} }}= n^{-\max_{1\leq \ell \leq d}\frac{\sum_{k=\ell}^d \beta_k}{2+\sum_{k=\ell}^s \alpha_k^{-1} }},\\
\omega_n^{\mathrm{random}}&\equiv n^{-\frac{1}{2+\sum_{k=1}^s \alpha_k^{-1}}}= n^{-\frac{\sum_{k=1}^d \beta_k}{2+\sum_{k=1}^s \alpha_k^{-1} }},\nonumber
\end{align*}
so that $\omega_n^{\mathrm{fixed}}\leq \omega_n^{\mathrm{random}}$, i.e., the local rate in the fixed lattice design case is \emph{no slower} than that in the random design case.

The following proposition gives an equivalent definition of $\kappa_\ast$ in the fixed lattice design case in Theorem \ref{thm:limit_distribution_pointwise}.

\begin{proposition}\label{prop:comparision_local_rates}
	The following are equivalent under (\ref{def:ordering}):
	\begin{enumerate}
		\item The maximizer of $\ell \mapsto \frac{\sum_{k=\ell}^d \beta_k}{2+\sum_{k=\ell}^s \alpha_k^{-1} }$ is unique and $\kappa_\ast = \arg\max\limits_{1\leq \ell \leq d} \frac{\sum_{k=\ell}^d \beta_k}{2+\sum_{k=\ell}^s \alpha_k^{-1} }$.
		\item For any $1\leq \ell \leq d$, $\frac{\alpha_\ell^{-1}}{2+\sum_{k=\ell}^s \alpha_k^{-1}}\neq \frac{\beta_\ell}{\sum_{k=\ell}^{d} \beta_k }$, and $\kappa_\ast=\min\big\{1\leq \ell \leq d:  \frac{\alpha_\ell^{-1}}{2+\sum_{k=\ell}^s \alpha_k^{-1}}< \frac{\beta_\ell}{\sum_{k=\ell}^{d} \beta_k }   \big\}=\min\{ 1\leq \ell \leq d: (\omega_n^{(\ell)})^{1/\alpha_\ell} n^{\beta_\ell}> 1\}$. Here  $\omega_n^{(\ell)}\equiv n^{-\frac{\sum_{k=\ell}^d \beta_k}{2+\sum_{k=\ell}^s \alpha_k^{-1} }}$ is the unique solution of the fixed-point equation
		\begin{align}\label{eqn:bias_var_eqn}
		\omega =  \frac{1}{\sqrt{\prod_{k=\ell}^{d} \big(\omega^{1/\alpha_k} n^{\beta_k}\big) }}.
		\end{align}		
	\end{enumerate}
\end{proposition}
\begin{proof}
	By algebra, for any relationship $\sim$ in the set $ \{<,\leq,>,\geq \}$, we have (i) $\frac{\sum_{k=\ell}^d \beta_k}{2+\sum_{k=\ell}^s \alpha_k^{-1} } \sim \frac{\sum_{k=\ell+1}^d \beta_k}{2+\sum_{k=\ell+1}^s \alpha_k^{-1} }$ if and only if (ii) $ \frac{\beta_\ell}{\sum_{k=\ell}^{d} \beta_k }\sim \frac{\alpha_\ell^{-1}}{2+\sum_{k=\ell}^s \alpha_k^{-1}}$ if and only if (iii) $2\sim \sum_{k=\ell}^d \beta_k \big(\frac{1}{\alpha_\ell \beta_\ell}-\frac{1}{\alpha_k\beta_k}\big)\equiv \psi(\ell)$. Under the ordering (\ref{def:ordering}), $\ell \mapsto \psi(\ell)$ is non-increasing, so the statement (1) holds if and only if 
	$\sim$ is taken as $<$ for all $1\leq \ell \leq \kappa_\ast-1$ and as $>$ for $\kappa_\ast\leq \ell \leq d$ in (i), if and only if the same $\sim$ are taken in (ii), if and only if the statement (2) holds.
\end{proof}

Proposition \ref{prop:comparision_local_rates} (2) shows that $\kappa_\ast$ can be determined by a sequence of bias-variance equations in (\ref{eqn:bias_var_eqn}). This gives an interesting interpretation of the quantities $\kappa_\ast, n_\ast$ in the $\bm{\beta}$-fixed lattice design case:
\begin{itemize}
	\item $\kappa_\ast(\leq (s+1)\wedge d)$ can be viewed as a `critical dimension' in the sense that samples in dimensions $\{1,\ldots \kappa_\ast-1\}$ do not contribute in the asymptotics of $\hat{f}_n$. In other words, the limit distribution of $\hat{f}_n$ is fully driven by samples in dimensions $\{\kappa_\ast,\ldots,d\}$. The uniqueness of the maximizer of $\ell \mapsto \frac{\sum_{k=\ell}^d \beta_k}{2+\sum_{k=\ell}^s \alpha_k^{-1} }$ gives a well-defined $\kappa_\ast$, and therefore the limiting distribution $\mathbb{C}(f_0,x_0)$.
	\item $n_\ast = n^{\sum_{k=\kappa_\ast}^{d} \beta_k}$ can be viewed as the `effective sample size' over the effective dimensions $\{\kappa_\ast,\ldots,d\}$ for the asymptotics of $\hat{f}_n$. 
\end{itemize}

In contrast, in the random design case the `critical dimension' $\kappa_\ast$ is always $\kappa_\ast=1$, as long as the Lebesgue density of $P$ is suitably regular at $x_0$. In this setting all dimensions $\{1,\ldots,d\}$ are effective, and the `effective sample size' is simply $n_\ast =n$.

The local rate of convergence in Theorem \ref{thm:limit_distribution_pointwise} can now be interpreted very naturally:
the exponent for the `effective sample size' $n_\ast$, namely, $\frac{1}{2+\sum_{k=\kappa_\ast}^s \alpha_k^{-1}}$ becomes the local smoothness along `effective dimensions' $\{\kappa_\ast,\ldots,d\}$ (note that $\alpha_k^{-1}=0$ for $s+1\leq k\leq d$).

\begin{remark}
In the special case that $f_0$ is globally flat, i.e. $f_0\equiv c$ for some $c \in \R$, we have $\alpha_1=\ldots=\alpha_d=\infty$, and therefore the local rate of convergence for the max-min block estimator (\ref{def:max_min_estimator}) is parametric $\mathcal{O}_{\mathbf{P}}(n^{-1/2})$. When $f_0$ is locally flat, it is shown by \cite{carolan1999asymptotic} (see also \cite{groeneboom1983concave}) that in the closely related univariate monotone density estimation problem, the Grenander estimator converges at a parametric rate, with a limiting distribution involving the maximal interval contained in the flat region. In the multivariate case, the shape for locally flat regions can be quite complicated. For example, for $d=2$ and any upper set $U\subset [1/2,1]^2$, consider $f_0\equiv \bm{1}_U$. Then the local rate of convergence for the max-min block estimator (\ref{def:max_min_estimator}) is still parametric at, say, $(1/4,1/4)$, but the limiting distribution would depend crucially on the exact shape of $U$. It is an interesting open problem to characterize all possible locally flat regions and derive the corresponding limiting distributions for the max-min block estimator (\ref{def:max_min_estimator}).
\end{remark}

\subsection{Comparison of local rates}

In this section, we make comparisons of the local rates in different fixed lattice and random designs. As will be seen below, the discrepancy of the local rates appears when either the local smoothness levels of the isotonic regression function, or the sizes of the lattice differ substantially along different dimensions.

\subsubsection{Difference in local rates due to imbalanced local smoothness levels}

Consider the case where the local smoothness levels are imbalanced with $\alpha_1=\ldots=\alpha_s = \alpha, \alpha_{s+1}=\ldots=\alpha_d =\infty$ for some $\alpha\geq 1, 1\leq s< d$, while the sizes of the lattice are balanced with $\beta_1=\ldots=\beta_d = 1/d$. By Theorem \ref{thm:limit_distribution_pointwise}, we have the following corollary:
\begin{corollary}\label{cor:fixed_random}
	Suppose that the assumptions in Theorem \ref{thm:limit_distribution_pointwise} hold with $\alpha_1=\ldots=\alpha_s = \alpha\geq 1, \alpha_{s+1}=\ldots=\alpha_d =\infty$ for some $1\leq s< d$, and $\beta_1=\ldots=\beta_d=1/d$. Then in the fixed lattice design case, 
	\begin{align*}
	&(n/\sigma^2)^{\frac{1}{2+s/\alpha}}\big(\hat{f}_n(x_0)-f_0(x_0)\big) \rightsquigarrow  \mathbb{C}(f_0,x_0),\quad \alpha>(d-s)/2;\\
	&(n^{1-s/d}/\sigma^2)^{1/2}\big(\hat{f}_n(x_0)-f_0(x_0)\big)\rightsquigarrow  \mathbb{C}(f_0,x_0),\quad \alpha<(d-s)/2.
	\end{align*}
	In the random design case,
	\begin{align*}
	(n/\sigma^2)^{\frac{1}{2+s/\alpha}}\big(\hat{f}_n(x_0)-f_0(x_0)\big) \rightsquigarrow  \mathbb{C}(f_0,x_0).
	\end{align*}
\end{corollary}
The local rates can be written more compactly:
\begin{itemize}
	\item (\emph{Fixed design}): $
	\hat{f}_n(x_0)-f_0(x_0)=\mathcal{O}_{\mathbf{P}}\big( n^{-\max\{\frac{1}{2+s/\alpha}, \frac{1}{2}\cdot \frac{d-s}{d}\}}\big)$.
	\item (\emph{Random design}): $
	\hat{f}_n(x_0)-f_0(x_0) = \mathcal{O}_{\mathbf{P}}(n^{-\frac{1}{2+s/\alpha}})$.
\end{itemize}

Below we consider two scenarios according to the phase transition boundary $\alpha=(d-s)/2$ given above in the balanced fixed lattice design case.\newline

\noindent \textbf{(Scenario 1: $\alpha>(d-s)/2$).} In this case, $\kappa_\ast = 1$ in the fixed lattice design case, so $\omega_n^{\mathrm{fixed}} = \omega_n^{\mathrm{random}} = n^{-\frac{1}{2+s/\alpha}}$. This includes the important case of $s=d$.  In the special case where $d=1$ and $\alpha_1=1$, Corollary \ref{cor:fixed_random} reduces to the limit distribution theory for univariate isotonic regression: Suppose for simplicity we consider the fixed balanced lattice design, or the uniform random design. Then $\mathbb{C}(f_0,x_0) =_d K(f_0,x_0)\cdot \mathbb{D}_1 = (f_0'(x_0)/2)^{1/3} \cdot\mathbb{D}_1$, where $\mathbb{D}_1$ is the well-known (rescaled) Chernoff distribution. To see this, with $\mathbb{B}$ denoting the standard two-sided Brownian motion starting at $0$, we have (cf. Section 3.3 of \cite{groeneboom2014nonparametric})
\begin{align}\label{eqn:chernoff}
\mathbb{D}_{1} &= \sup_{h_1>0}\inf_{h_2>0}\bigg[\frac{\G(h_1,h_2)}{h_1+h_2}+\big(h_2-h_1\big)\bigg]\\
& =_d \sup_{-h_1<0}\inf_{h_2>0}\bigg[\frac{\big(\mathbb{B}(h_2)+h_2^2\big)-\big(\mathbb{B}(-h_1)+(-h_1)^2\big)}{h_2-(-h_1)}\bigg] \nonumber\\
& = \textrm{slope at zero of the greatest convex minorant of } t\mapsto \mathbb{B}(t)+t^2 \nonumber\\
& =_d \textrm{slope at zero of the least concave majorant of } t\mapsto \mathbb{B}(t)-t^2.\nonumber
\end{align}

It is also interesting to observe that for the most natural case $\alpha_1=\ldots=\alpha_d=1$, the local rate is $\mathcal{O}_{\mathbf{P}}(n^{-1/(2+d)})$. This local rate is, somewhat surprisingly, \emph{faster} than the global minimax rate $\mathcal{O}(n^{-1/2d})$ in $L_2$ metric for $d\geq 3$, cf. \cite{han2017isotonic}. The reason for this is that the global minimax rate in $L_2$ metric is dominated by the anti-chain structure of the multiple isotonic regression functions (cf. \cite{han2017isotonic}), while the smoothness constraint rules out such a structure locally at a fixed point. To put the problem in other words, the global minimax rate in $L_2$ metric is too conservative in capturing the smoothness structure of the isotonic functions as soon as $d\geq 3$. \newline

\noindent \textbf{(Scenario 2: $\alpha<(d-s)/2$).} In this case $\kappa_\ast = s+1>1$, so the local rate of convergence in the fixed lattice design case is much faster than that in the random design case: $\omega_n^{\mathrm{fixed}} \ll \omega_n^{\mathrm{random}} $. 

Let us consider one concrete situation to better understand this phenomenon: $\alpha=1, s=1$ and $d>3$. The local rate is then $\mathcal{O}_{\mathbf{P}}(n^{-\frac{1}{2}\cdot \frac{d-1}{d}})$ in the fixed lattice design case, and is $\mathcal{O}_{\mathbf{P}}(n^{-1/3})$ in the random design case. Suppose for simplicity the regression function $f_0(x_0)=f_0((x_0)_1,\ldots,(x_0)_d)= g_0((x_0)_1)$ for some one-dimensional nondecreasing function $g_0$. Consider fixed lattice and random design cases separately:

\begin{figure}
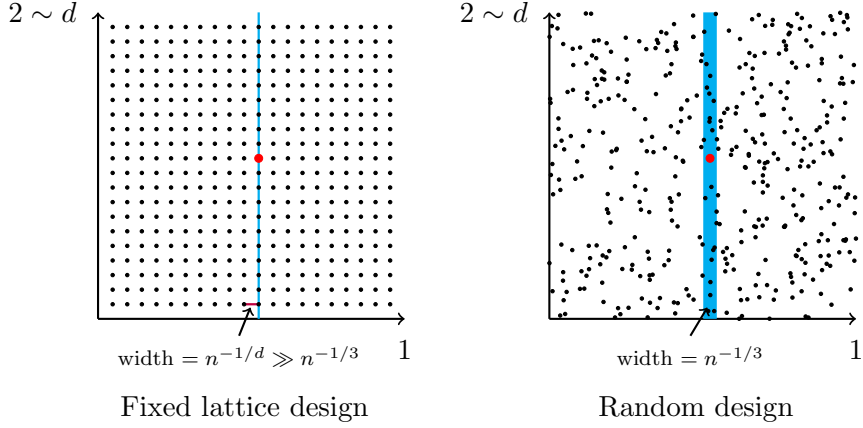

	    	\hspace*{-2em}

	\caption{Illustration of the oracle estimator in the balanced fixed lattice and random design cases. Horizontal direction = dimension $1$. Vertical direction = dimensions $2$ to $d$. Red point = $x_0$. Blue strip = samples over which the oracle estimator takes the average.}
	\label{fig:local_rate}
\end{figure}

\begin{itemize}
	\item In the fixed lattice design case, the oracle estimator first takes sample mean in dimensions $2$ to $d$, and then performs isotonic regression in dimension $1$ with reduced variance $\sigma_1^2\equiv \sigma^2/n^{(d-1)/d}$ and sample size $n_1=n^{1/d}$. However, as long as $d> 3$, there are no longer large samples within the oracle bandwidth $(\sigma_1^2/n_1)^{1/3} = (\sigma^2/n)^{1/3}\ll n^{-1/d}$ in dimension $1$ due to the smoothness. This means that the oracle estimator is simply the sample mean over dimensions $2$ to $d$ with a convergence rate $n^{-\frac{1}{2}\cdot \frac{d-1}{d}}$ when $d>3$. See the left panel of Figure \ref{fig:local_rate}.
	\item In the random design case, since the first coordinates of the design points are distinct with probability one, the oracle estimator is the one dimensional estimator with a bandwidth on the order of $n^{-1/3}$. This gives the usual convergence rate $n^{-1/3}$. See the right panel of Figure \ref{fig:local_rate}.
\end{itemize}

Corollary \ref{cor:fixed_random} with $\alpha=1,s=1,d>3$ can then be understood as saying that the max-min block estimator (\ref{def:max_min_estimator}) mimics this oracle behavior in terms of the local rate of convergence, in both the fixed lattice and random design cases. In more general settings of Corollary \ref{cor:fixed_random}, as soon as the local smoothness levels $\alpha<(d-s)/2$, the first $s$ dimensions are screened out in the fixed lattice design case, so the asymptotics only take place over pure noises in dimensions $\{s+1,\ldots,d\}$.

\subsubsection{Difference in local rates due to imbalanced lattice sizes}
Consider the case where the local smoothness levels are balanced with $\alpha_1=\ldots=\alpha_d = \alpha\geq 1$, while the sizes of the lattice are imbalanced with $\beta_1\leq \ldots\leq \beta_d$. Using $\sum_{k=1}^d \beta_k =1$, Theorem \ref{thm:limit_distribution_pointwise} and the equivalent definition of $\kappa_\ast$ in Proposition \ref{prop:comparision_local_rates}, we have the following.

\begin{corollary}
Suppose that the assumptions in Theorem \ref{thm:limit_distribution_pointwise} hold with $\alpha_1=\ldots=\alpha_d = \alpha\geq 1$, and $\beta_1\leq \ldots\leq \beta_d$. 

\noindent In the fixed lattice design case, suppose $\beta_\ell \neq \frac{1-\sum_{k=1}^{\ell-1} \beta_k}{2\alpha+d-\ell+1}$ for all $1\leq \ell\leq d$. Let $
\kappa_\ast \equiv \min\big\{1\leq \ell \leq d: \beta_\ell > \frac{1-\sum_{k=1}^{\ell-1} \beta_k}{2\alpha+d-\ell+1}\big\}$ and $d_\ast \equiv d-\kappa_\ast+1$. Then
\begin{align*}
&(n^{\sum_{k=\kappa_\ast}^d \beta_k}/\sigma^2)^{1/(2+d_\ast/\alpha)}\big(\hat{f}_n(x_0)-f_0(x_0)\big) \rightsquigarrow  \mathbb{C}(f_0,x_0).
\end{align*}
In the random design case,
\begin{align*}
(n/\sigma^2)^{\frac{1}{2+d/\alpha}}\big(\hat{f}_n(x_0)-f_0(x_0)\big) \rightsquigarrow  \mathbb{C}(f_0,x_0).
\end{align*}
\end{corollary}

\begin{figure}
	\hspace*{-2em}
	\begin{tikzpicture}[xscale=1.5,yscale=1.5]
	\draw[->][draw=black, thick] (0,0) -- (e,0);
	\draw[->][draw=black, thick] (0,0)--(0,e);
	\draw[draw=blue, thick] (0,e)--(e/2.3,2.127351);
	\draw[draw=red, thick] (e/2.3,2.127351)--(e,e/2);
	\draw[dashed] (0,e/2)--(e,e/2);
	\node [left] at (-0.1,e/2) {$\frac{1}{2}$};
	\draw[dashed] (e/2.3,0)--(e/2.3,2.127351);
	\node [below] at (e/2.3,-0.1) {$\frac{1}{2\alpha+2}$};
	\node [above, purple] at (2,2.3) {\scriptsize $\omega_n^{(1)}=$};
	\node [below, purple] at (2,2.35) {\scriptsize $n^{-\frac{1}{2+2/\alpha}}$};
	\node [above, purple] at (0.6,1.9) {\scriptsize $\omega_n^{(2)}=$};
	\node [below, purple] at (0.6,1.95) {\scriptsize $n^{-\frac{1-\beta_1}{2+1/\alpha}}$};

	\node [below] at (e,-0.1) {$1/2$};
	\node [left] at (-0.1,e) {$1$};
	\node [right] at (e,0) {$\beta_1$};
	\node [above] at (0,e) {$\beta_2$};

	\fill[fill=gray!10] plot[smooth, samples=100, domain=4:e+4] (e/2.7+4,e/2.7) -- (e/1.5+4,e/1.5)--(e/2.7+4,0.8148*e);
	\fill[fill=gray!25] plot[smooth, samples=100, domain=4:e+4] (4,e/2.3) -- (e/2.7+4,e/2.7) -- (e/2.7+4,0.8148*e)--(4,e);
	\fill[fill=gray!40] plot[smooth, samples=100, domain=4:e+4] (4,0) -- (e/2.7+4,e/2.7)--(4,e/2.3);
	\draw[->][draw=black, thick] (4,0) -- (e+4,0);
	\draw[->][draw=black, thick] (4,0)--(4,e);
	\draw[dashed] (e/1.5+4,0)--(e/1.5+4,e/1.5);
	\draw[dashed] (e/2.7+4,0)--(e/2.7+4,e/2.7);
	\node [below] at (e+4,-0.1) {$1/2$};
	\node [below] at (e/1.5+4,-0.1) {$1/3$};
	\node [right] at (e+4,0) {$\beta_1$};
	\node [left] at (3.9,e) {$1/2$};
	\node [above] at (4,e) {$\beta_2$};
	\node [below] at (e/2.7+4,-0.1) {$\frac{1}{2\alpha+3}$};
	\node [left] at (3.9,e/2.3) {$\frac{1}{2\alpha+2}$};
	\node [purple] at (5.76,1.65) {\scriptsize $\omega_n^{(1)}= n^{-\frac{1}{2+3/\alpha}}$};
	\node [above, purple] at (4.50,1.85) {\scriptsize $\omega_n^{(2)}=$};
	\node [below, purple] at (4.50,1.9) {\scriptsize $n^{-\frac{1-\beta_1}{2+2/\alpha}}$};
	\node [above, purple] at (4.95,0.5) {\scriptsize $\omega_n^{(3)}=n^{-\frac{1-\beta_1-\beta_2}{2+1/\alpha}}$};
	\draw [fill,green] (4+e/1.5, e/1.5 ) circle [radius=0.05];
	\draw [fill,green] (e, e/2 ) circle [radius=0.05];

	\node [above] at (1.3,-1) {$d=2$};    	
	\node [above] at (5.3,-1) {$d=3$};
	\end{tikzpicture}
	\caption{Phase transitions of the local rate of convergence for the $\bm{\beta}$-fixed lattice design in $d=2$ (left panel, in the line segment $\{\beta_1\leq 1/2, \beta_2=1-\beta_1\}$) and $d=3$ (right panel, in the triangle $\{\beta_1\leq 1/3, \beta_1\leq \beta_2\leq (1-\beta_1)/2\}$). $\omega_n^{(\ell)}$ = local rate with $\kappa_\ast = \ell$. Local rates in fixed lattice and random design cases match for $\kappa_\ast = 1$. Green points = values of $(\beta_1,\beta_2)$ for the balanced fixed lattice design.}
	\label{fig:phase_transition}
\end{figure}
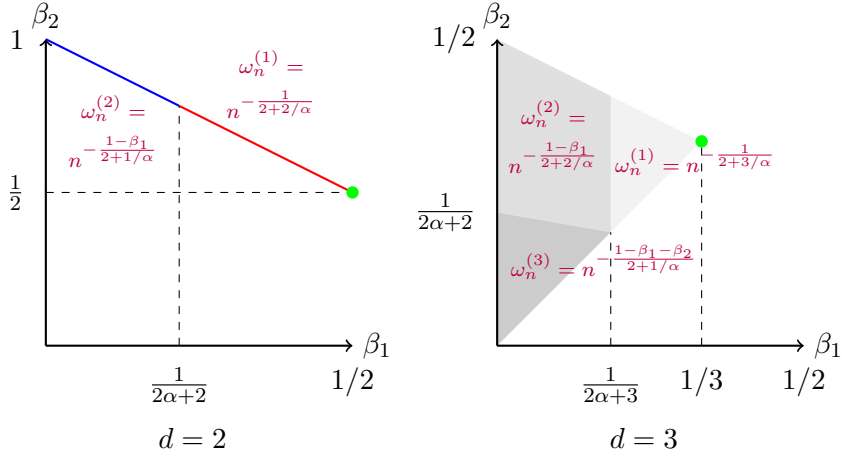

Alternatively, we may write the local rates more compactly:
\begin{itemize}
	\item (\emph{Fixed design}): $
	\hat{f}_n(x_0)-f_0(x_0)=\mathcal{O}_{\mathbf{P}}\big( n^{-\max_{1\leq \ell\leq d} \frac{\sum_{k=\ell}^d \beta_k }{2+(d-\ell+1)/\alpha} }\big)$.
	\item (\emph{Random design}): $
	\hat{f}_n(x_0)-f_0(x_0) = \mathcal{O}_{\mathbf{P}}(n^{-\frac{1}{2+d/\alpha}})$.
\end{itemize}

The basic pattern for the phase transition phenomenon here is similar to the discussion in the previous section. The difference is that now it is the imbalance of the sizes of the lattice in different dimensions, rather than the local smoothness levels, that screens out dimensions with too sparsely spaced design points. See Figure \ref{fig:phase_transition} for the concrete phase transition in the line segment $\{\beta_1+\beta_2=1, \beta_1\leq \beta_2\}$ for $d=2$ and in the triangle $\{\beta_1+\beta_2+\beta_3=1, \beta_1\leq \beta_2\leq \beta_3\}$ for $d=3$.

\begin{remark}
The improvement of the local rate in the fixed lattice design case over the random design case is strongly tied to the lattice structure. It is possible that the local rate in non-lattice fixed design case is slower than that in the random design case. For instance, in the extreme case in $d=2$, suppose that the fixed design points are all located on a straight line $\{(x_1,x_2)\subset [0,1]^2:x_2-1/2 = \theta (x_1-1/2), 0\leq x_1\leq 1\}$ with $\theta<0$. Then we do not have consistent estimation in general as $g(\cdot)\equiv f(\cdot,\theta \cdot)$ may not be monotonically non-decreasing. The situations for general fixed designs will be more complicated. Although we expect that the local rates in `most' fixed design cases will match that in the random design case, it remains an open question to give a complete characterization.
\end{remark}

\subsection{An illustrative simulation result}

We present here an illustrative simulation result to assess the accuracy of the distributional approximation in Theorem \ref{thm:limit_distribution_pointwise} for the case $d=2$ and $\alpha_1=\alpha_2=1$. 

We consider two isotonic regression functions: $f_1(x)=e^{x_1+x_2}$ and $f_2(x)=e(x_1+x_2)$, and $x_0=(1/2,1/2)$. Clearly, $f_2$ is linearization of $f_1$ at $x_0$. In particular, the product of the partial derivatives of the two isotonic functions are the same at $x_0$, i.e., $K(f_1,x_0)=K(f_2,x_0)$. Theorem \ref{thm:limit_distribution_pointwise} indicates that the limiting distributions for the weighted statistics $n^{1/4}(\hat{f}_n(x_0)-f_i(x_0))$ will be the same for $i=1,2$.

\begin{figure}
	\centering
	\includegraphics[width=1\textwidth]{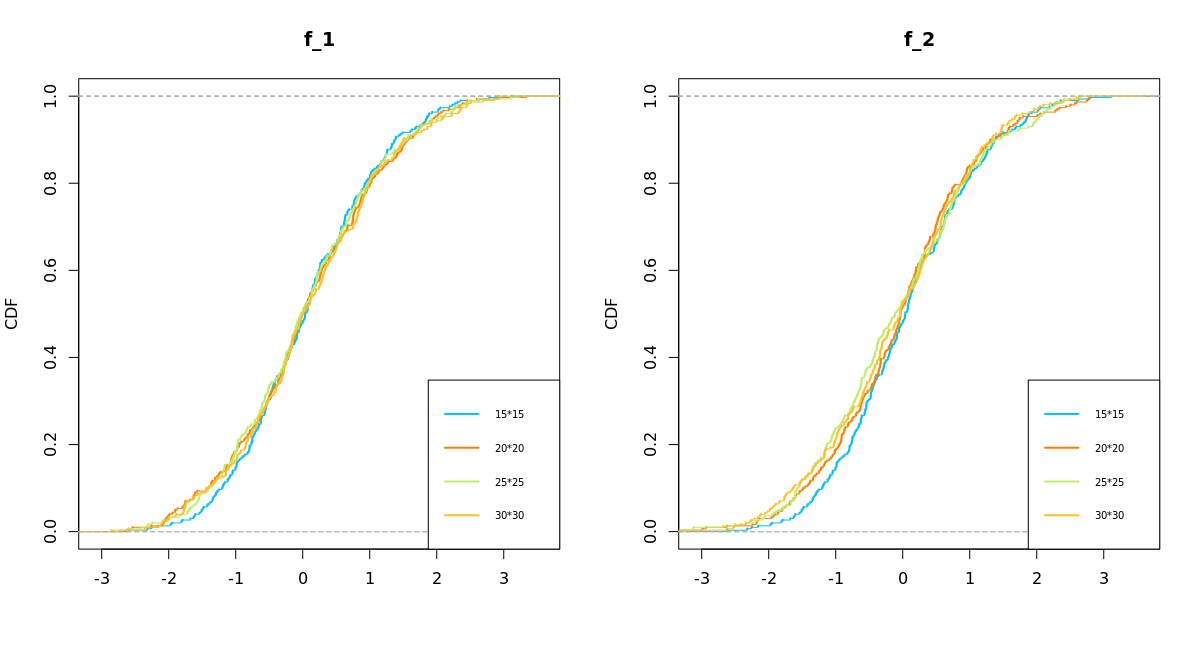}
	\caption{Empirical cumulative distribution functions for $n^{1/4}(\hat{f}_n(x_0)-f_i(x_0)) (i=1,2)$ based on $B=300$ repetitions. Here $d=2$, $f_1(x_1,x_2)=e^{x_1+x_2}$ in the left panel and $f_2(x_1,x_2)=e(x_1+x_2)$ in the right panel, and $x_0=(1/2,1/2)$. Lattice sizes: $15\times 15, 20\times 20, 25\times 25, 30\times 30$.  }
	\label{fig:sim1}
\end{figure}

\begin{figure}
	\centering
	\includegraphics[width=1\textwidth]{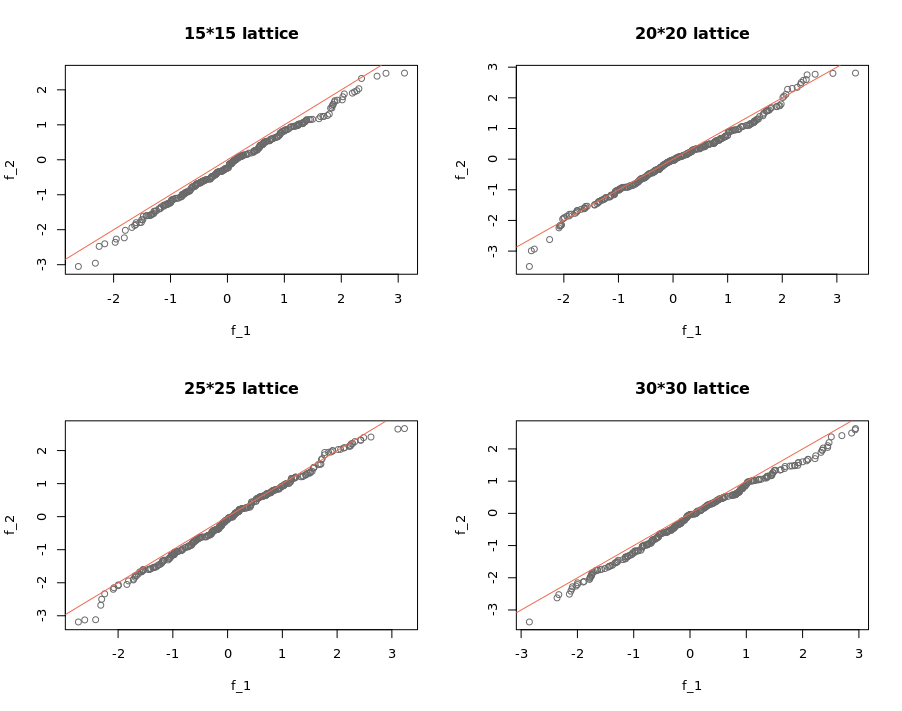}
	\caption{QQ-plots for the distributions of $n^{1/4}(\hat{f}_n(x_0)-f_1(x_0))$ (horizontal) versus $n^{1/4}(\hat{f}_n(x_0)-f_2(x_0))$ (vertical) based on the same simulation as in Figure \ref{fig:sim1}. Lattice sizes: $15\times 15, 20\times 20, 25\times 25, 30\times 30$. }
	\label{fig:sim2}
\end{figure}

We consider the following lattice sizes in the simulation: $15\times 15, 20\times 20, 25\times 25,30\times 30$. Figure \ref{fig:sim1} plots the empirical cumulative distribution functions for the weighted statistics $n^{1/4}(\hat{f}_n(x_0)-f_i(x_0)) (i=1,2)$ based on $B=300$ repetitions with i.i.d. normal errors $\mathcal{N}(0,1)$; it shows that even for the quite small lattice of size $15\times 15$, the max-min block estimator (\ref{def:max_min_estimator}) already achieves reasonable distributional approximation. Not surprisingly, the shapes of the empirical cumulative distribution functions for $f_1,f_2$ are rather similar. This can be further verified through the QQ-plots for different lattice sizes in Figure \ref{fig:sim2}.

\section{Local asymptotic minimax lower bound}\label{section:local_minimax}

We derive in Theorem \ref{thm:limit_distribution_pointwise} the precise limiting distribution of the max-min block estimator (\ref{def:max_min_estimator}) with a local rate of convergence and a limit distribution depending on the unknown smoothness of the regression function at the point of interest. It is natural to wonder if the local rate and the limiting distribution are optimal.

\begin{theorem}\label{thm:minimax_lower_bound}
		Suppose Assumptions \ref{assump:smoothness}-\ref{assump:design} hold, and the errors $\{\xi_i\}$ are i.i.d. $\mathcal{N}(0,\sigma^2)$. Then with $\kappa_\ast, n_\ast$ defined in Theorem \ref{thm:limit_distribution_pointwise}, we have
		\begin{align*}
		&\sup_{\tau>0}\liminf_{n \to \infty} \inf_{\tilde{f}_n} \sup_{f \in \mathcal{F}_d: \ell_2^2(f,f_0)\leq \tau \sigma^2/n} \E_f \bigg[ (n_\ast/\sigma^2)^{\frac{1}{2+\sum_{k=\kappa_\ast}^s \alpha_k^{-1}}} \bigabs{\tilde{f}_n(x_0)-f(x_0)}\bigg]\\
		&\qquad\qquad \geq  L_{d, \pnorm{\bm{\alpha}}{\infty},P,f_0}.
		\end{align*}
		Furthermore, if all mixed derivatives of $f_0$ vanish at $x_0$, then
			\begin{align*}
		&\sup_{\tau>0}\liminf_{n \to \infty} \inf_{\tilde{f}_n} \sup_{f \in \mathcal{F}_d: \ell_2^2(f,f_0)\leq \tau \sigma^2/n} \E_f \bigg[ (n_\ast/\sigma^2)^{\frac{1}{2+\sum_{k=\kappa_\ast}^s \alpha_k^{-1}}} \bigabs{\tilde{f}_n(x_0)-f(x_0)}\bigg]\\
		&\qquad\qquad \geq L_{d, \pnorm{\bm{\alpha}}{\infty},P} \cdot \bigg( \prod_{k=\kappa_\ast}^s \bigg(\frac{\partial_k^{\alpha_k} f_0(x_0)}{(\alpha_k+1)!}\bigg)^{1/\alpha_k}\bigg)^{\frac{1}{2+\sum_{k=\kappa_\ast}^s \alpha_k^{-1}}}.
		\end{align*}
	Here $\pnorm{\bm{\alpha}}{\infty}\equiv\max_{\kappa_\ast\leq k\leq s}\alpha_k$, $
	\ell_2^2(f,f_0)\equiv n^{-1}\sum_{i=1}^n (f(X_i)-f_0(X_i))^2$ for the fixed lattice design case and $
	\ell_2^2(f,f_0)\equiv P (f-f_0)^2$ for the random design case.
\end{theorem}
\begin{proof}
See Section \ref{section:proof_minimax}.
\end{proof}

Theorem \ref{thm:minimax_lower_bound} shows that the max-min block estimator (\ref{def:max_min_estimator}) enjoys a strong oracle property:  Both (i) the adaptive local rate of convergence and (ii) the dependence on the constants, whenever explicit, concerning the unknown regression function $f_0$ in the limit distribution are optimal in a local asymptotic minimax sense, up to a constant factor depending only on  $d,\pnorm{\bm{\alpha}}{\infty}, P$. Note that here the local minimax lower bound is computed over an $\ell_2$-ball with radius of order $\mathcal{O}(n^{-1/2})$.

\begin{remark}
The dependence of the constant $L_{d, \pnorm{\bm{\alpha}}{\infty},P}$ on $P$ in the second claim of Theorem \ref{thm:minimax_lower_bound} can be further improved if $s=d$ in the random design setting: a slight modification of the proof of Theorem \ref{thm:minimax_lower_bound} shows that $L_{d, \pnorm{\bm{\alpha}}{\infty},P} = (\pi(x_0))^{-1/(2+\sum_{k=1}^d \alpha_k^{-1})} L_{d, \pnorm{\bm{\alpha}}{\infty}}$. The limit distribution of max-min block estimator (\ref{def:max_min_estimator}) achieves the optimal dependence on $\pi(x_0)$ in this setting, cf. Remark \ref{rmk:dependence_C_alpha}.
\end{remark}

\begin{remark}
In a related block-decreasing density estimation problem, \cite{pavlides2012local} establishes a local minimax lower bound in the special case of $\alpha_1=\ldots=\alpha_d=1$. 
Their result and Theorem \ref{thm:minimax_lower_bound} concern different problems but 
are similar in spirit. Global minimax lower bounds in $L_1$ and global risk bounds for histogram-type estimators in the same model are studied in \cite{biau2003risk}. 
\end{remark}

\section{Outline of the proofs}\label{section:proof_outline}

\subsection{Outline for the proof of Theorem \ref{thm:limit_distribution_pointwise}}\label{section:proof_outline_thm1}

The proof of Theorem \ref{thm:limit_distribution_pointwise} is rather involved, so we highlight the main proof ideas here. Recall that $\kappa_\ast =1$ in the random design case. Let
\begin{align}\label{def:r_n}
\omega_n \equiv n_\ast^{-\frac{1}{2+\sum_{k=\kappa_\ast}^s \alpha_k^{-1}}},\quad r_n \equiv (\omega_n^{1/\alpha_1},\ldots,\omega_n^{1/\alpha_d})\bm{1}_{[\kappa_\ast:d]}
\end{align}
The components of $r_n$ indicate the localization rate along each dimension. Now we may re-parametrize the max-min block estimator (\ref{def:max_min_estimator}) on the scale of the rate  vector $r_n$. To this end, let $h_1^\ast,h_2^\ast \in \R_{\geq 0}^d$ be such that
\begin{align}\label{def:h_1_h_2}
\hat{f}_n(x_0) = \max_{h_1\geq 0}\min_{h_2\geq 0} \bar{Y}|_{[x_0-h_1r_n, x_0+h_2 r_n]}=\bar{Y}|_{[x_0-h_1^\ast r_n, x_0+h_2^\ast r_n]}.
\end{align}
We remind the reader that $\bar{Y}|_\cdot$ is the average response over subsets as defined in (\ref{def:average}). Such a re-parametrization relates the problem to its limit Gaussian version. Note that the first $\kappa_\ast-1$ coordinates of $r_n$ is $0$, but this is no problem: we will show that such $h_1^\ast,h_2^\ast$ exist with high probability.

For any $c>1$, define the localized max-min block estimator
\begin{align}\label{def:local_max_min}
\hat{f}_{n,c}(x_0) &\equiv \max_{c^{-\gamma_\ast}\bm{1}\leq h_1\leq c\bm{1} } \min_{c^{-\gamma_\ast}\bm{1}\leq h_2\leq c\bm{1} } \bar{Y}|_{[x_0-h_1 r_n,x_0+h_2r_n]},
\end{align}
where $\gamma_\ast>0$ is chosen large enough. The difference of (\ref{def:local_max_min}) compared with (\ref{def:h_1_h_2}) is that the range of max and min in the global estimator (\ref{def:h_1_h_2}) is restricted to a compact rectangle away from $0$ and $\infty$ in (\ref{def:local_max_min}), so the squared bias and variance in the partial sum process in (\ref{def:local_max_min}) are on the same order. We may therefore expect a non-degenerate limit theory for properly normalized (\ref{def:h_1_h_2}) by `interchanging' the limit and max-min operations in (\ref{def:local_max_min}) and showing that (\ref{def:h_1_h_2}) and (\ref{def:local_max_min}) are sufficiently `close' to each other.

Formally, we need the following localization-delocalization result, due to \cite{rao1969estimation}.
\begin{proposition}\label{prop:local_delocal}
	Suppose that three sequences of random variables $\{W_{n,c}\}$, $\{W_n\}$ and $\{W_c\}$ satisfy the following conditions:
	\begin{enumerate}
		\item $\lim_{c \to \infty} \limsup_{n \to \infty} \Prob(W_{n,c}\neq W_n)=0$.
		\item For every $c>0$, $W_{n,c}\rightsquigarrow W_c$ as $n \to \infty$.
		\item $\lim_{c \to \infty} \Prob(W_c\neq W)=0$.
	\end{enumerate}
	Then $W_n\rightsquigarrow W$ as $n \to \infty$.
\end{proposition}
We choose
\begin{align*}
W_{n,c}&\equiv \omega_n^{-1}\big(\hat{f}_{n,c}(x_0)-f_0(x_0)\big),\quad W_n\equiv \omega_n^{-1}\big(\hat{f}_n(x_0)-f_0(x_0)\big).
\end{align*}
Now we need to verify the conditions of Proposition \ref{prop:local_delocal} for the above defined $W_{n,c}$ and $W_n$, and find out the limit process $W_c$.

To illustrate the most important ideas in our proof, we focus on the simplest 2-dimensional balanced fixed lattice design case where $\alpha_1=\alpha_2=1$, $r_n = (n^{-1/4},n^{-1/4})$ and $\omega_n = n^{-1/4}$ in (\ref{def:r_n}), cf. Assumption \ref{assump:smoothness}.

The first step is to show that the block max-min estimator (\ref{def:h_1_h_2}) can be localized through its local version (\ref{def:local_max_min}):
\begin{align}\label{ineq:proof_outline_1}
\lim_{c \to \infty} \limsup_{n \to \infty} \Prob\big(\hat{f}_n(x_0)\neq \hat{f}_{n,c}(x_0)\big)=0.
\end{align}
It is easy to see that proving (\ref{ineq:proof_outline_1}) reduces to showing that ${h_1^\ast}$ and ${h_2^\ast}$ are both bounded away from $\infty$ and $0$ in probability,  which we refer to as the \emph{large deviation} and \emph{small deviation} problems respectively. In other words, neither the bias nor the variance of the partial sum process in (\ref{def:h_1_h_2}) within the block $[x_0-h_1^\ast r_n,x_0+h_2^\ast r_n]$ will be too large. We accomplish this goal for, e.g. $h_2^\ast$, in several steps:

\noindent (a)	First, similarly as in many one-dimensional problems, we establish a local rate of convergence: 
\begin{align*}
{\hat{f}_n(x_0)-f_0(x_0)}=\mathcal{O}_{\mathbf{P}}(n^{-1/4}).
\end{align*}
	
\begin{figure}
\begin{tikzpicture}[xscale=1.5,yscale=1.5]
\fill[fill=pink] plot[smooth, samples=100, domain=-0.1:0.1] (-0.03,0) -- (0.03,0)-- (0.03,2.5)--(-0.03,2.5);
\fill[fill=blue] plot[smooth, samples=100, domain=-0.1:0.1] (0,-0.03) -- (0,0.03)-- (2.5,0.03)--(2.5,-0.03);
\fill[fill=gray!30] plot[smooth, samples=100, domain=-0.5:0] (-0.5,-0.5)--(0,-0.5)--(0,0)--(-0.5,0);
\draw[->][draw=black, thick] (-0.7,0) -- (e,0);
\draw[->][draw=black, thick] (0,-0.7)--(0,e);
\node [below] at (e,-0.1) {$(h)_1$};
\node [below] at (-0.2,-0.1) {$x_0$};
\draw[dashed,thick,pink] (-0.5,-0.5)--(-0.5, 2.5)--(0.7,2.5)--(0.7,-0.5)--(-0.5,-0.5);
\draw[dashed,thick,blue] (-0.5,-0.5)--(2.5,-0.5)--(2.5,0.7)--(-0.5,0.7)--(-0.5,-0.5);
\node [left] at (-0.1,e) {$(h)_2$};
\draw[->][thick] (1.2,1.7)--(0.1,2);
\draw[->][thick] (1.7,1.2)--(2,0.1);
\node [above] at (1.65,1.20) {large bias};
		
\fill[fill=gray!30] plot[smooth, samples=100, domain=3.5:5.6] (4.8,-0.5) -- (5.6,-0.5)-- (5.6,1.6)--(4.8,1.6);
\fill[fill=cyan] plot[smooth, samples=100, domain=5.6:5.8] (5.6,1.6) -- (5.8,1.6)-- (5.8,2.5)--(5.6,2.5);
\draw[->][draw=black, thick] (4-0.7,1.6) -- (e+4,1.6);
\draw[->][draw=black, thick] (5.6,-0.7)--(5.6,e);
\node [below] at (e+4,1.6) {$(h)_1$};
\node [left] at (3.9+1.6,e) {$(h)_2$};
\node [below] at (-0.2+1.6+4,1.55) {$x_0$};
\draw[dashed,thick,red] (5.6,-0.5)--(5.8,-0.5)--(5.8,2.5)--(4.8,2.5)--(4.8,1.6)--(5.6,1.6)--(5.6,-0.5);
\node [below] at (4.8,1.6) {$-c^{-b}$};
\node [left] at (5.6,2.4) {$c$};
\node [left] at (5.6,-0.35) {$-c^2$};
\node [below] at (5.9,1.6) {$c^{-\gamma_\ast}$};
\draw[->][thick] (4.5,0.5)--(4.8,0.7);
\node [above] at (3.95,0.5) {positive};
\node [below] at (3.95,0.5) {noise};
\node [above] at (1,-1.2) {Large deviation};    	
\node [above] at (5,-1.2) {Small deviation};
\end{tikzpicture}
\caption{Proof outlines: Large and small deviation problems.}
\label{fig:localization}
\end{figure}
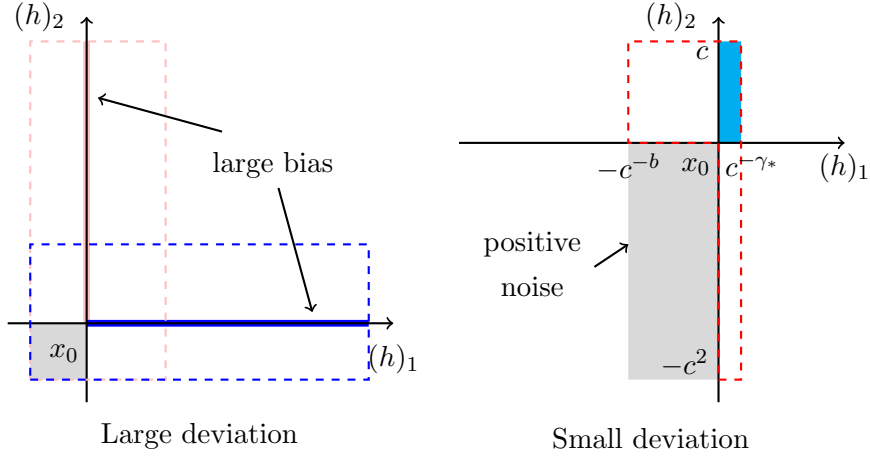

\noindent (b) Next we handle the large deviation problem. In other words, we want to show that $\max\{(h_2^\ast)_1,(h_2^\ast)_2\}\leq c$ with high probability for large $c>0$. By the max-min formula, 
\begin{align*}
\hat{f}_n(x_0)-f_0(x_0)&\geq \underbrace{\big(\bar{f_0}|_{[x_0-n^{-1/4},x_0+h_2^\ast n^{-1/4}]}-f_0(x_0)\big)}_{\textrm{bias}}+\underbrace{\bar{\xi}|_{[x_0-n^{-1/4},x_0+h_2^\ast n^{-1/4}]}}_{\textrm{noise}}.
\end{align*}
The bias term can be handled via (localized) Taylor expansion: suppose $\max\{(h_2^\ast)_1,(h_2^\ast)_2\}>c\gg 1$ (see for example the thick red and blue lines in the left panel of Figure \ref{fig:localization}), then 
\begin{align*}
\mathrm{bias}\gtrsim n^{-1/4} \big(\max\{(h_2^\ast)_1,(h_2^\ast)_2\}-\mathrm{const}.\big)\gtrsim c\cdot n^{-1/4}.
\end{align*}
The noise term is essentially contributed by the shaded area in the left panel of Figure \ref{fig:localization}:
\begin{align*}
\mathrm{noise}\gtrsim - \abs{\mathcal{O}_{\mathbf{P}}(1)} n^{-1/4}.
\end{align*}
Combining the above estimates, if $\max\{(h_2^\ast)_1,(h_2^\ast)_2\}>c$, then
\begin{align*}
n^{1/4} \big(\hat{f}_n(x_0)-f_0(x_0)\big) \gtrsim c-\abs{\mathcal{O}_{\mathbf{P}}(1)},
\end{align*}
which by (a) should occur with small probability for large $c>0$.

\noindent (c) Finally we handle the small deviation problem. 
By (b), we may assume that $\max\{(h_2^\ast)_1,(h_2^\ast)_2\}\leq c$. Let $0<b<\gamma_\ast$ be constants to be determined. Suppose now $(h_2^\ast)_1<c^{-\gamma_\ast}$ (as in the blue strip in the right panel of Figure \ref{fig:localization}). Using the max-min formula again (but in a different way) and lower bounding the bias yield that
\begin{align*}
	\hat{f}_n(x_0)-f_0(x_0)&\geq \max_{ \substack{0\leq (h_1)_1\leq c^{-b}\\0\leq (h_1)_2\leq c^2} }\bar{\xi}|_{[x_0-h_1 n^{-1/4},x_0+h_2^\ast n^{-1/4}]}\\
	&\qquad\qquad + \min_{ \substack{0\leq (h_1)_1\leq c^{-b}\\0\leq (h_1)_2\leq c^2} }\big(\bar{f_0}|_{[x_0-h_1n^{-1/4},x_0+h_2^\ast n^{-1/4}]}-f_0(x_0)\big)\\
	&\gtrsim n^{-1/4}\bigg[ \max_{ \substack{0\leq (h_1)_1\leq c^{-b}\\0\leq (h_1)_2\leq c^2} } \frac{\G(h_1,h_2^\ast)}{\prod_{k=1}^2(h_1+h_2^\ast)_k} - C_1\cdot c^2 \bigg].
\end{align*}
The idea here is to choose a larger block for $h_1$ compared with the localized block for $h_2^\ast$. This creates a large positive fluctuation of the noise process $\G(\cdot,\cdot)$ within the shaded region that dominates the relatively small fluctuation within the region in the red dashed line; see the right panel in Figure \ref{fig:localization}. Indeed, it can be shown that for $\gamma_\ast <b+1$, the following small fluctuation holds with high probability: $\G(h_1,h_2^\ast)-\G(h_1,0)\geq -C_2\cdot \sqrt{c^{2-\gamma_\ast}\log c}$ for large $C_2>0$. The scaling $c^{2-\gamma_\ast}$ is the order of the area of the region within the red dashed line in the right panel of Figure \ref{fig:localization}. On the other hand, the large positive fluctuation $\max_{0\leq (h_1)_1\leq c^{-b}, 0\leq (h_1)_2\leq c^2} \G(h_1,0)\geq C_3\cdot \sqrt{c^{2-b}}$ holds for small $C_3>0$ with high probability (see the shaded area in the right panel of Figure \ref{fig:localization}). Therefore with high probability, for $c$ large,
\begin{align*}
n^{1/4}\big(\hat{f}_n(x_0)-f_0(x_0) \big)&\gtrsim \frac{C_3\sqrt{c^{2-b}}-C_2 \sqrt{c^{2-\gamma_\ast}\log c}}{(c^2+c)(c^{-b}+c^{-\gamma_\ast})}-C_1 c^2 \\
&\geq C_4 c^{(b/2)-1}-C_1 c^2.
\end{align*}
Now by choosing $b>6$ and $\gamma_\ast \in (b,b+1)$, the above display can only occur with small probability for large $c$ by (a), so $(h_2^\ast)_1>c^{-\gamma_\ast}$ with high probability.

Once the first step (\ref{ineq:proof_outline_1}) is completed, we may proceed with the second step and conclude that 
\begin{align*}
W_{n,c}& \equiv \max_{c^{-\gamma_\ast}\bm{1}\leq h_1 \leq c\bm{1}}\min_{c^{-\gamma_\ast}\bm{1}\leq h_2\leq c\bm{1}} \mathbb{U}_n(h_1,h_2)\\
&= \max_{c^{-\gamma_\ast}\bm{1}\leq h_1 \leq c\bm{1}}\min_{c^{-\gamma_\ast}\bm{1}\leq h_2\leq c\bm{1}} \bigg[n^{1/4}\bar{\xi}|_{[x_0-h_1 n^{-1/4},x_0+h_2 n^{-1/4}]}\\
&\qquad\qquad +n^{1/4}\big(\bar{f_0}|_{[x_0-h_1 n^{-1/4},x_0+h_2 n^{-1/4}]}-f_0(x_0)\big)\bigg]\\
&\rightsquigarrow \max_{c^{-\gamma_\ast}\bm{1}\leq h_1 \leq c\bm{1}}\min_{c^{-\gamma_\ast}\bm{1}\leq h_2\leq c\bm{1}} \bigg[\frac{\sigma \cdot \G(h_1,h_2)}{\prod_{k=1}^2 \big((h_1)_k+(h_2)_k\big)}\\
&\qquad\qquad+\frac{1}{2}\sum_{k=1}^2 \partial_k f_0(x_0)\big( (h_2)_k-(h_1)_k\big) \bigg]\\
& \equiv \max_{c^{-\gamma_\ast}\bm{1}\leq h_1 \leq c\bm{1}}\min_{c^{-\gamma_\ast}\bm{1}\leq h_2\leq c\bm{1}} \mathbb{U}(h_1,h_2)\equiv W_c.
\end{align*}
To establish the weak convergence in the above display, we need to establish weak convergence of the process $\mathbb{U}_n$ to $\mathbb{U}$ in $\ell^\infty([c^{-\gamma_\ast}\bm{1},c\bm{1}]\times [c^{-\gamma_\ast}\bm{1},c\bm{1}])$. Finite dimensional convergence follows immediately by the Taylor expansion in Assumption \ref{assump:smoothness}. Asymptotic equicontinuity will be verified by general tools developed for uniform central limit theorems for partial sum processes (essentially) in \cite{alexander1987central} and further developed in \cite{van1996weak}. Note that such asymptotic equicontinuity is possible as the max-min formula searches over rectangles, rather than upper and lower sets as for the least squares estimator.

The last step attempts at localizing the limit of $W\equiv W_\infty$ by showing that ${\tilde{h}_1},{\tilde{h}_2}$ are bounded away from $\infty$ and $0$ in probability, where $\tilde{h}_1,\tilde{h}_2$ are such that $W = \mathbb{U}(\tilde{h}_1,\tilde{h}_2)$. This problem can be viewed as the limit Gaussian analogue of the first step, and therefore shares a similar proof strategy as detailed above, but further simplifications are possible due to the exact Gaussian structure in the limit.

Finally we list the key properties of the partial sum and limit processes that are used in the proofs. For simplicity, we only consider fixed balanced lattice design with $\alpha_1=\ldots=\alpha_d=1$. Fix $\rho \in (0,1)$. Let $\omega_{n,\rho} \equiv n^{-\frac{1-\rho}{1+d(1-\rho)} }$, $r_{n,\rho}\equiv \omega_{n,\rho}\bm{1}$ and for any $h_1,h_2 \in \R_{\geq 0}^d$, let 
\begin{align*}
\G_{n,\rho}(h_1,h_2)\equiv  \omega_{n,\rho}^{-1}(n\omega_{n,\rho}^d)^{-1} \sum_{i: x_0-h_1 r_{n,\rho}\leq X_i\leq x_0+h_2 r_{n,\rho}}\xi_i.
\end{align*} 
Suppose the following properties hold:
\begin{enumerate}
	\item[(P1)] $\G_{n,\rho} \rightsquigarrow \G_\rho$ in $\ell^\infty([0,c\bm{1}]\times [0,c\bm{1}])$ for any $c>0$, and the limit process $\G_\rho(\cdot,\cdot)$ is separable and self-similar with index $\rho \in (0,1)$ in the sense that for any  $\gamma\in \R_{\geq 0}^d$, $
	\G_\rho(\gamma \cdot,\gamma \cdot)=_d \big(\prod_{k=1}^d \gamma_k\big)^\rho\cdot \G_\rho(\cdot,\cdot)$.
	\item[(P2)] It holds that
	\begin{align*}
	\sup_{h>0} \frac{\abs{\G_{n,\rho}(h,\bm{1})}\vee \abs{\G_{n,\rho}(\bm{1},h)} }{\prod_{k=1}^d (h_k+1)} +\sup_{h>0} \frac{\abs{\G_{\rho}(h,\bm{1})}\vee \abs{\G_\rho(\bm{1},h)} }{\prod_{k=1}^d (h_k+1)} = \mathcal{O}_{\mathbf{P}}(1).
	\end{align*}
	\item[(P3)] $\Prob\big(\sup_{0\leq h\leq \bm{1}} \G_\rho(h,0)\leq 0\big)=\Prob\big(\sup_{0\leq h\leq \bm{1}} \G_\rho(0,h)\leq 0\big)=0$.
\end{enumerate}
Then
\begin{align*}
\omega_{n,\rho}^{-1}\big(\hat{f}_n(x_0)-f_0(x_0)\big)\rightsquigarrow K_\rho(f_0,x_0)\cdot \mathbb{D}_{d,\rho},
\end{align*}
where $K_\rho(f_0,x_0) \equiv \big\{\prod_{k=1}^d \big(\partial_k f_0(x_0)/2\big)\big\}^{\frac{1-\rho}{1+d(1-\rho)}}$ and
\begin{align*}
\mathbb{D}_{d,\rho}\equiv \sup_{h_1>0}\inf_{h_2>0} \bigg[\frac{\G_\rho(h_1,h_2)}{\prod_{k=1}^d \big((h_1)_k+(h_2)_k\big)}+\sum_{k=1}^d\big( (h_2)_k-(h_1)_k\big) \bigg].
\end{align*}
Some comments on the properties (P1)-(P3): 
\begin{itemize}
	\item (P1) requires a functional limit theory for the process $\G_{n,\rho}$ to its limit $\G_\rho$ over all compact rectangles. The self-similarity index $\rho \in (0,1)$ reflects the dependence structure within the errors $\{\xi_i\}$. In the i.i.d. case as considered in this paper, $\rho = 1/2$. 
	\item (P2) requires some uniform control of the moduli of the processes $\G_{n,\rho}, \G_\rho$. That the first term in (P2) being stochastically bounded can also be written as
	\begin{align*}
	\sup_{h>0}\omega_{n,\rho}^{-1}\big(\abs{ \bar{\xi}|_{[x_0-hr_{n,\rho},x_0+r_{n,\rho}]}}\vee \abs{ \bar{\xi}|_{[x_0-r_{n,\rho},x_0+hr_{n,\rho}]}} \big)&= \mathcal{O}_{\mathbf{P}}(1).
	\end{align*}
	For i.i.d. errors, this can established by martingale properties of the partial sum process. The stochastic boundedness of the second term in (P2) can be verified by good tail estimates on $\G_\rho$.
	\item (P3) is a regularity condition on the limit process $\G_\rho$. In the univariate case with i.i.d. errors, this can be easily verified by the reflection principle of Brownian motion.
\end{itemize}  
It is also possible to consider, with substantially increased technicalities, more general assumptions on the local smoothness of $f_0$ at $x_0$ and the design points as in Assumptions \ref{assump:smoothness}-\ref{assump:design}, but we shall omit these details here.

\subsection{Outline for the proof of Theorem \ref{thm:minimax_lower_bound}}

The basic minimax machinery we use is the following. 
\begin{proposition}\label{prop:minimax}
	Suppose that the errors $\xi_i$'s are i.i.d. $\mathcal{N}(0,\sigma^2)$. Let $f_n$ be such that $
	n\ell_2^2(f_n,f_0)\leq\alpha$ and $\abs{f_n(x_0)-f_0(x_0)}\geq \gamma_n$.
	Then
	\begin{align*}
	\inf_{\tilde{f}_n} \sup_{ f\in \{f_n,f_0\} } \E_{f} \big[ \abs{\tilde{f}_n(x_0)-f(x_0)}\big]\geq \frac{\gamma_n}{8} \exp\bigg(-\frac{\alpha}{2\sigma^2}\bigg).
	\end{align*}
	The $\ell_2$ metric is defined in the statement of Theorem \ref{thm:minimax_lower_bound}.
\end{proposition}
Results of this type are well-known in the context of density estimation \cite{groeneboom1995isotonic,jongbloed2000minimax}, which can also be viewed a special case of minimax reduction scheme with two hypothesis (cf. \cite{tsybakov2008introduction}). We provide a (short) proof in Section \ref{section:proof_minimax} in the context of regression for the convenience of the reader. 

Now the problem reduces to that of finding a permissible perturbation $f_n$. By the second step in the outline for the proof of Theorem \ref{thm:limit_distribution_pointwise}, the constants concerning the unknown regression function in the limiting distribution essentially come from the local Taylor expansion, so it is tempting to consider the local perturbation function of the form
\begin{align*}
\tilde{f}_n(x)=
\begin{cases}
f_0(x_0-hr_n) & \textrm{ if } x \in [x_0-hr_n, x_0],\\
f_0(x) & \textrm{ otherwise},
\end{cases}
\end{align*}
with a good choice of $h \in \R^d$. The complication arises from the fact that $\tilde{f}_n \notin \mathcal{F}_d$ in general, so suitable modifications are needed for a valid construction of $f_n$. Details can be found in Section \ref{section:proof_minimax}.

\section{Discussion and final remarks}\label{section:discussion}

   We developed in this paper pointwise limiting distribution theory for the max-min block estimator (\ref{def:max_min_estimator}) under local smoothness conditions at a fixed point, and considered both fixed lattice and random designs. One important question is how to use this limit distribution theory for inference. The common bootstrap is known to be inconsistent in a closely related univariate monotone density estimation problem \cite{kosorok2008bootstrapping,sen2010inconsistency}, so simple bootstrap procedures are unlikely to succeed here as well. It is an important yet non-trivial task to develop a tuning-free inference method for the multiple isotonic regression model (\ref{model}), even with the limiting distribution theory developed in this paper. A detailed study for the inference problem will therefore be pursued elsewhere. 
	
	It should however be mentioned that the limit distribution theory is foundational for further developments of inference methods. In the univariate isotonic regression, the limit distribution theory \cite{brunk1970estimation,wright1981asymptotic} is essential both in terms of the results and the proof techniques for the popular inference procedure based on likelihood ratio test developed much later (cf. \cite{banerjee2001likelihood,banerjee2007likelihood}). We expect that the results and proof techniques in this paper will also be useful in this and other directions. However, such developments would not be parallel to the univariate analysis based on the linearity of the monotonicity relationship graph.

    From another angle, the limiting distribution theory developed in this paper differs markedly from the univariate cases where usually maximum likelihood/least squares estimators are studied. The common difficulty for these MLE/LSEs in higher dimensions is that the underlying `empirical process' (= partial sum process in (\ref{def:max_min_estimator}) in our setting) is not tight in the large sample limit, so it is hard to obtain limit distribution theories for these estimators. Our results and techniques can therefore also be viewed as a compliment to the extensive literature on the limit theories for univariate MLE/LSEs.

\section{Proof of Theorem \ref{thm:limit_distribution_pointwise}}\label{section:proof_limit_dist}

Let $\epsilon_0>0$ be such that $f_0$ is (sufficiently) differentiable on the rectangle $[x_0-\epsilon_0\bm{1},x_0+\epsilon_0\bm{1}]$. Recall $\kappa_\ast =1$ in the random design case, and $\omega_n \in \R, r_n\in \R^d$ are defined in (\ref{def:r_n}). Let $d_\ast \equiv d-\kappa_\ast+1$ and $s_\ast \equiv s-\kappa_\ast+1$. We often omit the requirement that $[x_u,x_v]\cap \{X_i\}\neq \emptyset$ in (\ref{def:max_min_estimator}) for notational simplicity. In the random design setting, we assume that $P$ is the uniform distribution on $[0,1]^d$ to avoid unnecessary notational digressions; then the covariance structure of $\G$ is given by the simplified expression (\ref{eqn:cov_G_simple}).

In the sequel, we consider separately the cases for $1\leq \kappa_\ast<s+1$, and $\kappa_\ast = s+1$. In the former case, there is at least one non-trivial term in the Taylor expansion of $f_0$ at $x_0$, so the problem of limiting distribution is essentially local. In the latter case, since there is only noise present, so the problem is non-local.

\subsection{Local rate of convergence}
In this subsection we establish the local rate of convergence for $\hat{f}_n(x_0)$. This corresponds to step (a) in Section \ref{section:proof_outline_thm1}.

\begin{proposition}\label{prop:rate_estimator}
	Let $1\leq \kappa_\ast<s+1$. Assume the same conditions as in Theorem \ref{thm:limit_distribution_pointwise}. Then $
	\hat{f}_n(x_0)-f_0(x_0) = \mathcal{O}_{\mathbf{P}}(\omega_n)$.
\end{proposition}

We need the following to control the contribution of the noise.

\begin{lemma}\label{lem:size_max_partial_sum}
	Assume the same conditions as in Theorem \ref{thm:limit_distribution_pointwise}. Let $r_n$ be as in (\ref{def:r_n}). Then for any fixed $\tau>0$, in both fixed lattice and random design cases, we have $
	\sup_{h>0} \abs{\bar{\xi}|_{[x_0-h r_n,x_0+\tau r_n]}} = \mathcal{O}_{\mathbf{P}}(\omega_n)$.
\end{lemma}
\begin{proof}
See Section \ref{section:proof_lemmas}.
\end{proof}

\begin{proof}[Proof of Proposition \ref{prop:rate_estimator}]
Let $x_u^\ast$ be such that $\hat{f}_n(x_0) =  \min_{x_v\geq x_0} \bar{Y}|_{[x_u^\ast,x_v]}$. Fix a small enough $\tau>0$. Since for $n$ large, $[x_0,x_0+\tau r_n]\cap \{X_i\}\neq \emptyset$ holds in the fixed lattice design case, and with probability tending to one in the random design case, using the max-min formula we have
\begin{align}\label{ineq:local_rate_1}
\hat{f}_n(x_0) & \leq \bar{f_0}|_{[x_u^\ast,x_0+\tau r_n]}+\bar{\xi}|_{[x_u^\ast,x_0+\tau r_n]}.
\end{align}
In both fixed lattice and random design cases, by monotonicity of $f_0$, we have for $n$ large enough,
\begin{align}\label{ineq:local_rate_2}
\bar{f_0}|_{[x_u^\ast,x_0+ \tau r_n]}-f_0(x_0)&\leq f_0(x_0+\tau r_n)-f_0(x_0)\\
&= \sum_{\bm{j} \in J_\ast} \frac{\partial^{\bm{j}} f_0(x_0) }{\bm{j}!}(1+\mathfrak{o}(1)) (\tau r_n)^{\bm{j}}  = \mathcal{O}(\omega_n).\nonumber 
\end{align}
On the other hand, in both fixed lattice and random design cases, Lemma \ref{lem:size_max_partial_sum} entails that 
\begin{align}\label{ineq:local_rate_3}
\sup_{h_1>0} \abs{\bar{\xi}|_{[x_0-h_1 r_n,x_0+\tau r_n]}} = \mathcal{O}_{\mathbf{P}}(\omega_n).
\end{align}
The one-sided claim follows by combining (\ref{ineq:local_rate_1})-(\ref{ineq:local_rate_3}). The other direction follows from similar arguments. 
\end{proof}

\subsection{Localizing the estimator}

In this subsection, we tackle the large and small deviation problems, i.e. steps (b)-(c), as described in Section \ref{section:proof_outline_thm1}.

\begin{proposition}\label{prop:localize_estimator_range}
	Let $1\leq \kappa_\ast<s+1$. Assume the same conditions as in Theorem \ref{thm:limit_distribution_pointwise} and $\kappa_\ast$ is unique. Then
	\begin{align*}
	\lim_{c \to \infty} \limsup_{n \to \infty} \Prob\big(\hat{f}_n(x_0)\neq \hat{f}_{n,c}(x_0)\big)=0.
	\end{align*}
	In other words, $
	\lim_{c \to \infty} \limsup_{n \to \infty} \Prob(W_{n,c}\neq W_n)=0.$
\end{proposition}

\begin{proof}
\noindent\textbf{(Step 1: the large deviation problem).}  In this step we handle the large deviation problem. For the proof in this step only, define $r_n \equiv (\omega_n^{1/\alpha_1},\ldots,\omega_n^{1/\alpha_d})$ for notational convenience. Let $h_1^\ast,h_2^\ast$ be defined as in (\ref{def:h_1_h_2}) but using the current $r_n$, and 
\begin{align}\label{def:H_c}
H_c \equiv \big\{h\geq 0: h_\ell =0, 1\leq \ell \leq \kappa_\ast-1, \max_{\kappa_\ast \leq k\leq d} h_k\leq c\big\}.
\end{align} We will show that
\begin{align}\label{ineq:localize_estimator_0}
\lim_{c \to \infty} \limsup_{n \to \infty} \Prob\big(h_1^\ast \notin H_c\big)\vee \Prob\big(h_2^\ast \notin H_c\big)=0.
\end{align}
We only show this for $\Prob\big(h_2^\ast \notin H_c\big)$; the situation for $h_1^\ast$ is similar.

For any $h_2^\ast $, let $
\bar{h}_2^\ast \equiv h_2^\ast\wedge (\epsilon_0 (r_n)_1^{-1},\ldots,\epsilon_0(r_n)_d^{-1})$. Clearly $\bar{h}_2^\ast\leq h_2^\ast$. Then by the max-min formula and the monotonicity of $f_0$, in the fixed lattice design case it holds for any $h_1\geq 0$ that
\begin{align}\label{ineq:localize_estimator_1}
\hat{f}_n(x_0)&\geq \bar{f_0}|_{[x_0-h_1r_n,x_0+h_2^\ast r_n]}+\bar{\xi}|_{[x_0-h_1r_n,x_0+h_2^\ast r_n]}\\
&\geq \bar{f_0}|_{[x_0-h_1r_n,x_0+\bar{h}_2^\ast r_n]}+\bar{\xi}|_{[x_0-h_1r_n,x_0+h_2^\ast r_n]}.\nonumber
\end{align}
\noindent\textbf{(Step 1a: fixed design, effective dimension).}  First we will show that we may take $(h_2^\ast)_\ell = 0$ for $1\leq \ell\leq \kappa_{\ast}-1$ with high probability as $n \to \infty$. Note this is only for the fixed lattice design case. Since we have a lattice design, we only need to show that $(h_2^\ast)_\ell < n^{-\beta_\ell}(r_n)_\ell^{-1}$ for $1\leq \ell\leq \kappa_{\ast}-1$. On the event $\cup_{1\leq \ell\leq \kappa_\ast-1} \{(h_2^\ast)_\ell\geq n^{-\beta_\ell}(r_n)_\ell^{-1}\}$, we have by Lemma \ref{lem:taylor_expansion} that for $n$ large enough,
\begin{align*}
\bar{f_0}|_{[x_0-r_n\bm{1}_{[\kappa_\ast:d]},x_0+\bar{h}_2^\ast r_n]}-f_0(x_0)\gtrsim \omega_n \max_{1\leq \ell \leq \kappa_\ast-1} (h_2^\ast)_\ell^{\alpha_\ell}\geq  \omega_n\cdot n^{\delta},
\end{align*}
where the last inequality follows from Proposition \ref{prop:comparision_local_rates} that $\min_{1\leq \ell \leq \kappa_\ast-1} n^{-\beta_\ell}(r_n)_\ell^{-1}  = n^\delta$ for some $\delta>0$. On the other hand, (a slight modification of) Lemma \ref{lem:size_max_partial_sum} yields that
\begin{align}\label{ineq:localize_estimator_1_0}
\abs{\bar{\xi}|_{[x_0-r_n\bm{1}_{[\kappa_\ast:d]},x_0+h_2^\ast r_n]}\big}&\leq \sup_{h_2\geq 0} \abs{\bar{\xi}|_{[x_0-r_n\bm{1}_{[\kappa_\ast:d]},x_0+h_2 r_n]}\big} = \mathcal{O}_{\mathbf{P}}(\omega_n).
\end{align}
Combined with (\ref{ineq:localize_estimator_1}) using $h_1=(1/2)\bm{1}_{[\kappa_\ast:d]}$, this shows that on the event $\cup_{1\leq \ell\leq \kappa_\ast-1} \{(h_2^\ast)_\ell\geq 1\}$, 
\begin{align}\label{ineq:localize_estimator_2}
\hat{f}_n(x_0)-f_0(x_0)\gtrsim \omega_n\cdot n^\delta(1-\mathfrak{o}_{\mathbf{P}}(1)),
\end{align}
which can only occur with arbitrarily small probability according to Proposition \ref{prop:rate_estimator}. This means 
\begin{align}\label{ineq:localize_estimator_3}
\lim_{n \to \infty} \Prob\big(\cup_{1\leq \ell\leq \kappa_\ast-1}\{(h_2^\ast)_\ell \geq 1\}\big)=0.
\end{align}
Hence we may take $(h_2^\ast)_\ell = 0$ for $1\leq \ell\leq \kappa_\ast-1$ from now on. We remind again the reader that this claim is for fixed lattice design only.

\noindent\textbf{(Step 1b: fixed and random designs).}  Next, consider the event $\{\max_{\kappa_\ast \leq k\leq d} (h_2^\ast)_k>c\}$. For $c> \max_{s+1\leq k\leq d} (1-x_0)_k$, we only need to consider the event $\{\max_{\kappa_\ast \leq k\leq s} (h_2^\ast)_k>c\}$. Again using Lemma \ref{lem:taylor_expansion}, for $c,n$ large enough, and with probability at least $1-\mathcal{O}(n^{-2})$ in the random design case,
\begin{align*}
\bar{f_0}|_{[x_0-r_n\bm{1}_{[\kappa_\ast:d]},x_0+\bar{h}_2^\ast r_n]}-f_0(x_0)\gtrsim \tilde{\omega}_n \gtrsim \omega_n\cdot c.
\end{align*}
For the noise term, (\ref{ineq:localize_estimator_1_0}) holds in both fixed lattice and random design cases. Hence on the intersection of the event $\{\max_{\kappa_\ast \leq k\leq d} (h_2^\ast)_k>c\}$ and an event with probability tending to $1$,
\begin{align*}
\hat{f}_n(x_0)-f_0(x_0)\geq \omega_n(c-\mathcal{O}_{\mathbf{P}}(1))
\end{align*}
holds in both fixed lattice and random design cases. However, in view of Proposition \ref{prop:rate_estimator}, this occurs with arbitrarily small probability for large values of $c>0, n \in \N$. So the event $\{\max_{\kappa_\ast \leq k\leq d} (h_2^\ast)_k>c\}$ must occur with arbitrarily small probability for $c, n$ large enough. This proves (\ref{ineq:localize_estimator_0}).

\noindent \textbf{(Step 2: the small deviation problem).} In this step we handle the small deviation problem. $r_n$ is now defined as in (\ref{def:r_n}). Fix $\epsilon>0$. By Step 1, we may choose $c>0, n\in \N$ large enough such that the event $\Omega_{\epsilon,c}^{(0)}\equiv \{h_2^\ast \in H_c\}$ holds with probability at least $1-\epsilon$. Let $a,b,\gamma_\ast >0$ with $a>1, 0<b<\gamma_\ast<b+(a-1)$ be constants to be determined later on, and let $\mathcal{H}_{a,b,\gamma_\ast}(c)\equiv \{(h_1,h_2) \in \R_{\geq 0}^d\times \R_{\geq 0}^d: 0\leq (h_1)_k\leq c^a\bm{1}_{\kappa_\ast \leq k\leq s}+(x_0)_k \bm{1}_{s+1\leq k\leq d}, 0\leq (h_1)_d \leq c^{-b}, 0\leq (h_2)_k \leq c\bm{1}_{\kappa_\ast \leq k\leq s}+(1-x_0)_k \bm{1}_{s+1\leq k\leq d},0\leq (h_2)_d\leq c^{-\gamma_\ast}\}$ be defined as in Lemma \ref{lem:size_difference_gaussian_small_dev}. Consider the event $\Omega^{(1)}_c\equiv \{ (h_2^\ast)_d<c^{-\gamma_\ast}\}$. For simplicity of notation, we consider $s<d$; the case $s=d$ follows similarly with slightly different estimates due to Lemma \ref{lem:size_difference_gaussian_small_dev}. Let $Z_{ni}$ be defined by
\begin{align*}
Z_{ni}(h_1,h_2)\equiv \omega_n  \xi_{i}\bm{1}_{X_i \in [x_0-h_1 r_n, x_0+h_2 r_n]}.
\end{align*}
It is verified in the proof of Lemma \ref{lem:weak_conv_gaussian_proc} ahead that for any finite $\tau>0$,
\begin{align*}
\G_n(\cdot,\cdot)\equiv \sum_{i=1}^n Z_{ni}(\cdot,\cdot) \rightsquigarrow \sigma \cdot \G(\cdot,\cdot)\quad \textrm{in } \ell^\infty([0,\tau\bm{1}],[0,\tau\bm{1}]).
\end{align*}
Hence by Lemma \ref{lem:size_difference_gaussian_small_dev}, as long as $c>0, n\in \N$ are large enough, there exists a constant $C_1=C_1(d,\sigma,a)$ such that the  event 
\begin{align*}
\Omega^{(2)}_\epsilon\equiv \bigg\{\sup_{ (h_1,h_2) \in \mathcal{H}_{a,b,\gamma_\ast}(c)} \abs{\G_n(h_1,h_2)-\G_n(h_1,h_2\bm{1}_{[s+1:d-1]} )}\leq (C_1/\epsilon) \sqrt{c^{as_\ast-\gamma_\ast} \log c}\bigg\}
\end{align*} 
holds with probability at least $1-\epsilon$. On the other hand, by Lemma \ref{lem:small_deviation_min_max}, for $u\equiv \sqrt{c^{as_\ast-b}/(x_0)_d}\cdot \rho_\epsilon$ where $\rho_\epsilon$ is taken from Lemma \ref{lem:small_deviation_min_max}, if $a>1$ and $c>1$, it holds for $n$ large enough that
\begin{align*}
&\Prob\bigg(\min\nolimits_{ \substack{ 0\leq (h_2)_k \leq c\bm{1}_{\kappa_\ast \leq k\leq d}\\ 0\leq (h_2)_d\leq c^{-\gamma_\ast} }} \max\nolimits_{\substack{0\leq (h_1)_k\leq c^a\bm{1}_{\kappa_\ast \leq k\leq d} \\ 0\leq (h_1)_d \leq c^{-b}} } \G_n(h_1,h_2\bm{1}_{[s+1:d-1]})\leq u\bigg)\\
& \lesssim \Prob\bigg(\min\nolimits_{ \substack{ 0\leq (h_2)_k \leq c\bm{1}_{s+1 \leq k\leq d}\\ (h_2)_d=0 }} \max\nolimits_{\substack{0\leq (h_1)_k\leq c^a\bm{1}_{\kappa_\ast \leq k\leq d} \\ 0\leq (h_1)_d \leq c^{-b}} } \G(h_1,h_2)\leq u\bigg)\\
&\leq \Prob\bigg(\min\nolimits_{ \substack{ 0\leq (h_2)_k \leq \bm{1}_{s+1 \leq k\leq d}\\ (h_2)_d=0 }} \max\nolimits_{ \substack{(h_1)_k \leq \bm{1}_{\kappa_\ast \leq k\leq d}\\  (h_1)_d\leq (x_0)_d}} \sqrt{c^{as_\ast-b}/(x_0)_d}\cdot \G(h_1,h_2)\leq u\bigg)\leq \epsilon.
\end{align*}
Hence there exists some constant $C_2=C_2(x_0,\rho_\epsilon)>0$ such that the event 
\begin{align}
\Omega_\epsilon^{(3)}&\equiv \{\textrm{for any } 0\leq (h_2)_k\leq c\bm{1}_{\kappa_\ast \leq k\leq d}, 0\leq (h_2)_d\leq c^{-\gamma_\ast}, \\
&\qquad \exists h_1\textrm{ with } 0\leq (h_1)_k \leq c^a\bm{1}_{\kappa_\ast \leq k\leq d}, 0\leq (h_1)_d\leq c^{-b} \nonumber\\
&\qquad \textrm{ such that } \G_n(h_1,h_2\bm{1}_{[s+1:d-1]})\geq C_2\cdot \sqrt{c^{as_\ast-b}} \}\nonumber
\end{align}
holds with probability at least $1-\epsilon$ for $n$ large enough.

\noindent \textbf{(Step 2a: fixed design, noise).} In the fixed lattice design case, on the event $\Omega_{\epsilon,c}^{(0)}\cap \Omega_c^{(1)}\cap \Omega_\epsilon^{(2)}\cap \Omega_\epsilon^{(3)}$, we have for $c>0, n \in\N$ large enough,
\begin{align}\label{ineq:localize_estimator_4}
&\omega_n^{-1} \max_{ \substack{0\leq (h_1)_k\leq c^a\bm{1}_{\kappa_\ast \leq k\leq d}\\ 0\leq (h_1)_d\leq c^{-b}} } \bar{\xi}|_{[x_0-h_{1} r_n, x_0+h_{2}^\ast r_n ]} \\
& = \max_{ \substack{0\leq (h_1)_k\leq c^a\bm{1}_{\kappa_\ast \leq k\leq d}\\ 0\leq (h_1)_d\leq c^{-b}} }  \frac{  \sum_{i=1}^n Z_{ni}(h_1,h_2^\ast) }{\omega_n^2 \prod_{k=\kappa_\ast}^d  \big( \floor{(h_1r_n)_k n^{\beta_k}}+\floor{(h_2^\ast r_n)_k n^{\beta_k}}+1 \big)}\nonumber\\
&\geq \max_{ \substack{0\leq (h_1)_k\leq c^a\bm{1}_{\kappa_\ast \leq k\leq d}\\ 0\leq (h_1)_d\leq c^{-b}} } \frac{\G_n(h_1,h_2^\ast\bm{1}_{[s+1:d-1]})-(C_1/\epsilon)\sqrt{c^{as_\ast-\gamma_\ast}\log c}}{\omega_n^2 \prod_{k=\kappa_\ast}^d  \big( \floor{(h_1r_n)_k n^{\beta_k}}+\floor{(h_2^\ast r_n)_k n^{\beta_k}}+1 \big)}\nonumber\\
&\geq \frac{C_2\sqrt{c^{as_\ast-b}}- (C_1/\epsilon) \sqrt{c^{as_\ast-\gamma_\ast} \log c}}{(c^a+c)^{s_\ast}  (c^{-b}+c^{-\gamma_\ast})(1+\mathfrak{o}(1))} \geq C_3 \cdot c^{(b-as_\ast)/2}.\nonumber
\end{align}
\noindent \textbf{(Step 2b: random design, noise).}
In the random design case, let $h_1(c),h_2(c)$ be such that $(h_1(c))_k = c^a \bm{1}_{\kappa_\ast \leq k \leq s} + (x_0)_k \bm{1}_{s+1\leq k\leq d-1} + c^{-b} \bm{1}_{k=d}$, and $(h_2(c))_k = c \bm{1}_{\kappa_\ast\leq k \leq s} + (1-x_0)_k \bm{1}_{s+1\leq k\leq d-1} + c^{-\gamma_\ast} \bm{1}_{k=d}$. Using Bernstein's inequality, 
\begin{align*}
\Prob\big(\bigabs{ \big(\Prob_n-P\big) \bm{1}_{X \in [x_0-h_1(c) r_n,x_0+h_2(c) r_n]} }\geq \sigma_c^2\big)\leq Ce^{-C^{-1}  n\sigma_c^2},
\end{align*}
where $\sigma_c^2 \equiv \mathrm{Var}(\bm{1}_{X \in [x_0-h_1(c) r_n,x_0+h_2(c) r_n]})\approx c^{as_\ast-b} \prod_{k=\kappa_\ast}^d (r_n)_k $ for $c, n$ large. Hence the event 
\begin{align}
\Omega_{c}^{(4)}&\equiv \{\Prob_n \bm{1}_{X \in [x_0-h_1(c) r_n,x_0+h_2(c) r_n]} \\
&\qquad\qquad\leq  P \bm{1}_{X \in [x_0-h_1(c) r_n,x_0+h_2(c) r_n]} + \sigma_c^2 \}\nonumber
\end{align} 
occurs with probability at least $1- C e^{-C^{-1}  n\sigma_c^2}$. So, in the random design setting, on the event $\Omega_{\epsilon,c}^{(0)}\cap \Omega_c^{(1)}\cap \Omega_\epsilon^{(2)}\cap \Omega_\epsilon^{(3)} \cap \Omega_c^{(4)}$, it holds that
\begin{align}\label{ineq:localize_estimator_5}
&\omega_n^{-1} \max_{ \substack{0\leq (h_1)_k\leq c^a\bm{1}_{\kappa_\ast \leq k\leq d}\\ 0\leq (h_1)_d\leq c^{-b}} } \bar{\xi}|_{[x_0-h_{1} r_n, x_0+h_{2}^\ast r_n ]}\\
& = \max_{ \substack{0\leq (h_1)_k\leq c^a\bm{1}_{\kappa_\ast \leq k\leq d}\\ 0\leq (h_1)_d\leq c^{-b}} }  \frac{  \sum_{i=1}^n Z_{ni}(h_1,h_2^\ast) }{\omega_n^2 \big(1\vee n\Prob_n \bm{1}_{X \in [x_0-h_1 r_n, x_0+h_2^\ast r_n]}\big) }\nonumber\\
&\geq \frac{C_2\sqrt{c^{as_\ast-b}}- (C_1/\epsilon) \sqrt{c^{as_\ast-\gamma_\ast} \log c}}{(c^a+c)^{s_\ast}  (c^{-b}+c^{-\gamma_\ast})+ (1+\mathfrak{o}(1)) c^{as_\ast-b} } \geq C_3 \cdot c^{(b-as_\ast)/2}.\nonumber
\end{align}
\noindent \textbf{(Step 2c: fixed and random designs, bias).} On the other hand, in both fixed lattice and random design cases,
\begin{align}\label{ineq:localize_estimator_6}
&\min_{ \substack{0\leq (h_1)_k\leq c^a\bm{1}_{\kappa_\ast \leq k\leq d}\\ 0\leq (h_1)_d\leq c^{-b}} } \bar{f_0}|_{[x_0-h_{1} r_n, x_0+h_{2}^\ast r_n ]}-f_0(x_0) \\
&\geq f_0(x_0-c^a \bm{1}_{[\kappa_\ast:d]} r_n)-f_0(x_0) \geq - C_4\cdot c^{a\max_{\kappa_\ast \leq k\leq s} \alpha_k}\cdot \omega_n.\nonumber
\end{align}
Combining the estimates (\ref{ineq:localize_estimator_4}), (\ref{ineq:localize_estimator_5}) and (\ref{ineq:localize_estimator_6}), we see that for fixed $\epsilon>0$, if $c>0, n\in \N$ are chosen large enough, on the intersection of the event $\{(h_2^\ast)_d<c^{-\gamma_\ast}\}$ and an event with probability at least $1-4\epsilon$, it holds that
\begin{align*}
\hat{f}_n(x_0)-f_0(x_0)&\geq \max_{ \substack{0\leq (h_1)_k\leq c^a\bm{1}_{\kappa_\ast \leq k\leq d}\\ 0\leq (h_1)_d\leq c^{-b}} } \bar{\xi}|_{[x_0-h_{1} r_n, x_0+h_{2}^\ast r_n ]}\\
&\qquad\qquad+\min_{ \substack{0\leq (h_1)_k\leq c^a\bm{1}_{\kappa_\ast \leq k\leq d}\\ 0\leq (h_1)_d\leq c^{-b}} } \bar{f_0}|_{[x_0-h_{1} r_n, x_0+h_{2}^\ast r_n ]}-f_0(x_0)\\
&\geq  \omega_n c^{a\max_{\kappa_\ast \leq k\leq s} \alpha_k} \big(C_3\cdot c^{(b-as_\ast)/2-a\max_{\kappa_\ast \leq k\leq s} \alpha_k}-C_4 \big).
\end{align*}
We choose $a\geq3, b \geq 2(1+a\max_{\kappa_\ast \leq k\leq s} \alpha_k)+as_\ast$ and $\gamma_\ast = b+1$, so that the above display can only occur with arbitrarily small probability by Proposition \ref{prop:rate_estimator} for large $c>0, n \in \N$. Hence the event $\{(h_2^\ast)_d<c^{-\gamma_\ast}\}$, and thereby $\{\min_{\kappa_\ast \leq k\leq d} (h_2^\ast)_k<c^{-\gamma_\ast}\}$, must occur with arbitrarily small probability for large $c>0$. The small deviation for $h_1^\ast$ can be handled similarly so we omit the details. 
\end{proof}

\subsection{Compact convergence}

In this subsection, we establish the weak convergence of the localized process $W_{n,c}$ to the localized limit $W_c$.

\begin{proposition}\label{prop:compact_converge}
	Let $1\leq \kappa_\ast<s+1$. Assume the same conditions as in Theorem \ref{thm:limit_distribution_pointwise}. For any $c>1$, 
	\begin{align*}
	&W_{n,c}\equiv \omega_n^{-1} \big(\hat{f}_{n,c}(x_0)-f_0(x_0)\big)\\
	&\qquad\rightsquigarrow \max_{c^{-\gamma_\ast}\bm{1}\leq h_1 \leq c\bm{1}}\min_{c^{-\gamma_\ast}\bm{1}\leq h_2\leq c\bm{1}} \bigg[\frac{\sigma \cdot \G(h_1,h_2)}{\prod_{k=\kappa_\ast}^d \big((h_1)_k+(h_2)_k\big)}\\
	&\qquad\qquad+\sum_{ \substack{\bm{j} \in J_\ast,\\ j_k =0,1\leq k\leq \kappa_\ast-1}} \frac{ \partial^{\bm{j}} f_0(x_0)}{(\bm{j}+\bm{1})!} \prod_{k=\kappa_\ast}^s \frac{(h_2)_k^{j_k+1}-(-h_1)_k^{j_k+1}}{(h_2)_k+(h_1)_k}\bigg]\equiv W_c.
	\end{align*}
	Here $\sigma$ and $\G$ are specified in Theorem \ref{thm:limit_distribution_pointwise}.
\end{proposition}

We need the following functional central limit theorem.

\begin{lemma}\label{lem:weak_conv_gaussian_proc}
For any $h_1,h_2>0$, let
\begin{align*}
\G_n(h_1,h_2)\equiv \omega_n \sum_{i: x_0-h_1 r_n\leq X_i\leq x_0+h_2 r_n} \xi_i.
\end{align*}
Then for any $c>1$, $\G_n \rightsquigarrow \sigma \cdot  \G$ in $\ell^\infty\big([c^{-\gamma_\ast}\bm{1},c\bm{1}]\times [c^{-\gamma_\ast}\bm{1},c\bm{1}]\big)$.
\end{lemma}

\begin{proof}
See Section \ref{section:proof_lemmas}.
\end{proof}

\begin{proof}[Proof of Proposition \ref{prop:compact_converge}]
	Note that in the fixed lattice design case,
	\begin{align*}
	W_{n,c}& =  \max_{c^{-\gamma_\ast}\bm{1}\leq h_1 \leq c\bm{1}}\min_{c^{-\gamma_\ast}\bm{1}\leq h_2\leq c\bm{1}}  \frac{\omega_n^{-1} }{\prod_{k=\kappa_\ast}^d \big( \floor{ (h_1r_n)_k n^{\beta_k}}+\floor{ (h_2r_n)_k n^{\beta_k}}+1\big) }\\
	&\qquad\bigg[\sum_{i: x_0-h_1 r_n\leq X_i\leq x_0+h_2 r_n} \xi_i+  \sum_{i: x_0-h_1 r_n\leq X_i\leq x_0+h_2 r_n}\big(f_0(X_i)-f_0(x_0)\big) \bigg]\\
	& = \max_{c^{-\gamma_\ast}\bm{1}\leq h_1 \leq c\bm{1}}\min_{c^{-\gamma_\ast}\bm{1}\leq h_2\leq c\bm{1}} \bigg[\frac{\G_n(h_1,h_2)}{\prod_{k=\kappa_\ast}^d \big((h_1)_k+(h_2)_k\big)}\cdot (1+\mathfrak{o}(1))\\
	&\qquad\qquad+\sum_{ \substack{\bm{j} \in J_\ast,\\ j_k =0,1\leq k\leq \kappa_\ast-1}} \frac{ \partial^{\bm{j}} f_0(x_0)}{(\bm{j}+\bm{1})!} \prod_{k=\kappa_\ast}^s \frac{(h_2)_k^{j_k+1}-(-h_1)_k^{j_k+1}}{(h_2)_k+(h_1)_k}\bigg]+\mathfrak{o}(1).
	\end{align*}
	Here the last equality follows from Lemma \ref{lem:taylor_expansion}: for $c^{-\gamma_\ast}\bm{1}\leq h_1,h_2\leq c\bm{1}$,
	\begin{align*}
	& \frac{\sum_{i: x_0-h_1 r_n\leq X_i\leq x_0+h_2 r_n}\big(f_0(X_i)-f_0(x_0)\big)}{\prod_{k=\kappa_\ast}^d \big( \floor{ (h_1r_n)_k n^{\beta_k}}+\floor{ (h_2r_n)_k n^{\beta_k}}+1\big) }\\
	& = \mathfrak{o}(\omega_n)+\omega_n \sum_{ \substack{\bm{j} \in J_\ast,\\ j_k =0,1\leq k\leq \kappa_\ast-1}} \frac{ \partial^{\bm{j}} f_0(x_0)}{(\bm{j}+\bm{1})!} \prod_{k=\kappa_\ast}^s \frac{(h_2)_k^{j_k+1}-(-h_1)_k^{j_k+1}}{(h_2)_k+(h_1)_k}.
	\end{align*}
	Since the map $\max_{c^{-\gamma_\ast}\bm{1}\leq h_1 \leq c\bm{1}}\min_{c^{-\gamma_\ast}\bm{1}\leq h_2\leq c\bm{1}} : \R^{ [c^{-\gamma_\ast}\bm{1},c\bm{1}]\times [c^{-\gamma_\ast}\bm{1},c\bm{1}]}\to \R$ is continuous with respect to $\pnorm{\cdot}{\infty}$ on $[c^{-\gamma_\ast}\bm{1},c\bm{1}]\times [c^{-\gamma_\ast}\bm{1},c\bm{1}]$, the claim of the proposition for the fixed lattice design case follows by Lemma \ref{lem:weak_conv_gaussian_proc} and the continuous mapping theorem.	The random design case follows from similar arguments by using Lemma \ref{lem:ratio_indicator}. 
\end{proof}

\subsection{Localizing the limit}

In this subsection, we establish that the limit $W$ can be localized through $W_c$ in the sense described in Section \ref{section:proof_outline_thm1}.

\begin{proposition}\label{prop:localize_limit_range}
Let $1\leq \kappa_\ast<s+1$. For $c>1$, let $W_c$ be as in Proposition \ref{prop:compact_converge}, and $W\equiv W_\infty$. Then $\lim_{c \to \infty}\Prob\big(W_c\neq W\big)=0$.
\end{proposition}

To prove Proposition \ref{prop:localize_limit_range}, we need the following.

\begin{lemma}\label{lem:finite_expectation_W}
	Let $W$ be defined as in Proposition \ref{prop:localize_limit_range}, then $\pnorm{W}{\psi_2}<\infty$. Here $\pnorm{\cdot}{\psi_2}$ is the sub-Gaussian Orcliz norm (definition see e.g. \cite{van1996weak}).
\end{lemma}

\begin{lemma}\label{lem:brownian_motion_uniform}
	Let $\G$ be defined as in Theorem \ref{thm:limit_distribution_pointwise}. Then for any $u\geq 1$,
	\begin{align*}
	\Prob\bigg(\max_{h_1>0}\frac{ \abs{\G(h_1,\bm{1})}}{\prod_{k=\kappa_\ast}^d \big((h_1)_k+1\big)}>u\bigg) \leq C_d \exp(-u^2/C_d).
	\end{align*}
	Here $C_d>0$ is a constant depending only on $d$.
\end{lemma}

\begin{proof}[Proofs]
See Section \ref{section:proof_lemmas}.
\end{proof}

\begin{proof}[Proof of Proposition \ref{prop:localize_limit_range}]
For simplicity of notation we assume $\sigma=1$ without loss of generality and set $a_{\bm{j}}\equiv \partial^{\bm{j}} f_0(x_0)/ (\bm{j}+\bm{1})!$. The strategy of the proof largely follows that of Proposition \ref{prop:localize_estimator_range}, but with some simplifications. Let 
\begin{align}\label{def:U}
\mathbb{U}(h_1,h_2)\equiv \frac{\G(h_1,h_2)}{\prod_{k=\kappa_\ast}^d \big((h_1)_k+(h_2)_k\big)}+\sum_{ \substack{\bm{j} \in J_\ast,\\ j_k =0,1\leq k\leq \kappa_\ast-1}}a_{\bm{j}} \prod_{k=\kappa_\ast}^s \frac{(h_2)_k^{j_k+1}-(-h_1)_k^{j_k+1}}{(h_2)_k+(h_1)_k}.
\end{align}
Let $h_1^\ast, h_2^\ast \in \R^d_{\geq 0}$ be such that $W = \mathbb{U}(h_1^\ast,h_2^\ast)$. Since the Gaussian process $\G$ only depends on the last $d_\ast$ coordinates of its arguments, we may assume that $(h_i^\ast)_\ell = 0$ for $1\leq \ell\leq \kappa_\ast-1$ and $i=1,2$.

\noindent \textbf{(Step 1).} We will first show that
 \begin{align}\label{ineq:localization_limit_1}
	\lim_{c \to \infty} \Prob(h_1^\ast \notin H_c^\ast)\vee \Prob(h_2^\ast \notin H_c^\ast)=0,
\end{align}
where $H_c^\ast$ is defined in (\ref{def:H_c}). We only need to  prove that $\{\max_{\kappa_\ast \leq k\leq s} (h_2^\ast)_k\leq c\}$ holds with large probability for $c$ large. Using a similar argument for the inequality as in the proof of Lemma \ref{lem:taylor_expansion}, on the event $\{\max_{\kappa_\ast \leq k\leq s} (h_2^\ast)_k>c\}$, 
\begin{align*}
W &\geq \mathbb{U}(\bm{1}, h_2^\ast) =\frac{\G(\bm{1},h_2^\ast)}{\prod_{k=\kappa_\ast}^d \big((h_2^\ast)_k+1)} + \sum_{ \substack{\bm{j} \in J_\ast,\\ j_k =0,1\leq k\leq \kappa_\ast-1}} a_{\bm{j}} \prod_{k=\kappa_\ast}^s \frac{(h_2^\ast)_k^{j_k+1}-(-1)^{j_k+1}}{(h_2^\ast)_k+1}\\
&\geq  -\sup_{h_2\geq 0}\frac{ \abs{\G(\bm{1},h_2)}}{\prod_{k=\kappa_\ast}^d \big((h_2)_k+1\big)}+ \mathcal{O}\left(\max_{\kappa_\ast\leq k\leq s} \big(1\vee (h_2^\ast)_k^{\alpha_k}\big)\right) \\
&\geq - \abs{\mathcal{O}_{\mathbf{P}}(1)}+\mathcal{O}(c).
\end{align*}
The last inequality follows from Lemma \ref{lem:brownian_motion_uniform}. On the other hand, by Lemma \ref{lem:finite_expectation_W} we know that $W =\mathcal{O}_{\mathbf{P}}(1)$, this means that necessarily $
\lim_{c \to \infty}\Prob(h_2^\ast \notin H_c^\ast) =0$. Similarly we can show that $\lim_{c \to \infty}\Prob(h_1^\ast \notin H_c^\ast) =0$, thereby proving the claim (\ref{ineq:localization_limit_1}).

\noindent \textbf{(Step 2).} Next we handle the small deviation problem. Using similar arguments as in the proof of Proposition \ref{prop:localize_estimator_range}, on the intersection of the event $\{ (h_2^\ast)_d<c^{-\gamma_\ast}\}$ and an event with probability at least $1-4\epsilon$, it holds that
\begin{align*}
W&\geq  \max\nolimits_{ \substack{0\leq (h_1)_k\leq c^a\bm{1}_{\kappa_\ast \leq k\leq d}\\ 0\leq (h_1)_d\leq c^{-b}} } \min\nolimits_{ \substack{0\leq (h_2)_k\leq c\bm{1}_{\kappa_\ast\leq k\leq d}\\ 0\leq (h_2)_d\leq c^{-\gamma_\ast}}} \frac{\G(h_1,h_2)}{\prod_{k=\kappa_\ast}^d (h_1+h_2)_k}-C_4\cdot c^{a\max_{\kappa_\ast \leq k\leq s} \alpha_k}\\
&\geq c^{a\max_{\kappa_\ast \leq k\leq s} \alpha_k}\big(C_3\cdot c^{(b-as_\ast)/2-a\max_{\kappa_\ast \leq k\leq s} \alpha_k}-C_4 \big)\to \infty
\end{align*}
as $c \to \infty$ by choosing $a\geq3, b \geq 2(1+a\max_{\kappa_\ast \leq k\leq s} \alpha_k)+as_\ast$ and $\gamma_\ast = b+1$. The claim follows from Lemma \ref{lem:finite_expectation_W}.
\end{proof}

\subsection{Completion of the proof of Theorem \ref{thm:limit_distribution_pointwise} for $1\leq \kappa_\ast<s+1$}\label{section:proof_completion}

\begin{proof}[Proof of Theorem \ref{thm:limit_distribution_pointwise}]
By Proposition \ref{prop:local_delocal} combined with Propositions \ref{prop:localize_estimator_range}-\ref{prop:localize_limit_range}, it follows that $\omega_n^{-1}\big(\hat{f}_n(x_0)-f_0(x_0)\big)$ converges in distribution to the desired random variable (up to a scaling factor of $\sigma$). Hence we only need to verify the distributional equality in the statement of the theorem, when all mixed derivatives of $f_0$ vanish at $x_0$ in $J_\ast$. To this end, let $\mathbb{U}$ be defined as in the proof of Proposition \ref{prop:localize_limit_range} in (\ref{def:U}) (with all mixed derivative terms vanishing). Then for $\gamma_0,\gamma_1=\ldots=\gamma_{\kappa_\ast-1}=0, \gamma_{\kappa_\ast}\ldots,\gamma_s,\gamma_{s+1}=\ldots=\gamma_d=1$ such that
\begin{align}\label{ineq:limit_dist_rescale}
&\gamma_0 \bigg(\prod_{k=\kappa_\ast}^d \gamma_k\bigg)^{-1/2}=\sigma ,\quad  \gamma_0\gamma_k^{\alpha_k} = \frac{\partial_k^{\alpha_k} f_0(x_0)}{(\alpha_k+1)!},\quad \kappa_\ast \leq k\leq s.
\end{align}
we have
\begin{align*}
&\gamma_0\cdot \mathbb{U}\left( \big(\gamma_k (h_1)_k\big)_{k=1}^d, \big(\gamma_k (h_2)_k\big)_{k=1}^d\right)\\
& = \gamma_0 \cdot \frac{\G\left(\big(\gamma_k (h_1)_k\big)_{k=1}^d, \big(\gamma_k (h_2)_k\big)_{k=1}^d\right)}{\prod_{k=\kappa_\ast}^d \gamma_k \cdot \prod_{k=\kappa_\ast}^d \big((h_1)_k+(h_2)_k\big)}+\gamma_0 \sum_{k=\kappa_\ast}^s \gamma_k^{\alpha_k}\frac{ (h_2)_k^{\alpha_k+1}-(h_1)_k^{\alpha_k+1}}{(h_2)_k+(h_1)_k}\\
& =_d \bigg[\gamma_0 \bigg(\prod_{k=\kappa_\ast}^d \gamma_k\bigg)^{-1/2}\bigg]\cdot \frac{ \G(h_1,h_2)}{\prod_{k=\kappa_\ast}^d \big((h_1)_k+(h_2)_k\big)}+ \sum_{k=\kappa_\ast}^s (\gamma_0 \gamma_k^{\alpha_k}) \frac{ (h_2)_k^{\alpha_k+1}-(h_1)_k^{\alpha_k+1}}{(h_2)_k+(h_1)_k}\\
& = \frac{\sigma\cdot  \G(h_1,h_2)}{\prod_{k=\kappa_\ast}^d \big((h_1)_k+(h_2)_k\big)}+\sum_{k=\kappa_\ast}^s  \frac{\partial_k^{\alpha_k} f_0(x_0)}{(\alpha_k+1)!} \frac{ (h_2)_k^{\alpha_k+1}-(h_1)_k^{\alpha_k+1}}{(h_2)_k+(h_1)_k}.
\end{align*}
This shows that under the choice (\ref{ineq:limit_dist_rescale}), 
\begin{align*}
&\omega_n^{-1}\big(\hat{f}_n(x_0)-f_0(x_0)\big)\\
&\rightsquigarrow \sup_{ \substack{h_1>0,\\ (h_1)_k\leq (x_0)_k, \\s+1\leq k\leq d}}\inf_{\substack{h_2>0,\\ (h_2)_k\leq (1-x_0)_k, \\s+1\leq k\leq d}}\gamma_0\cdot \mathbb{U}\left( \big(\gamma_k (h_1)_k\big)_{k=1}^d, \big(\gamma_k (h_2)_k\big)_{k=1}^d\right)\\
&= \sup_{ \substack{h_1>0,\\ (h_1)_k\leq (x_0)_k, \\s+1\leq k\leq d}}\inf_{\substack{h_2>0,\\ (h_2)_k\leq (1-x_0)_k, \\s+1\leq k\leq d}} \gamma_0\cdot\mathbb{U}(h_1,h_2).
\end{align*}
Finally we only need to note that solving (\ref{ineq:limit_dist_rescale}) yields that 
\begin{align*}
\gamma_0 = \bigg(\sigma^2 \prod_{k=\kappa_\ast}^s \bigg(\frac{\partial_k^{\alpha_k} f_0(x_0)}{(\alpha_k+1)!}\bigg)^{1/\alpha_k}\bigg)^{\frac{1}{2+\sum_{k=\kappa_\ast}^s \alpha_k^{-1}}}.
\end{align*}
This completes the proof.
\end{proof}

\subsection{Proof of Theorem \ref{thm:limit_distribution_pointwise} for $\kappa_\ast = s+1$}

\begin{proof}[Proof of Theorem \ref{thm:limit_distribution_pointwise}]	
The strategy of the proof follows the general principle developed for the case $\kappa_\ast <s+1$, so we only provide a sketch for the fixed lattice design case. 

First, by similar arguments as in the proof of Proposition \ref{prop:rate_estimator}, we can establish a local rate of convergence
\begin{align}\label{ineq:proof_noise_case_1}
n_\ast^{1/2}\big(\hat{f}_n(x_0)-f_0(x_0)\big) = \mathcal{O}_{\mathbf{P}}(1).
\end{align}
Second, note that $(h_1^\ast)_k\leq (x_0)_k, (h_2^\ast)_k\leq (1-x_0)_k$ for $s+1\leq k\leq d$ so there is no large deviation problem. For the small deviation problem, let $2<b<\gamma_\ast$ be some fixed constants, and consider the event $\Omega_c^{(1)}\equiv \{ (h_2^\ast)_d < c^{-\gamma_\ast}\}$. We may show, similar to Lemma \ref{lem:size_difference_gaussian_small_dev}, that there exists some $C_1 = C_1(\sigma,d)$ such that for $c>0, n\in \N$ large enough, the event 
\begin{align*}
\Omega^{(2)}_\epsilon\equiv \bigg\{\sup_{ \substack{0\leq (h_1)_k\leq (x_0)_k \bm{1}_{s+1\leq k\leq d} \\  0\leq (h_1)_d\leq c^{-b} \\0\leq (h_2)_k \leq (1-x_0)_k\bm{1}_{s+1\leq k\leq d}\\ (h_2)_d\leq c^{-\gamma_\ast} }} \abs{\G_n(h_1,h_2)-\G_n(h_1,h_2\bm{1}_{[1:d-1]})}\leq (C_1/\epsilon) \sqrt{c^{-\gamma_\ast} \log c}\bigg\}
\end{align*} 
holds with probability at least $1-\epsilon$ for $n$ large enough. By Lemma \ref{lem:small_deviation_min_max}, there exists some $C_2=C_2(\epsilon)>0$ such that the event $\Omega_\epsilon^{(3)}\equiv \{\textrm{for any } 0\leq (h_2)_k\leq (1-x_0)_k \bm{1}_{s+1\leq k\leq d}, (h_2)_d = 0, \textrm{ there exists } 0\leq (h_1)_k\leq (x_0)_k\bm{1}_{s+1\leq k\leq d}, (h_1)_d\leq c^{-b}\textrm{ such that }\G_n(h_1,h_2\bm{1}_{[1:d-1]})\geq C_2 \sqrt{c^{-b}}\}$ holds with probability $1-\epsilon$ for $n$ large enough. On the event $\Omega_c^{(1)}\cap \Omega_\epsilon^{(2)}\cap \Omega_\epsilon^{(3)}$, we have
\begin{align*}
&n_\ast^{1/2}\big(\hat{f}_n(x_0)-f_0(x_0)\big)\\
&\geq n_\ast^{1/2}\max_{\substack{0\leq (h_1)_k\leq (x_0)_k \bm{1}_{s+1\leq k\leq d}\\ 0\leq (h_1)_d \leq c^{-b}}} \bar{\xi}|_{[x_0-h_1\bm{1}_{[s+1:d]}, x_0+h_2^\ast\bm{1}_{[s+1:d]}]}\\
&\geq \max_{\substack{0\leq (h_1)_k\leq (x_0)_k \bm{1}_{s+1\leq k\leq d}\\ 0\leq (h_1)_d \leq c^{-b}}}  \frac{\G_n(h_1,h_2^\ast\bm{1}_{[1:d-1]})-(C_1/\epsilon)\sqrt{c^{-\gamma_\ast}\log c}}{n_\ast^{-1}\cdot \prod_{k=s+1}^d  \big( \floor{(h_1r_n)_k n^{\beta_k}} +\floor{(h_2^\ast r_n)_k n^{\beta_k}} +1 \big)}\\
&\geq \frac{C_2\sqrt{c^{-b}} - (C_1/\epsilon) \sqrt{c^{-\gamma_\ast} \log  c}}{(c^{-b}+c^{-\gamma_\ast})(1+\mathfrak{o}(1))}.
\end{align*}
Hence for $c>0, n\in \N$ large enough, on the intersection of $\Omega_c^{(1)}$ and an event with probability at least $1-2\epsilon$, 
\begin{align*}
n_\ast^{1/2}\big(\hat{f}_n(x_0)-f_0(x_0)\big)\geq C_3\cdot c^{b/2},
\end{align*}
where $C_3=C_3(C_1,C_2,\epsilon)$. However, by (\ref{ineq:proof_noise_case_1}), this cannot occur with high probability for large $c>0$. This concludes the small deviation problem. The rest of the proofs parallels that in the case $\kappa_\ast<s+1$ so we omit details.
\end{proof}

\section{Proof of Theorem \ref{thm:minimax_lower_bound}}\label{section:proof_minimax}

We first prove Proposition \ref{prop:minimax}.

\begin{proof}[Proof of Proposition \ref{prop:minimax}]
The proof is basically contained in \cite{tsybakov2008introduction}. We provide some details for the convenience of the reader, only in the context of fixed lattice design case. The random design follows from similar arguments. To match the notation used therein, let $s_n = \gamma_n/2$. Note that
\begin{align*}
&\inf_{\tilde{f}_n} \sup_{ f\in \{f_n,f_0\} } \E_{f} \big[s_n^{-1} \abs{\tilde{f}_n(x_0)-f(x_0)}\big]\\
&\geq \inf_{\tilde{f}_n} \sup_{ f\in \{f_n,f_0\} } \Prob_{f} \big[ \abs{\tilde{f}_n(x_0)-f(x_0)}\geq s_n\big]\geq p_{e,1},
\end{align*}
where the quantity $p_{e,1}$ is defined in page 80 of \cite{tsybakov2008introduction}. Let $P_f^n$ denote the probability measure corresponding to the Gaussian regression model $Y_i=f(X_i)+\xi_i, 1\leq i\leq n$. Then the Kullback-Leibler divergence between $P_f^n$ and $P_g^n$ is given by the rescaled discrete $\ell_2$ distance between $f$ and $g$: 
\begin{align*}
D_{\textrm{KL}}(P_f^n||P_g^n)= \frac{n}{2\sigma^2}\ell_2^2(f,g)= \frac{1}{2\sigma^2}\cdot \sum_{i=1}^n \big(f(X_i)-g(X_i)\big)^2.
\end{align*}
Now Theorem 2.2 of \cite{tsybakov2008introduction} applies to conclude.
\end{proof}

\begin{proof}[Proof of Theorem \ref{thm:minimax_lower_bound}]
We assume that $f_0$ is locally $C^{\max_{1\leq k\leq s} \alpha_k}$ at $x_0$ with vanishing mixed derivatives for simplicity of notation. The lower bound for the case $\kappa_\ast = s+1$ is trivial, so we only consider the case $\kappa_\ast<s+1$. Consider the fixed lattice design case. Let $r_n$ be defined as in (\ref{def:r_n}), and $h$ be a vector defined by
\begin{align*}
h_k&\equiv \bigg(\frac{\gamma (\alpha_k+1)!}{\partial_k^{\alpha_k} f_0(x_0)}\bigg)^{1/\alpha_k} \bm{1}_{\kappa_\ast \leq k\leq s}+\tau \bm{1}_{s+1\leq k\leq d},
\end{align*}
where
\begin{align*}
\gamma \equiv \bigg(\frac{2\sigma^2}{ s_\ast (2d \pnorm{\bm{\alpha}}{\infty})^{d+2\pnorm{\bm{\alpha}}{\infty}} \sum\limits_{\kappa_\ast\leq k\leq s}(\alpha_k+1)}\prod_{k=\kappa_\ast}^s \bigg(\frac{\partial_k^{\alpha_k} f_0(x_0)}{(\alpha_k+1)!}\bigg)^{1/\alpha_k}\bigg)^{\frac{1}{2+\sum_{k=\kappa_\ast}^s \alpha_k^{-1}}},
\end{align*}
and $\tau>0$ is a small enough fixed constant. Consider local perturbation functions 
\begin{align*}
f_n(x) \equiv \big(f_0(x)\wedge f_0(x_0-hr_n)\big)\bm{1}_{x\leq x_0} + f_0(x)\bm{1}_{x \nleq x_0}.
\end{align*}
Clearly $f_n \in \mathcal{F}_d$ by construction. Let $\mathcal{X}_n \equiv \{x \in \{X_i\}: f_n(x)\neq f_0(x)\}$. We claim that for $n$ large enough,
\begin{align}\label{ineq:X_n_range}
\mathcal{X}_n \subset [x_0-2d\pnorm{\bm{\alpha}}{\infty}\cdot hr_n, x_0].
\end{align}
To prove (\ref{ineq:X_n_range}), it suffices to show that for any $x \in \{X_i\}$, $x_0-\epsilon_0\bm{1}\leq x\leq x_0, x\ngeq x_0-2d\pnorm{\bm{\alpha}}{\infty}\cdot hr_n$, we have $f_0(x)<f_0(x_0-hr_n)$. To see this, first note that
\begin{align*}
f_0(x_0-hr_n)-f_0(x_0)&=-\omega_n  \sum_{k=\kappa_\ast}^s \frac{\partial_k^{\alpha_k} f_0(x_0)\big(1+\mathfrak{o}(1)\big)}{\alpha_k!}  h_k^{\alpha_k}\\
& \geq  -\omega_n \cdot s_\ast (\pnorm{\bm{\alpha}}{\infty}+1) \gamma (1+\mathfrak{o}(1)).
\end{align*}
On the other hand,
\begin{align*}
f_0(x) - f_0(x_0)&\leq -\omega_n \sum_{k=\kappa_\ast}^s \frac{\partial_k^{\alpha_k} f_0(x_0)\big(1+\mathfrak{o}(1)\big) }{\alpha_k!} (2d\pnorm{\bm{\alpha}}{\infty} h_k)^{\alpha_k} \bm{1}_{x_k\leq (x_0-2d \pnorm{\bm{\alpha}}{\infty}\cdot h r_n)_k}\\
&\qquad - \sum_{k=1}^{\kappa_\ast-1}  \frac{\partial_k^{\alpha_k} f_0(x_0)\big(1+\mathfrak{o}(1)\big) }{\alpha_k!} n^{-\alpha_k\beta_k} \bm{1}_{x_k\leq (x_0)_k-n^{-\beta_k}} \\
&\leq -\omega_n \cdot (2d\pnorm{\bm{\alpha}}{\infty}) \gamma (1+\mathfrak{o}(1)) \bm{1}_{\exists \kappa_\ast \leq k\leq s, x_k\leq (x_0-2d \pnorm{\bm{\alpha}}{\infty}\cdot h r_n)_k}\\
&\qquad - \mathcal{O}(n^\delta)\cdot  \omega_n \bm{1}_{\exists 1\leq k\leq \kappa_\ast-1, x_k\leq (x_0)_k - n^{-\beta_k}}.
\end{align*}
Here $\delta>0$ by Proposition \ref{prop:comparision_local_rates}. Comparing the above two displays proves the claim (\ref{ineq:X_n_range}). Therefore, with
\begin{align*}
\bar{f}_n(x)\equiv 
\begin{cases}
f_0(x_0-2d\pnorm{\bm{\alpha}}{\infty}\cdot hr_n) & \textrm{ if } x \in [x_0-2d\pnorm{\bm{\alpha}}{\infty}\cdot hr_n,x_0],\\
f_0(x) & \textrm{ otherwise},
\end{cases}
\end{align*}
we have for $n$ large enough, in the fixed lattice design case,
\begin{align*}
&n\ell_2^2(f_n,f_0)\leq n\ell_2^2(\bar{f}_n,f_0)\\
&=\sum_{X_i \in [x_0-2d\pnorm{\bm{\alpha}}{\infty}\cdot hr_n,x_0]} \bigg[\sum_{k=\kappa_\ast}^s\frac{\partial_k^{\alpha_k}f_0(x_0)}{\alpha_k!}  (1+\mathfrak{o}(1)) \big((X_i)_k-\big((x_0)_k-2d\pnorm{\bm{\alpha}}{\infty}\cdot hr_n\big) \big)^{\alpha_k}\bigg]^2\\
&\leq s_\ast \sum_{k=\kappa_\ast}^s \bigg(\frac{\partial_k^{\alpha_k} f_0(x_0)}{\alpha_k!} \bigg)^2 (1+\mathfrak{o}(1))\\
&\qquad\qquad\times \sum_{X_i \in [x_0-2d\pnorm{\bm{\alpha}}{\infty}\cdot hr_n,x_0]} \big((X_i)_k-\big((x_0)_k-2d\pnorm{\bm{\alpha}}{\infty}\cdot hr_n\big)\big)^{2\alpha_k}\\
& = s_\ast \sum_{k=\kappa_\ast}^s \bigg(\frac{\partial_k^{\alpha_k} f_0(x_0)}{\alpha_k!} \bigg)^2 (1+\mathfrak{o}(1))\\
&\qquad\qquad \times\prod_{\ell \neq k,\ell \geq \kappa_\ast} \big(2d \pnorm{\bm{\alpha}}{\infty}\cdot h_\ell (r_n)_{\ell} n^{\beta_\ell}\big) \sum_{1\leq m\leq 2d \pnorm{\bm{\alpha}}{\infty}\cdot h_k (r_n)_k n^{\beta_k}}\bigg(\frac{m}{n^{\beta_k}}\bigg)^{2\alpha_k}\\
& = s_\ast \sum_{k=\kappa_\ast}^s \bigg(\frac{\partial_k^{\alpha_k} f_0(x_0)}{\alpha_k!} \bigg)^2 (1+\mathfrak{o}(1))\\
&\qquad\qquad\times \frac{1}{2\alpha_k+1} \prod_{\ell \geq \kappa_\ast} (2d \pnorm{\bm{\alpha}}{\infty}\cdot h r_n)_\ell\cdot (2d \pnorm{\bm{\alpha}}{\infty}\cdot h r_n)_k^{2\alpha_k} n^{\sum_{k=\kappa_\ast}^d \beta_k}\\
& \leq s_\ast (2d \pnorm{\bm{\alpha}}{\infty})^{d+2\pnorm{\bm{\alpha}}{\infty}} (1+\mathfrak{o}(1)) \prod_{k=\kappa_\ast}^d h_k\cdot \sum_{k=\kappa_\ast}^s \bigg(\frac{\partial_k^{\alpha_k} f_0(x_0) }{(\alpha_k+1)!} h_k^{\alpha_k}\bigg)^2(\alpha_k+1)\\
& \leq \bigg(s_\ast (2d \pnorm{\bm{\alpha}}{\infty})^{d+2\pnorm{\bm{\alpha}}{\infty}} \sum_{k=\kappa_\ast}^s (\alpha_k+1)\bigg)(1+\mathfrak{o}(1)) \prod_{k=\kappa_\ast}^s \bigg(\frac{(\alpha_k+1)!}{\partial_k^{\alpha_k}f_0(x_0)}\bigg)^{1/\alpha_k}\cdot  \gamma^{2+\sum_{k=\kappa_\ast}^s \alpha_k^{-1}}\\
& = 2\sigma^2(1+\mathfrak{o}(1)).
\end{align*}
In the random design case, we can deduce the same inequality as above up to a constant factor depending on $P$. On the other hand, it follows from direct Taylor expansion that
\begin{align*}
&\abs{f_n(x_0)-f_0(x_0)}\geq \omega_n  s_\ast \gamma (1+\mathfrak{o}(1)) \\
& = \omega_n (1+\mathfrak{o}(1)) \bigg(\frac{2\sigma^2 (s_\ast)^{1+\sum_{k=\kappa_\ast}^s \alpha_k^{-1}}} {(2d \pnorm{\bm{\alpha}}{\infty})^{d+2\pnorm{\bm{\alpha}}{\infty}}\sum\limits_{\kappa_\ast\leq k\leq s}(\alpha_k+1)}\prod_{k=\kappa_\ast}^s \bigg(\frac{\partial_k^{\alpha_k} f_0(x_0)}{(\alpha_k+1)!}\bigg)^{1/\alpha_k}\bigg)^{\frac{1}{2+\sum_{k=\kappa_\ast}^s \alpha_k^{-1}}}.
\end{align*}
Now apply Proposition \ref{prop:minimax} to conclude the lower bound when $\kappa_\ast < s+1$. 

Use techniques similar to the proof of the inequality of Lemma \ref{lem:taylor_expansion}, we may control the cross terms with mixed derivatives in the Taylor expansion when such terms do not vanish.
\end{proof}

\section{Remaining proofs}\label{section:proof_lemmas}

\subsection{Proof of Lemma \ref{lem:mixed_derivative_vanish}}

	\begin{proof}[Proof of Lemma \ref{lem:mixed_derivative_vanish}]
		(1) If $\alpha_k=1$, then it is easy to see by monotonicity that the assumption requires $\partial_k^{\alpha_k}f_0(x_0)>0$. If $\alpha_k>1$, since for any $x \in [0,1]^d$ close enough to $x_0$ with $x_j = (x_0)_j, j\neq k$,
		\begin{align*}
		0\leq \partial_k^{1} f_0(x) = \frac{\partial_k^{\alpha_k}f_0(x_0)}{(\alpha_k-1)!} (x-x_0)_k^{\alpha_k-1}+\mathfrak{o}\big((x-x_0)_k^{\alpha_k-1}\big),
		\end{align*}
		it follows that $\alpha_k$ must be odd and $\partial_k^{\alpha_k} f_0(x_0)>0$. 
		
		\noindent (2) Let $J_0 \equiv \{\bm{j} \in J: 0<\sum_{k=1}^s j_k/\alpha_k<1\}$. Suppose $\partial^{\bm{j}}f_0(x_0)\neq 0$ for some $\bm{j} \in J_0$. We will prove that $\partial_k^{\alpha_k'} f_0(x_0)\neq 0$ for some $1\leq \alpha_k'<\alpha_k$ and $1\leq k\leq s$. Let $\omega_n \searrow 0$ and $r_n\equiv (\omega_n^{1/\alpha_1},\ldots,\omega_n^{1/\alpha_d})$, and $D_n \equiv \max_{\bm{0}\neq \bm{j} \in J, \partial^{\bm{j}} f_0(x_0)\neq 0} r_n^{\bm{j}}/\omega_n$. Since $r_n^{\bm{j}}/\omega_n = \omega_n^{\sum_{k=1}^s j_k/\alpha_k-1}$, we have $D_n \to \infty$. Let $t = (x-x_0)/r_n \in \R^d$ and
		\begin{align*}
		Q(t) = \lim_n \frac{f_0(x_0+tr_n)-f_0(x_0)}{\omega_n D_n} = \sum_{\bm{j} \in J_0} \bigg(\lim_n \frac{r_n^{\bm{j}}}{\omega_n D_n}\bigg) \frac{\partial^{\bm{j}} f_0(x_0) }{\bm{j}!} t^{\bm{j}}.
		\end{align*}
		Then $Q$ is a non-decreasing polynomial in the sense that if $t_1\leq t_2$, we have $Q(t_1)\leq Q(t_2)$. Note that in the above display the summation is over $J_0$ instead of $J$ since we assumed that $\partial^{\bm{j}} f_0(x_0)\neq 0$. Since $Q$ only depends on its first $s$ arguments, we slightly abuse the notation for $Q$ as a polynomial of $(t_1,\ldots,t_s)$ in the sequel. As $\partial_s^{\alpha_s'} f_0(x_0)=0$ for all $1\leq \alpha_s'<\alpha_s$, we have $Q(0,\ldots,0,t_s) = 0$. Assume now $Q(0,\ldots,0, t_{k+1},\ldots,t_s) = 0$ while $Q(0,\ldots,0,t_k,\ldots,t_s)\neq 0$. We may write
		\begin{align*}
		Q(0,\ldots,0, t_k,\ldots,t_s) = \sum_{\alpha_k'\leq j<\alpha_k} Q_j(t_{k+1},\ldots,t_s) t_k^j
		\end{align*}
		with some $\alpha_k'\geq 1$ and $Q_{\alpha_k'}(t_{k+1},\ldots,t_s)\neq 0$. By monotonicity of $Q$, $\alpha_k'$ is odd. Furthermore, for $t_k = \pm 1$ and any $k+1\leq \ell \leq s$, we have
		\begin{align*}
		0\leq \lim_n \omega_n^{-\alpha_k'} \partial_\ell Q(0,\ldots,0, \omega_n t_k,t_{k+1},\ldots,t_s) = t_k^{\alpha_k'} \partial_\ell Q_{\alpha_k'} (t_{k+1},\ldots,t_s).
		\end{align*}
		Hence $\partial_{k}^{\alpha_k'} f_0(x_0)/(\alpha_k')!=Q_{\alpha_k'}(0) = Q_{\alpha_k'}(t_{k+1},\ldots,t_s)\neq 0$. 
		
		The boundedness of $\omega_n^{-1}\abs{f_0(x)-f_0(x_0)}$ for $x \in [0,1]^d$ such that $\abs{(x-x_0)_k}\leq (r_n)_k, 1\leq k\leq d$, follows since for $\bm{j} \in J_\ast$, $\omega_n^{-1}\abs{(x-x_0)^{\bm{j}}} \leq \omega_n^{-1} r_n^{\bm{j}} = 1$.
		
		\noindent (3) Suppose $\mathrm{gcd}(\alpha_{k_1},\alpha_{k_2})=1$ for all $1\leq k_1<k_2\leq s$. Then for any $\bm{j}$, $\sum_{k=1}^d j_k/\alpha_k=1$ and $j_k>0$, we have $\alpha_k|j_k$ and $j_\ell =0$ for all $\ell \neq k$. This means $J_1=\emptyset$. On the other hand, if $\mathrm{gcd}(\alpha_{1},\alpha_{2})>1$, write $\alpha_{1} = m_1 m, \alpha_{2} = m_2 m$ for some positive odd numbers $m_1,m_2,m$ with $m\geq 3$. Then with $j_{1}\equiv m_1 p, j_2 \equiv m_2 q$ where $p,q\geq 1, p+q =m$ (such $p,q$ exist as $m\geq 3$), we have $j_1/\alpha_1+j_2/\alpha_2 = 1$, so $J_1 \neq \emptyset$.
		
		\noindent (4) Without loss of generality suppose $\alpha_{1} = m_1 m, \alpha_{2} = m_2 m$ for some positive odd numbers $m_1,m_2,m$ with $m\geq 3$ and $d=2$. Consider
		\begin{align*}
		f(x_0+t) \equiv (t_1^{m_1})^m+ (t_1^{m_1})^{m-1} (t_2^{m_2})+ (t_1^{m_1})(t_2^{m_2})^{m-1}+ (t_2^{m_2})^m.
		\end{align*}
		Let $y_i \equiv t_i^{m_i}, i=1,2$ and $g(y) \equiv y_1^m + y_1^{m-1} y_2+y_1y_2^{m-1}+y_2^{m}$. Then $\partial_1 g(y) = m y_1^{m-1} +(m-1) y_1^{m-2} y_2+ y_2^{m-1}\geq 0$, as $\abs{y_1^{m-2} y_2}\leq \frac{m-2}{m-1} y_1^{m-1}+ \frac{1}{m-1} y_2^{m-1}$. So with $x\equiv x_0+t$, $\partial_1 f(x) = (\partial/\partial t_1) f(x_0+t) = (\partial g/\partial y_1)(\partial y_1/\partial t_1) \geq 0$. Similarly $\partial_2 f(x)\geq 0$, so $f \in \mathcal{F}_2$. It is easy to check that $f$ satisfies Assumption \ref{assump:smoothness} with $\bm{\alpha} = (\alpha_1,\alpha_2) = (m_1m,m_2m)$, but $\partial^{m_1(m-1),m_2} f(x_0), \partial^{m_1,m_2(m-1)} f(x_0)\neq 0$.
	\end{proof}

\subsection{Proof of Lemma \ref{lem:size_max_partial_sum}}

\begin{proof}[Proof of Lemma \ref{lem:size_max_partial_sum}]
	Without loss of generality, we assume $\tau = 1$. Let $x_n \equiv x_0+\tau r_n$. We only prove the lemma for the random design case; the fixed lattice design case follows from simpler arguments. Let $s=d$ for simplicity. Let $\mathcal{E}_n$ be the event specified by Lemma \ref{lem:ratio_indicator} with $g\equiv 1$ therein. Then,
	\begin{align*}
	&\E \sup_{h>0} \abs{\bar{\xi}|_{[x_0-h r_n,x_0+\tau r_n]}}\bm{1}_{\mathcal{E}_n} = \E \sup_{h> 1} \abs{\bar{\xi}|_{[x_n-h r_n,x_n]}}\bm{1}_{\mathcal{E}_n}\\
	&\leq \sum_{\ell_k\geq 0, 1\leq k\leq d}\E \sup_{2^{\ell_k}\leq h_k\leq 2^{\ell_k+1}} \frac{\abs{\sum_{i=1}^n \xi_i \bm{1}_{X_i \in [x_n-hr_n,x_n]} }}{n\Prob_n \bm{1}_{X \in [x_n-hr_n,x_n]}  }\bm{1}_{\mathcal{E}_n}\\
	&\leq  \sum_{\ell_k\geq 0, 1\leq k\leq d}\frac{2}{nP\bm{1}_{X \in [x_n-2^{\bm{\ell}}r_n,x_n]} } \E \sup_{2^{\ell_k}\leq h_k\leq 2^{\ell_k+1}} \biggabs{\sum_{i=1}^n \xi_i \bm{1}_{X_i \in [x_n-hr_n,x_n]} }\\
	&\leq C_P\cdot n^{-1}\prod_{k=1}^d (r_n)_k^{-1}\cdot  \sum_{\ell_k\geq 0, 1\leq k\leq d} 2^{-\sum_{k=1}^d \ell_k } \E \sup_{2^{\ell_k}\leq h_k\leq 2^{\ell_k+1}} \biggabs{\sum_{i=1}^n \xi_i \bm{1}_{X_i \in [x_n-hr_n,x_n]} }\\
	&\leq C_{P,d} \cdot n^{-1}\prod_{k=1}^d (r_n)_k^{-1}  \sum_{\ell_k\geq 0, 1\leq k\leq d} 2^{-\sum_{k=1}^d \ell_k } \bigg(\E \biggabs{\sum_{i=1}^n \xi_i \bm{1}_{X_i \in [x_n-2^{\bm{\ell}+1}r_n,x_n]} }^2\bigg)^{1/2}\\
	& \qquad\qquad\qquad (\textrm{by Lemma 2 of \cite{deng2018isotonic}} )\\
	&\leq C_{P,d}'\cdot  \sigma \cdot n^{-1/2}\prod_{k=1}^d (r_n)_k^{-1/2} \sum_{\ell_k\geq 0, 1\leq k\leq d} 2^{-\frac{1}{2}\sum_{k=1}^d \ell_k } \lesssim_{P,d} \sigma \cdot n^{-\frac{1}{2+\sum_{k=1}^{d}\alpha_k^{-1} }},
	\end{align*}
	as desired.
\end{proof}

\subsection{Proof of Lemma \ref{lem:weak_conv_gaussian_proc}}

We need the following to prove Lemma \ref{lem:weak_conv_gaussian_proc}.

\begin{lemma}[\cite{alexander1987central}; see also Theorem 2.11.9 of \cite{van1996weak}]\label{lem:clt_partial_sum_process}
	For each $n$, let $Z_{n1},\ldots,Z_{n m_n}$ be independent stochastic processes indexed by a totally bounded semi-metric space $(\mathcal{F},\rho)$. Suppose that
	\begin{enumerate}
		\item $\sum_{i=1}^{m_n} \E \pnorm{Z_{ni}}{\mathcal{F}}^2 \bm{1}_{\pnorm{Z_{ni}}{\mathcal{F}}>\eta }\to 0$ for every $\eta>0$.
		\item $\sup_{\rho(f,g)<\delta_n} \sum_{i=1}^{m_n} \E\big(Z_{ni}(f)-Z_{ni}(g)\big)^2 \to 0$ for every $\delta_n\to 0$.
		\item $\int_0^{\delta_n} \sqrt{\log \mathcal{N}_{[\,]}(\epsilon, \mathcal{F}, L_2^n)}\ \d{\epsilon}\to 0$ for every $\delta_n\to 0$. Here the bracketing number $\mathcal{N}_{[\,]}(\epsilon, \mathcal{F}, L_2^n)$ is defined as the minimal number of sets $N_\epsilon$ in a partition $\mathcal{F}\equiv \cup_{j=1}^{N_\epsilon} \mathcal{F}_{\epsilon j}^n$ of the index set into sets $\mathcal{F}_{\epsilon j}^n$ such that for every partitioning set $\mathcal{F}_{\epsilon j}^n$,
		\begin{align*}
		\sum_{i=1}^{m_n} \E \sup_{f,g \in \mathcal{F}_{\epsilon j}^n} \abs{Z_{ni}(f)-Z_{ni}(g)}^2\leq \epsilon^2.
		\end{align*}
	\end{enumerate}
	Then the sequence $\sum_{i=1}^{m_n} (Z_{ni}-\E Z_{ni})$ is asymptotically $\rho$-equicontinuous.
\end{lemma}

\begin{proof}[Proof of Lemma \ref{lem:weak_conv_gaussian_proc}]
	We will prove a slightly stronger statement by strengthening the convergence in $\ell^\infty\big([c^{-\gamma_\ast}\bm{1},c\bm{1}]\times [c^{-\gamma_\ast}\bm{1},c\bm{1}]\big)$ to convergence in $\ell^\infty\big([0,c\bm{1}]\times [0,c\bm{1}]\big)$. We first verify the finite-dimensional convergence. This is the easy step. For fixed lattice design, for any $(h_1,h_2), (h_1',h_2') \in \R_{\geq 0}^{d}\times \R_{\geq 0}^{d}$, we only need to note that
	\begin{align*}
	&\mathrm{Cov}\big(\G_n(h_1,h_2),\G_n(h_1',h_2')\big)\\
	&=\omega_n^2 \sum_{i,j} \mathrm{Cov}\big(\xi_i \bm{1}_{X_i \in [x_0-h_1 r_n,x_0+h_2 r_n]},\xi_j \bm{1}_{X_j \in [x_0-h_1' r_n,x_0+h_2' r_n]}\big)\\
	&=\omega_n^2 \sum_{X_i \in [x_0-(h_1\wedge h_1')r_n,x_0+(h_2\wedge h_2')r_n]} \E \xi_i^2 \\
	&= \pnorm{\xi_1}{2}^2 (1+\mathfrak{o}(1)) \prod_{k=\kappa_\ast}^d \left( (h_1)_k\wedge (h_1')_k+ (h_2)_k\wedge (h_2')_k\right).
	\end{align*}
	For random design, we may proceed with the above lines conditionally.
	
	Next we verify the asymptotic equicontinuity by Lemma \ref{lem:clt_partial_sum_process}. Let
	\begin{align*}
	Z_{ni}(h_1,h_2)\equiv \omega_n  \xi_{i}\bm{1}_{i: x_0-h_1 r_n\leq X_i\leq x_0+h_2 r_n}.
	\end{align*}
	Let $m_n\equiv n$, $\mathcal{F}\equiv \{(h_1,h_2) \in [0,c\bm{1}]\times [0,c\bm{1}]\subset \R^{d}_{\geq 0}\times \R^{d}_{\geq 0}\}$. Let $\rho$ be defined as follows: for any $(h_1,h_2),(h_1',h_2') \in [0,c\bm{1}]\times [0,c\bm{1}]$,
	\begin{align*}
	\rho\big((h_1,h_2),(h_1',h_2') \big)\equiv \pnorm{ \pi_{\kappa_\ast}(h_1-h_1')}{}+\pnorm{\pi_{\kappa_\ast}(h_2-h_2')}{}
	\end{align*}
	where $\pnorm{\cdot}{}$ is the standard Euclidean norm, and $\pi_{\kappa_\ast}: \R^d \to \R^{d_\ast}$ is the canonical projection onto last $d_\ast$ coordinates. We now verify conditions (1)-(3) of Lemma \ref{lem:clt_partial_sum_process} as follows.
	
	For condition (1), note that $
	\pnorm{Z_{ni}}{\mathcal{F}}= \omega_n \abs{\xi_{i}}\bm{1}_{i: x_0-c r_n\leq X_i\leq x_0+c r_n}$. Hence for any $\eta>0$, in the fixed lattice design case,
	\begin{align*}
	\sum_{i=1}^{m_n} \E \pnorm{Z_{ni}}{\mathcal{F}}^2 \bm{1}_{\pnorm{Z_{ni}}{\mathcal{F}}>\eta }& \leq \sum_{i=1}^n \omega_n^2  \E \abs{\xi_i}^2 \bm{1}_{i: x_0-c r_n\leq X_i\leq x_0+c r_n} \cdot \bm{1}_{ \abs{\xi_i}\omega_n>\eta}\\
	& = \omega_n^2  \E \abs{\xi_1}^2 \bigg(\sum_{i=1}^n \bm{1}_{i: x_0-c r_n\leq X_i\leq x_0+c r_n}\bigg)\cdot \bm{1}_{ \abs{\xi_1}\omega_n >\eta}\\
	& \leq \omega_n^2 \prod_{k=\kappa_\ast}^d (2c (r_n)_k n^{\beta_k}) \cdot \E \abs{\xi_1}^2 \bm{1}_{ \abs{\xi_1}\omega_n >\eta}  \\
	&= (2c)^{d_\ast}  \E \abs{\xi_1}^2 \bm{1}_{ \abs{\xi_1}\omega_n>\eta}\to 0
	\end{align*}
	as $n \to \infty$. The random design case follows from obvious modifications using conditioning arguments and the fact that the density $\pi$ is bounded away from $\infty$.

	For condition (2), for any $(h_1,h_2), (h_1',h_2')$ such that $\rho\big((h_1,h_2),(h_1',h_2') \big)< \delta_n$, in the fixed lattice design case
	\begin{align*}
	&\sum_{i=1}^{m_n} \E\big(Z_{ni}(h_1,h_2)-Z_{ni}(h_1',h_2')\big)^2\\
	&\leq \omega_n^2 \E \abs{\xi_1}^2 \bigg(\sum_{i=1}^n \bm{1}_{X_i \in [x_0-h_1 r_n, x_0+h_2r_n ]\Delta [x_0-h_1' r_n, x_0+h_2' r_n ]} \bigg)\\
	&\leq \omega_n^2  \E \abs{\xi_1}^2  \bigg[ \prod_{k=\kappa_\ast}^d\big(( h_1+h_2)_k (r_n)_k n^{\beta_k}\big)- \prod_{k=\kappa_\ast}^d\big(( h_1\wedge h_1'+h_2\wedge h_2')_k (r_n)_k n^{\beta_k}\big)\\
	&\qquad\qquad+ \prod_{k=\kappa_\ast}^d\big((h_1'+h_2')_k (r_n)_k n^{\beta_k}\big)-\prod_{k=\kappa_\ast}^d\big(( h_1\wedge h_1'+h_2\wedge h_2')_k (r_n)_k n^{\beta_k}\big) \bigg]\\
	&\leq 2\delta_n \cdot (2c)^{d_\ast}\sqrt{d_\ast}\cdot  \E\abs{\xi_1}^2 \to 0.
	\end{align*}
	The last inequality follows by Lemma \ref{lem:product_difference}. The random design case follows similarly. This verifies (2). 
	
	For (3), we can identify $(h_1,h_2) \in \mathcal{F}$ as the indicator $\bm{1}_{[x_0-h_1 r_n,x_0+h_2 r_n]}\equiv \tau(h_1,h_2)$, where $h_i \in [0,c\bm{1}] (i=1,2)$. For each $\epsilon>0$, let $h_\epsilon\equiv \epsilon^2/\big(\E \abs{\xi_1}^2 d_\ast(2c)^{d_\ast} 2\big)$. Then
	\begin{align*}
	\mathcal{F}&\subset \bigcup_{h_1,h_2 \in \{0,h_\epsilon,\ldots, c\}^d } \{\tau(h_1',h_2'): h_i\leq h_i'\leq h_i+h_\epsilon , i=1,2 \}
	\equiv \bigcup_{j=1}^{N_\epsilon} \mathcal{F}^{n}_{\epsilon j},
	\end{align*}
	where $N_\epsilon \leq (c/h_\epsilon)^{2d}$.
	
	Now we show that this is a valid $\epsilon$-bracketing in the sense defined in Lemma \ref{lem:clt_partial_sum_process}. For $j=1,\ldots,N_\epsilon$, $\mathcal{F}_{\epsilon j}^{n}=\{\tau(h_1',h_2'): h_i\leq h_i'\leq h_i+h_\epsilon, i=1,2 \}$ for some $ h_1,h_2\in \{0,h_\epsilon,\ldots, c\}^d $. Then, in the fixed lattice design case,
	\begin{align*}
	&\sum_{i=1}^{m_n} \E \sup_{f,g \in \mathcal{F}_{\epsilon j}^{n}} \abs{Z_{ni}(f)-Z_{ni}(g)}^2\\
	& \leq  \omega_n^2  \E \abs{\xi_1}^2 \bigg(\sum_{i=1}^n\bm{1}_{X_i \in [x_0-h_1 r_n, x_0+h_2r_n ] \Delta [x_0-(h_1-h_\epsilon) r_n, x_0+(h_2-h_\epsilon)r_n ]} \bigg)\\
	&= \omega_n^2  \E \abs{\xi_1}^2 \bigg[\prod_{k=\kappa_\ast}^d\big(( h_1+h_2)_k (r_n)_k n^{\beta_k}\big)-\prod_{k=\kappa_\ast}^d\big((h_1+h_2-2h_\epsilon)_k (r_n)_k n^{\beta_k}\big) \bigg]\\
	&\leq \E \abs{\xi_1}^2 d_\ast (2c)^{d_\ast} (2h_\epsilon)=\epsilon^2.
	\end{align*}
	The random design case follows similarly (with a slight modification of the definition of $h_\epsilon$). On the other hand,
	\begin{align*}
	\int_0^{\delta_n} \sqrt{\log \mathcal{N}_{[\,]}(\epsilon, \mathcal{F}, L_2^n)}\ \d{\epsilon} &\leq \int_0^{\delta_n} \sqrt{2d \log (c/h_\epsilon)}\ \d{\epsilon}\lesssim \int_0^{\delta_n} \sqrt{\log (C_{\xi,d,c}/\epsilon)}\ \d{\epsilon}\to 0
	\end{align*}
	as $\delta_n \to 0$. This verifies (3).
	
	Now we may apply Lemma \ref{lem:clt_partial_sum_process} to conclude the claim of the lemma.
\end{proof}

\subsection{Proofs of Lemmas \ref{lem:finite_expectation_W} and \ref{lem:brownian_motion_uniform}}

We need Dudley's entropy integral bound for sub-Gaussian processes, recorded below for the convenience of the reader.

\begin{lemma}[Theorem 2.3.7 of \cite{gine2015mathematical}]\label{lem:dudley_entropy_integral}
	Let $(T,d)$ be a pseudo metric space, and $(X_t)_{t \in T}$ be a sub-Gaussian process such that $X_{t_0}=0$ for some $t_0 \in T$. Then
	\begin{align*}
	\E \sup_{t \in T} \abs{X_t} \leq C\int_0^{\mathrm{diam}(T)} \sqrt{\log \mathcal{N}(\epsilon,T,d)}\ \d{\epsilon}.
	\end{align*}
	Here $C>0$ is a universal constant.
\end{lemma}

The following Gaussian concentration inequality will also be useful.
\begin{lemma}[Borell's Gaussian concentration inequality]\label{lem:Gaussian_concentration}
	Let $(T,d)$ be a pseudo metric space, and $(X_t)_{t \in T}$ be a mean-zero Gaussian process. Then with $\sigma^2\equiv \sup_{t \in T} \mathrm{Var}(X_t)$, for any $u >0$,
	\begin{align*}
	\Prob\big( \bigabs{\sup_{t \in T} \abs{X_t} - \E \sup_{t \in T} \abs{X_t}}>u\big)\le 2\exp\big(-u^2/2\sigma^2\big).
	\end{align*}
\end{lemma}

\begin{proof}[Proof of Lemma \ref{lem:brownian_motion_uniform}]
	Without loss of generality we assume $s=d$. Let $G$ be the Gaussian process defined on $\R_{\geq 0}^d$ through $G(h)\equiv \G(h,\bm{1})$, whose natural induced metric on $\R_{\geq 0}^d$ is given by $
	d_G^2(h_1,h_2)\equiv \E(G(h_1)-G(h_2))^2= \lambda_{d_\ast}\big([-\bm{1}_{[\kappa_\ast:d]},h_1 \bm{1}_{[\kappa_\ast:d]}]\Delta [-\bm{1}_{[\kappa_\ast:d]},h_2\bm{1}_{[\kappa_\ast:d]}]\big)$, where $\lambda_{d_\ast}$ is the Lebesgue measure on $\R^{d_\ast}$. Then for $u\geq 0$,
	\begin{align}\label{ineq:finte_expectation_W_2}
	&\Prob\bigg(\max_{h_1>0}\frac{ \abs{\G(h_1,\bm{1})}}{\prod_{k=\kappa_\ast}^d \big((h_1)_k+1\big)}>u\bigg)\\
	&\leq \sum_{\ell_k\geq 0,\kappa_\ast\leq k\leq d} \Prob \bigg(\sup_{ \substack{2^{\ell_k}-1\leq h_k\leq 2^{\ell_k+1}-1,\\ \kappa_\ast \leq k\leq d}} \frac{\abs{G(h)}}{ \prod_{k=\kappa_\ast}^d \big(h_k+1\big) }>u\bigg)\nonumber\\
	&\leq \sum_{\ell_k\geq 0,\kappa_\ast\leq k\leq d} \Prob\bigg(\sup_{ \substack{2^{\ell_k}-1\leq h_k\leq 2^{\ell_k+1}-1,\\ \kappa_\ast \leq k\leq d}} \abs{G(h)}> \prod_{k=\kappa_\ast}^d 2^{\ell_k} \cdot u\bigg)\nonumber\\
	&\leq \sum_{\ell_k \geq 1, \kappa_\ast\leq k \leq d} \Prob\bigg(\sup_{ \substack{0\leq h_k\leq 2^{\ell_k},\\ \kappa_\ast \leq k\leq d}} \abs{G(h)}> 2^{-d_\ast}\prod_{k=\kappa_\ast}^d 2^{\ell_k} \cdot u\bigg).\nonumber
	\end{align}
	Since for $h_i$ such that $0\leq (h_i)_k\leq 2^{\ell_k}$, $i=1,2$, Lemma \ref{lem:product_difference} entails that $
	d_G^2(h_1,h_2)\leq c_{d_\ast} \prod_{k=\kappa_\ast}^d 2^{\ell_k} \cdot \pnorm{h_1-h_2}{}$, it follows by Dudley's entropy integral (cf. Lemma \ref{lem:dudley_entropy_integral}) that
	\begin{align*}
	&\E \sup_{ \substack{0\leq h_k\leq 2^{\ell_k},\\ \kappa_\ast \leq k\leq d}} \abs{G(h)}\\
	&\leq C \int_0^{C(\prod_{k=\kappa_\ast}^d 2^{\ell_k})^{1/2} }\sqrt{\log \mathcal{N}\bigg(\epsilon, \prod_{k=\kappa_\ast}^d [0,2^{\ell_k}], d_G\bigg)}\ \d{\epsilon}\\
	&\leq C \int_0^{C(\prod_{k=\kappa_\ast}^d 2^{\ell_k})^{1/2} }\sqrt{\log \mathcal{N}\bigg(\frac{\epsilon^2}{c_{d_\ast} \prod_{k=\kappa_\ast}^d 2^{\ell_k} }, \prod_{k=\kappa_\ast}^d [0,2^{\ell_k}], \pnorm{\cdot}{}\bigg)}\ \d{\epsilon}\\
	&\leq C \int_0^{C (\prod_{k=\kappa_\ast}^d 2^{\ell_k})^{1/2} }\sqrt{\log \bigg(\frac{\prod_{k=\kappa_\ast}^d 2^{\ell_k}}{ \big(\epsilon^2/ c_{d_\ast}\prod_{k=\kappa_\ast}^d 2^{\ell_k} \big)^{d_\ast}}\bigg)}\ \d{\epsilon}
	\leq C_{d_\ast} L\bigg(\prod_{k=\kappa_\ast}^d 2^{\ell_k}\bigg),
	\end{align*}
	where $L(x) = \sqrt{x \log x}$. This implies that for $u$ larger than a constant only depending on $d_\ast$,
	\begin{align}\label{ineq:finte_expectation_W_3}
	& \Prob\bigg(\sup_{ \substack{0\leq h_k\leq 2^{\ell_k},\\\kappa_\ast \leq k\leq d} } \abs{G(h)}> 2^{-d_\ast}\prod_{k=\kappa_\ast}^d 2^{\ell_k} \cdot u\bigg)\\
	&\leq \Prob\bigg(\sup_{ \substack{0\leq h_k\leq 2^{\ell_k},\\\kappa_\ast \leq k\leq d}} \abs{G(h)} -\E\sup_{ \substack{0\leq h_k\leq 2^{\ell_k},\\\kappa_\ast \leq k\leq d}} \abs{G(h)}> 2^{-d_\ast-1}\prod_{k=\kappa_\ast}^d 2^{\ell_k} \cdot u\bigg)\nonumber\\
	&\leq 2 e^{-c_{d_\ast}'\cdot \prod_{k=\kappa_\ast}^d 2^{\ell_k} u^2},\nonumber
	\end{align}
	where in the last line we used Gaussian concentration inequality (cf. Lemma \ref{lem:Gaussian_concentration}) and the fact that $\sup_{0\leq h_k\leq 2^{\ell_k},\kappa_\ast \leq k\leq d}\mathrm{Var}(G(h))\leq 2^{d_\ast} \prod_{k=\kappa_\ast}^d 2^{\ell_k}$. 
	Combining (\ref{ineq:finte_expectation_W_2}) and (\ref{ineq:finte_expectation_W_3}) we see that for $u>0$ large enough,
	\begin{align*}
	&\Prob\bigg(\max_{h_1>0}\frac{ \abs{\G(h_1,\bm{1})} }{\prod_{k=\kappa_\ast}^d \big((h_1)_k+1\big)}>u\bigg)\leq 2\sum_{\ell_k \geq 1, \kappa_\ast\leq k \leq d} e^{-c_{d_\ast}'\cdot \prod_{k=\kappa_\ast}^d 2^{\ell_k} u^2}\leq C_{d_\ast}' e^{-u^2/C_{d_\ast}'}. \nonumber
	\end{align*}
	This completes the proof.
\end{proof}

\begin{proof}[Proof of Lemma \ref{lem:finite_expectation_W}]
	For simplicity of notation we assume that $\sigma=1$. Then
	\begin{align}\label{ineq:finte_expectation_W_1}
	W & \leq \max_{h_1>0} \mathbb{U}(h_1,\bm{1})\leq \max_{h_1>0}\frac{ \abs{\G(h_1,\bm{1})}}{\prod_{k=\kappa_\ast}^d \big((h_1)_k+1\big)}+ \mathcal{O}(1).
	\end{align}
	Now by Lemma \ref{lem:brownian_motion_uniform} and (\ref{ineq:finte_expectation_W_1}), we see that $\pnorm{(W)_+}{\psi_2}<\infty$. Similarly we can establish $\pnorm{(W)_-}{\psi_2}<\infty$.
\end{proof}

\section{Auxiliary lemmas}\label{section:proof_auxiliary}

\begin{lemma}\label{lem:taylor_expansion}
	Suppose Assumption \ref{assump:smoothness} holds. For $\omega_n \searrow 0$, let $r_n \equiv (\omega_n^{1/\alpha_1},\ldots,\omega_n^{1/\alpha_d})$. 
	\begin{itemize}
		\item(Fixed design) For $h_1,h_2\geq 0$ such that $\max\limits_{i=1,2}\max\limits_{1\leq k\leq d} (h_i)_k (r_n)_k \leq \epsilon_0$,
			\begin{align*}
		&\bar{f_0}|_{[x_0-h_1 r_n,x_0+h_2 r_n]}-f_0(x_0) \\
		& = \mathfrak{o}(\tilde{\omega}_n)+ \sum_{ \bm{j} \in J_\ast} \frac{ \partial^{\bm{j}} f_0(x_0)}{\bm{j}!} \prod_{k=1}^s  n^{-\beta_k j_k} A_{j_k}\big((h_1r_n)_k n^{\beta_k},(h_2r_n)_k n^{\beta_k} \big) \\
		& = \mathcal{O}(\tilde{\omega}_n).
		\end{align*}
		Here $A_q(t_1,t_2)\equiv \sum_{k=-\floor{t_1}}^{\floor{t_2}} k^q/(\floor{t_1}+\floor{t_2}+1)$ for $t_1,t_2\geq 0$, and $\tilde{\omega}_n \equiv \omega_n\cdot \max\limits_{1\leq k\leq s}\big(1\vee (h_1)_k^{\alpha_k}\vee (h_2)_k^{\alpha_k}\big)\equiv \omega_n\cdot \big(1\vee (h_1)_{k_0}^{\alpha_{k_0}}\vee (h_2)_{k_0}^{\alpha_{k_0}}\big)$. Moreover, for $c>0$ large enough, if $ (h_2)_{k_0} \geq \max\{c, c \max_{1\leq k\leq s} \big(1\vee (h_1)_k), n^{-\beta_{k_0}} (r_n)_{k_0}^{-1}\}$, then
		\begin{align*}
		\bar{f_0}|_{[x_0-h_1 r_n,x_0+h_2 r_n]}-f_0(x_0)\geq \tilde{\omega}_n(C_{\bm{\alpha}}-\mathcal{O}(c^{-1})).
		\end{align*}
		\item(Random design) Suppose that $\omega_n^{1+\sum_k \alpha_k^{-1}}\gg \log n/n$, and that $P$ is the uniform distribution on $[0,1]^d$. Then for any $c_0\geq 1$, uniformly for $h_1,h_2\geq 0$ such that $\max\limits_{i=1,2}\max\limits_{1\leq k\leq d} (h_i)_k (r_n)_k \leq \epsilon_0$ and that $c_0^{-1}\bm{1}\leq h_1\leq c_0\bm{1}$, it holds with probability at least $1-\mathcal{O}(n^{-2})$ that,
		\begin{align*}
		&\bar{f_0}|_{[x_0-h_1 r_n,x_0+h_2 r_n]}-f_0(x_0) \\
		& = \mathfrak{o}(\tilde{\omega}_n)+ \sum_{ \bm{j} \in J_\ast} \frac{ \partial^{\bm{j}} f_0(x_0)}{\bm{j}!} \prod_{k=1}^s  n^{-\beta_k j_k} I_{j_k}\big((h_1r_n)_k n^{\beta_k},(h_2r_n)_k n^{\beta_k} \big) \\
		&= \mathcal{O}(\tilde{\omega}_n).
		\end{align*}
		Here $I_q(t_1,t_2)\equiv \int_{-t_1}^{t_2} x^q\ \d{x}/(t_1+t_2)$ for $t_1,t_2\geq 0$, and $\tilde{\omega}_n \equiv \omega_n\cdot \max_{1\leq k\leq s} \big(1\vee (h_2)_k^{\alpha_k}\big)\equiv \omega_n\cdot \big(1\vee (h_2)_{k_0}^{\alpha_{k_0}} \big)$. Moreover, for $c>0$ large enough (possibly depending on $c_0$), if $ (h_2)_{k_0} \geq c$, then
		\begin{align*}
		\bar{f_0}|_{[x_0-h_1 r_n,x_0+h_2 r_n]}-f_0(x_0)\geq \tilde{\omega}_n(C_{\bm{\alpha}}-\mathcal{O}(c^{-1})).
		\end{align*}
	\end{itemize}
	In both cases, $C_{\bm{\alpha}}>0$ is a constant only depending on $\bm{\alpha}$.
\end{lemma}
\begin{proof}
	First consider fixed lattice design case. By Assumption \ref{assump:smoothness},
	\begin{align*}
	&\bar{f_0}|_{[x_0-h_1 r_n,x_0+h_2 r_n]}-f_0(x_0) \\
	&=   \mathfrak{o}(\tilde{\omega}_n)+\bigg(\prod_{k=1}^d \left( \floor{(h_1r_n)_k n^{\beta_k}}+\floor{(h_2r_n)_k n^{\beta_k}}+1\right)\bigg)^{-1}\\
	& \qquad\qquad\qquad\qquad \times  \sum_{ \bm{j} \in J_\ast} \frac{ \partial^{\bm{j}} f_0(x_0)}{\bm{j}!} \sum_{X_i \in [x_0-h_1 r_n, x_0+h_2 r_n]} (X_i-x_0)^{\bm{j}}\\
	& =  \mathfrak{o}(\tilde{\omega}_n)+\bigg(\prod_{k=1}^d \left( \floor{(h_1r_n)_k n^{\beta_k}}+\floor{(h_2r_n)_k n^{\beta_k}}+1\right)\bigg)^{-1}\\
	& \qquad\qquad\qquad\qquad \times  \sum_{ \bm{j} \in J_\ast} \frac{ \partial^{\bm{j}} f_0(x_0)}{\bm{j}!} \prod_{k=1}^s \sum_{-(h_1r_n)_k  n^{\beta_k}\leq m\leq (h_2r_n)_k n^{\beta_k}} \bigg(\frac{m}{n^{\beta_k}}\bigg)^{j_k}\\
	& = \mathfrak{o}(\tilde{\omega}_n)+ \sum_{ \bm{j} \in J_\ast} \frac{ \partial^{\bm{j}} f_0(x_0)}{\bm{j}!} \prod_{k=1}^s  n^{-\beta_k j_k} A_{j_k}\big((h_1r_n)_k n^{\beta_k},(h_2r_n)_k n^{\beta_k} \big).
	\end{align*}
	The whole term above is of order $\mathcal{O}(\tilde{\omega}_n)$ since
	\begin{align*}
	&\prod_{k=1}^s  n^{-\beta_k j_k} A_{j_k}\big((h_1r_n)_k n^{\beta_k},(h_2r_n)_k n^{\beta_k} \big)\\
	&\lesssim \prod_{k=1}^s  n^{-\beta_k j_k} \big(((h_1\vee h_2)r_n)_k n^{\beta_k} \big)^{j_k} = \omega_n (h_1\vee h_2)^{\bm{j}} \lesssim \tilde{\omega}_n,
	\end{align*}
	where the last inequality follows from the fact that for any $\bm{j} \in J_\ast$, and $h\geq 0$, $h^{\bm{j}}\leq \sum_k (j_k/\alpha_k) h_k^{\alpha_k}$. For the inequality in the statement of the lemma, by monotonicity of $\bar{f_0}|_{[x_0-h_1 r_n,x_0+h_2 r_n]}-f_0(x_0) $, we may assume $(h_2)_k = 0$ for $k\neq k_0$. Then for any $\bm{j} \in J_\ast$, if $(h_2)_{k_0}>c>1$,
	\begin{align*}
	\prod_{k=1}^s  n^{-\beta_k j_k} \big(((h_1\vee h_2)r_n)_k n^{\beta_k} \big)^{j_k}  &= \omega_n \prod_{k=1}^s (h_1\vee h_2)_k^{j_k}\\
	& \leq \tilde{\omega}_n\bigg[\prod_{k=1}^s \bigg\{1 \vee \frac{(h_2)_k^{j_k}}{\max\limits_{1\leq k\leq s}(1\vee (h_1)_k^{\alpha_k}\vee (h_2)_{k}^{\alpha_{k}})}\bigg\} \bigg] \\
	&\lesssim \tilde{\omega}_n (h_2)_{k_0}^{j_{k_0}-\alpha_{k_0}}\lesssim \tilde{\omega}_n \cdot c^{j_{k_0}-\alpha_{k_0}}.
	\end{align*}
	For $j_{k_0} = \alpha_{k_0}$, if $(h_2 r_n)_{k_0} n^{\beta_{k_0}}\geq 1, (h_2)_{k_0}> c\max_{1\leq k\leq s} (h_1)_k$ for a large enough constant $c>0$, using $A_q(t_1,t_2)\geq \int_0^{t_2} x^q\ \d{x}/(t_1+t_2+1)-(t_1+1)^{q}$ for $t_2\geq 1, t_1\geq 0$ and odd positive integer $q$, we have
	\begin{align*}
	&n^{-\beta_{k_0}\alpha_{k_0}} A_{\alpha_{k_0}}\big((h_1r_n)_{k_0} n^{\beta_{k_0}},(h_2r_n)_{k_0} n^{\beta_{k_0}} \big)\\
	&\geq n^{-\beta_{k_0}\alpha_{k_0}}\left(C_{\alpha_{k_0}}\big((h_2r_n)_{k_0} n^{\beta_{k_0}}\big)^{\alpha_{k_0}}- \big(1+(h_1r_n)_{k_0} n^{\beta_{k_0}}\big)^{\alpha_{k_0}}\right)\\
	&\gtrsim \omega_n h_2^{\alpha_{k_0}} \asymp \tilde{\omega}_n.
	\end{align*}
	Collecting the bounds we have proved the inequality.
	
	Next for random design case, using Lemma \ref{lem:ratio_random_replace}, with the desired probability we have
	\begin{align*}
	&\bar{f_0}|_{[x_0-h_1 r_n,x_0+h_2 r_n]}-f_0(x_0) \\
	&=   \mathfrak{o}(\tilde{\omega}_n)+ \sum_{\bm{j} \in J_\ast} \frac{\partial^{\bm{j}} f_0(x_0) }{\bm{j}!} \cdot \frac{\Prob_n (X-x_0)^{\bm{j}} \bm{1}_{[x_0-h_1 r_n,x_0+h_2 r_n]}}{\Prob_n\bm{1}_{[x_0-h_1 r_n,x_0+h_2 r_n]}}\\
	& = \mathfrak{o}(\tilde{\omega}_n)+ \sum_{\bm{j} \in J_\ast} \frac{\partial^{\bm{j}} f_0(x_0) }{\bm{j}!} \cdot \frac{P (X-x_0)^{\bm{j}} \bm{1}_{[x_0-h_1 r_n,x_0+h_2 r_n]}}{P\bm{1}_{[x_0-h_1 r_n,x_0+h_2 r_n]}}\\
	& = \mathfrak{o}(\tilde{\omega}_n)+ \sum_{ \bm{j} \in J_\ast} \frac{ \partial^{\bm{j}} f_0(x_0)}{\bm{j}!} \prod_{k=1}^s  n^{-\beta_k j_k} I_{j_k}\big((h_1r_n)_k n^{\beta_k},(h_2r_n)_k n^{\beta_k} \big).
	\end{align*}
	The inequality in the random design case can be proved similarly, but without the constraint that $(h_2 r_n)_{k_0} n^{\beta_{k_0}}\geq 1$ since $I_q(t_1,t_2)\geq \int_0^{t_2} x^q\ \d{x}/(t_1+t_2)-t_1^q$ holds for all $t_1,t_2\geq 0$ and odd positive integer $q$.
\end{proof}

\begin{lemma}\label{lem:ratio_indicator}
	Let $P$ be a probability distribution on $[0,1]^d$ with Lebesgue density bounded away from $0$ and $\infty$, and let $g:[0,1]^d \to \R_{\geq 0}$ be a measurable function that is uniformly bounded by $1$. Then for any $\Gamma \subset [0,1]^d$, $L_n\geq 1$ and $n$ large enough, it holds with probability at least $1-n^{-2}$ that 
	\begin{align*}
	\sup_{h_1,h_2 \in \Gamma: P g\bm{1}_{[h_1,h_2]}\geq  L_n \log n/n}\biggabs{\frac{\Prob_n g\bm{1}_{[h_1,h_2]}}{Pg\bm{1}_{ [h_1,h_2]} }-1  }\leq CL_n^{-1/2}.
	\end{align*}
	Here $C>0$ is an absolute constant, and $\Prob_n$ is the empirical measure.
\end{lemma}

We need Talagrand's concentration inequality \cite{talagrand1996new} for the empirical process in the form given by Bousquet \cite{bousquet2003concentration}, recorded as follows.

\begin{lemma}[Theorem 3.3.9 of \cite{gine2015mathematical}]\label{lem:talagrand_conc_ineq}
	Let $\mathcal{F}$ be a countable class of real-valued measurable functions such that $\sup_{f \in \mathcal{F}} \pnorm{f}{\infty}\leq b$. Then
	\begin{align*}
	\Prob\bigg(\sup_{f \in \mathcal{F}}\abs{\nu_n f} \geq \E\sup_{f \in \mathcal{F}}\abs{\nu_n f} +\sqrt{2\bar{\sigma}^2 x}+b x/3\sqrt{n} \bigg)\leq e^{-x},
	\end{align*}
	where $\bar{\sigma}^2\equiv \sigma^2+2b n^{-1/2} \E \sup_{f \in \mathcal{F}} \abs{\nu_n f}$ with $\sigma^2\equiv \sup_{f \in \mathcal{F}} \mathrm{Var}_P f$, and $\nu_n\equiv \sqrt{n}(\Prob_n-P)$.
\end{lemma}

Let $J(\delta,\mathcal{F},L_2) \equiv   \int_0^\delta  \sup_Q\sqrt{1+\log \mathcal{N}(\epsilon\pnorm{F}{Q,2},\mathcal{F},L_2(Q))}\ \d{\epsilon}$, where the supremum is taken over all finitely discrete probability measures. The following local maximal inequality for the empirical process due to \cite{gine2006concentration,van2011local} is useful.
\begin{lemma}\label{lem:local_maximal_ineq}
	Let $\mathcal{F}$ be a countable class of real-valued measurable functions such that $\sup_{f \in \mathcal{F}} \pnorm{f}{\infty}\leq 1$, and $X_1,\ldots,X_n$'s are i.i.d. random variables with law $P$. Then with $\mathcal{F}(\delta)\equiv \{f \in \mathcal{F}:Pf^2<\delta^2\}$, 
	\begin{align}\label{ineq:local_maximal_uniform}
	\E \sup_{f \in \mathcal{F}(\delta)}\bigabs{\nu_n(f)}  \lesssim J(\delta,\mathcal{F},L_2)\bigg(1+\frac{J(\delta,\mathcal{F},L_2)}{\sqrt{n} \delta^2 \pnorm{F}{P,2}}\bigg)\pnorm{F}{P,2}.
	\end{align}
\end{lemma}

\begin{proof}[Proof of Lemma \ref{lem:ratio_indicator}]
	Consider the class $\mathcal{F}\equiv \{g\bm{1}_{[h_1,h_2]}: 0\leq h_1\leq h_2\leq 1\}$. Let $\alpha_n = \sqrt{L_n\log n/n}$ where $L_n\geq 1$. For $j\geq 1$, let $\mathcal{F}_j\equiv \{g\bm{1}_{[h_1,h_2]}: 2^{j-1}\alpha_n\leq \sqrt{P g\bm{1}_{[h_1,h_2]}}\leq 2^j\alpha_n\}$. Let $\ell$ be the smallest integer such that $2^\ell \alpha_n\geq 1$. Then by Talagrand's concentration inequality (cf. Lemma \ref{lem:talagrand_conc_ineq}) for bounded empirical processes, it follows that for any $s_j\geq 0$,
	\begin{align*}
	\Prob\bigg[\sup_{f \in \mathcal{F}_j} \abs{\nu_n (f)}\geq C_1\bigg(\E\sup_{f \in \mathcal{F}_j} \abs{\nu_n (f)}+ 2^{j} \alpha_n \sqrt{s_j}+\frac{s_j}{\sqrt{n}}\bigg)\bigg]\leq e^{-s_j}.
	\end{align*}
	Hence by a union bound, using the estimate (by e.g. Lemma \ref{lem:local_maximal_ineq}) that $
	\E\sup_{f \in \mathcal{F}_j} \abs{\nu_n (f)}\lesssim 2^j \alpha_n \sqrt{\log n}$, and choosing $s_j = 3 \log n$, it follows that with probability at least $1-n^{-2}$, it holds that
	\begin{align*}
	\sup_{f \in \mathcal{F}: Pf\geq \alpha_n^2}\biggabs{\frac{\Prob_n f}{Pf}-1}&=n^{-1/2}\cdot \sup_{f \in \mathcal{F}: Pf\geq \alpha_n^2}\biggabs{\frac{\nu_n(f)}{Pf}} \\
	&= n^{-1/2}\cdot \max_{1\leq j\leq \ell} \sup_{f \in \mathcal{F}_j} \biggabs{\frac{\nu_n(f)}{Pf}}\\
	&\leq n^{-1/2}\cdot \max_{1\leq j\leq \ell} \frac{\E\sup_{f \in \mathcal{F}_j} \abs{\nu_n (f)}+ 2^{j} \alpha_n \sqrt{s_j}+\frac{s_j}{\sqrt{n}}}{2^{2j-2}\alpha_n^2}\\
	&\lesssim \sqrt{\frac{\log n}{n \alpha_n^2}}+\frac{\log n}{n \alpha_n^2}\leq L_n^{-1/2}.
	\end{align*}
	The proof is complete.
\end{proof}

\begin{lemma}\label{lem:ratio_random_replace}
Let $\omega_n \searrow 0$ be such that $\omega_n^{1+\sum_k \alpha_k^{-1}}\gg \log n/n$, and let $r_n \equiv  (\omega_n^{1/\alpha_1},\ldots,\omega_n^{1/\alpha_d})$. Let $P$ be a probability measure on $[0,1]^d$ with Lebesgue density bounded away from $0$ and $\infty$ and $x_0 \in (0,1)^d$. For any $c_0\geq 1$ and $\bm{j} \in J_\ast$, we may find some $u_n \searrow 0$ such that with probability at least $1-\mathcal{O}(n^{-2})$, 
\begin{align*}
&\sup_{ \substack{ c_0^{-1}\bm{1}\leq h_1\leq c_0 \bm{1},\\ h_2\geq 0, x_0+h_2r_n\leq 1} } \tilde{\omega}_n^{-1}(h_1,h_2;\bm{j})\bigg\lvert\frac{\Prob_n (X-x_0)^{\bm{j}}\bm{1}_{[x_0-h_1r_n,x_0+h_2r_n]} }{\Prob_n \bm{1}_{[x_0-h_1r_n,x_0+h_2r_n]}}\\
&\qquad-\frac{P (X-x_0)^{\bm{j}}\bm{1}_{[x_0-h_1r_n,x_0+h_2r_n]} }{P\bm{1}_{[x_0-h_1r_n,x_0+h_2r_n]}} \bigg\rvert  \lesssim u_n+L_n^{-1/2}.
\end{align*}
Here $\tilde{\omega}_n(h_1,h_2;\bm{j}) \equiv \omega_n \cdot \prod_k \big((h_1)_k^{j_k}\vee (h_2)_k^{j_k}\big)$, and $L_n$ is a slowly growing sequence taken from Lemma \ref{lem:ratio_indicator}.
\end{lemma}
\begin{proof}
For notational simplicity, we prove the case $c_0=1$ and write $h_2\equiv h$, $\tilde{\omega}_n(h)\equiv \tilde{\omega}_n(\bm{1},h;\bm{j})$. The general case follows from minor modifications. Let $g_0(x) \equiv (x-x_0)^{\bm{j}}$. By Lemma \ref{lem:ratio_indicator}, we only need to prove with the prescribed probability,
\begin{align}\label{ineq:ratio_random_replace_1}
\sup_{h\geq 0, x_0+hr_n\leq 1} \tilde{\omega}_n^{-1}(h)\biggabs{\frac{\Prob_n \big(g_0\bm{1}_{[x_0-r_n,x_0+hr_n]}\big) }{P\bm{1}_{[x_0-r_n,x_0+hr_n]}} } \lesssim 1,
\end{align}
and
\begin{align}\label{ineq:ratio_random_replace_2}
\sup_{h\geq 0, x_0+hr_n\leq 1}\tilde{\omega}_n^{-1}(h) \biggabs{\frac{(\Prob_n-P) \big(g_0\bm{1}_{[x_0-r_n,x_0+hr_n]}\big) }{P\bm{1}_{[x_0-r_n,x_0+hr_n]}} } \lesssim u_n.
\end{align}
For (\ref{ineq:ratio_random_replace_1}), note that for any $h\geq 0$, 
\begin{align*}
P\abs{g_0} \bm{1}_{[x_0-r_n,x_0+hr_n]}\gtrsim r_n^{\bm{j}+\bm{1}} = \omega_n^{1+\sum_k \alpha_k^{-1}}\gg \log n/n,
\end{align*}
so we may apply Lemma \ref{lem:ratio_indicator} to bound $(\ref{ineq:ratio_random_replace_1})$ by
\begin{align*}
\sup_{h\geq 0, x_0+hr_n\leq 1} \tilde{\omega}_n^{-1}(h)\frac{P  \abs{g_0}\bm{1}_{[x_0-r_n,x_0+hr_n]} }{P\bm{1}_{[x_0-r_n,x_0+hr_n]}}  \lesssim \sup_{h\geq 0, x_0+hr_n\leq 1} \tilde{\omega}_n^{-1}(h) \prod_{k=1}^d \big((h\vee \bm{1}) r_n\big)_k^{j_k} \lesssim 1.
\end{align*}
Next we consider (\ref{ineq:ratio_random_replace_2}). By Lemma \ref{lem:local_maximal_ineq}, since for any $\{\ell_k\in \mathbb{Z}_{\geq 0}\}_{k=1}^d$, $\sup_{2^{\ell_k}-1\leq h_k\leq 2^{\ell_k+1}-1} \mathrm{Var}_P \big(g_0\bm{1}_{[x_0-r_n,x_0+hr_n]}\big)\lesssim r_n^{2\bm{j}+\bm{1}} \prod_{k=1}^d 2^{2\ell_k}$, we have
\begin{align*}
&\E\sup_{2^{\ell_k}-1\leq h_k\leq 2^{\ell_k+1}-1}\bigabs{\nu_n\big( g_0\bm{1}_{[x_0-r_n,x_0+hr_n]}\big) }\leq C \sqrt{r_n^{2\bm{j}+\bm{1}} 2^{\sum_k 2\ell_k} \log n}.
\end{align*}
Then with 
\begin{align*}
u_n \equiv  K \max\bigg\{C  \sqrt{\frac{\log n}{n r_n^{\bm{1}}} },\frac{\log n}{n \omega_n^{1+\sum_k \alpha_k^{-1}}} \bigg\} = \mathfrak{o}(1)
\end{align*}
for a large enough constant $K>0$, by Talagrand's concentration inequality (cf. Lemma \ref{lem:talagrand_conc_ineq}), for $n$ large,
\begin{align*}
&\Prob\bigg(\sup_{h\geq 0}\tilde{\omega}_n^{-1}(h)\biggabs{\frac{\nu_n\big( g_0\bm{1}_{[x_0-r_n,x_0+hr_n]}\big) }{P\bm{1}_{[x_0-r_n,x_0+hr_n]}} }> \sqrt{n} u_n\bigg)\\
&\leq \sum_{\ell_k\geq 0,1\leq k\leq d} \Prob\bigg(\sup_{2^{\ell_k}-1\leq h_k\leq 2^{\ell_k+1}-1} \bigabs{\nu_n\big( g_0\bm{1}_{[x_0-r_n,x_0+hr_n]}\big) }\gtrsim\sqrt{n} u_n\cdot  \omega_n r_n^{\bm{1}} 2^{\sum_k 2\ell_k } \bigg)\\
&\leq \sum_{\ell_k\geq 0,1\leq k\leq d} \Prob\bigg(\sup_{2^{\ell_k}-1\leq h_k\leq 2^{\ell_k+1}-1} \bigabs{\nu_n\big( g_0\bm{1}_{[x_0-r_n,x_0+hr_n]}\big) }\\
&\qquad - \E \sup_{2^{\ell_k}-1\leq h_k\leq 2^{\ell_k+1}-1}\bigabs{\nu_n\big( g_0\bm{1}_{[x_0-r_n,x_0+hr_n]}\big) } \gtrsim \sqrt{n} u_n\cdot \omega_n r_n^{\bm{1}} 2^{\sum_k 2\ell_k }  \bigg)\\
&\leq \sum_{\ell_k\geq 0,1\leq k\leq d} \exp\bigg(-C' \min\bigg\{\frac{n u_n^2\cdot \omega_n^2 r_n^{2\cdot \bm{1}} 2^{\sum_k 4\ell_k}}{r_n^{2\bm{j}+\bm{1}} 2^{ \sum_k 2\ell_k} }, n u_n \cdot \omega_n r_n^{\bm{1}}  2^{\sum_k 2\ell_k}\bigg\}\bigg)\\
& \leq \sum_{\ell_k\geq 0,1\leq k\leq d} \exp\big(-C' n u_n (u_n\wedge \omega_n) r_n^{\bm{1}} \cdot 2^{\sum_k 2\ell_k}\big)\\
&\leq  \sum_{\ell_k\geq 0,1\leq k\leq d}  \exp\big(-\min\big\{K^2 C'C^2, KC' \big\} \log n \cdot 2^{\sum_k 2\ell_k}\big)\leq n^{-2},
\end{align*}
as desired.
\end{proof}

\begin{lemma}\label{lem:size_difference_gaussian_small_dev}
	Let $(a,b,\gamma_\ast) \in \R_{\geq 0}^3$ be such that $a>1, 0<b<\gamma_\ast< b+(a-1)$. Let $\mathcal{H}_{a,b,\gamma_\ast}(c)\equiv \{(h_1,h_2) \in \R_{\geq 0}^d\times \R_{\geq 0}^d: 0\leq (h_1)_k\leq c^a\bm{1}_{\kappa_\ast \leq k\leq s}+(x_0)_k \bm{1}_{s+1\leq k\leq d}, 0\leq (h_1)_d \leq c^{-b}, 0\leq (h_2)_k \leq c\bm{1}_{\kappa_\ast \leq k\leq s}+(1-x_0)_k \bm{1}_{s+1\leq k\leq d},0\leq (h_2)_d\leq c^{-\gamma_\ast}\}$. Let $\G$ be defined as in Theorem \ref{thm:limit_distribution_pointwise}. Then there exists some constant $C_{d,a}>0$ such that for any $c>1$,
	\begin{align*}
	\E\sup_{(h_1,h_2)\in \mathcal{H}_{a,b,\gamma_\ast}(c)} \abs{\G(h_1,h_2)-\G(h_1,h_2\bm{1}_{[s+1:d-1]} )}\leq C_{d,a}\cdot\sigma \big(c^{as_\ast -\gamma_\ast-a\bm{1}_{s=d}}\log c\big)^{1/2}.
	\end{align*}
\end{lemma}
\begin{proof}
	Without loss of generality, we assume $\sigma=1$. Let $\mathbb{W}(h_1,h_2)\equiv \G(h_1,h_2)-\G(h_1,h_2 \bm{1}_{[s+1:d-1]})$. Then the process $\{\mathbb{W}(h_1,h_2): h_1,h_2\geq 0\}$ is a Gaussian process with the natural induced metric $d_{\mathbb{W}}$ given as follows: for $(h_1,h_2),(h_1',h_2')$,
	\begin{align*}
	&d_{\mathbb{W}}^2\big((h_1,h_2),(h_1',h_2')\big)\\
	&\equiv \E \big(\mathbb{W}(h_1,h_2)-\mathbb{W}(h_1',h_2')\big)^2\\
	& \leq 2\E \big(\G(h_1,h_2)-\G(h_1',h_2')\big)^2+2\E \big(\G(h_1,h_2 \bm{1}_{[s+1:d-1]})-\G(h_1',h_2'\bm{1}_{[s+1:d-1]})\big)^2\\
	&\equiv (I)+(II).
	\end{align*}
	For $(I)$, note that
	\begin{align*}
	(I)&\leq \bigg[ \prod_{k=\kappa_\ast}^d\big(( h_1+h_2)_k \big)- \prod_{k=\kappa_\ast}^d\big(( h_1\wedge h_1'+h_2\wedge h_2')_k \big)\\
	&\qquad\qquad+ \prod_{k=\kappa_\ast}^d\big((h_1'+h_2')_k \big)-\prod_{k=\kappa_\ast}^d\big(( h_1\wedge h_1'+h_2\wedge h_2')_k\big) \bigg]\\
	&\leq C_{d_\ast}c^{d_\ast}\big(\pnorm{h_1-h_1'}{}+\pnorm{h_2-h_2'}{}\big).
	\end{align*}
	The last line follows from Lemma \ref{lem:product_difference}. A similar estimate can be obtained for $(II)$. Hence,
	\begin{align*}
	d_{\mathbb{W}}^2\big((h_1,h_2),(h_1',h_2')\big)&\leq C_{d_\ast}c^{d_\ast}\big(\pnorm{h_1-h_1'}{}+\pnorm{h_2-h_2'}{}\big).
	\end{align*}
	On the other hand, the diameter of the indexing set in the supremum of the process in question can be controlled as follows: If $s<d$,
	\begin{align*}
	D^2_c &\equiv \sup_{(h_1,h_2), (h_1',h_2')\in \mathcal{H}_{a,b,\gamma_\ast}(c)} d_{\mathbb{W}}^2\big((h_1,h_2),(h_1',h_2')\big) \\
	&\leq  C\sup_{(h_1,h_2)\in \mathcal{H}_{a,b,\gamma_\ast}(c)} \bigg(\prod_{k=\kappa_\ast}^d (h_1+h_2)_k-\prod_{k=\kappa_\ast}^d (h_1+h_2\bm{1}_{[s+1:d-1]})_k\bigg)\\
	&\leq C\big[(c^a+c)^{s_\ast}(c^{-b}+c^{-\gamma_\ast})-(c^a)^{s_\ast} c^{-b}\big] \\
	&\leq C\big[c^{as_\ast-b}(1+c^{-(\gamma_\ast-b)})(1+c^{-(a-1)})^{s_\ast}- c^{as_\ast -b}\big]\\
	&\leq C\big[c^{as_\ast-b}(1+C_d' c^{-(\gamma_\ast-b)\wedge (a-1)})- c^{as_\ast -b}\big] = C_d \cdot c^{as_\ast -\gamma_\ast}.
	\end{align*}
	If $s=d$, similarly we have $
	D_c^2\leq C_d\cdot c^{a(s_\ast-1)-\gamma_\ast}$. Now Dudley's entropy integral (cf. Lemma \ref{lem:dudley_entropy_integral}) yields that
	\begin{align*}
	&\E\sup_{(h_1,h_2)\in \mathcal{H}_{a,b,\gamma_\ast}(c)} \abs{\G(h_1,h_2)-\G(h_1,h_2\bm{1}_{[s+1:d-1]})}\\
	&\leq C \int_0^{D_c} \sqrt{\log \mathcal{N}(\epsilon, \{0\leq h_1\leq c^a\bm{1}, 0\leq h_2\leq c\bm{1} \}, d_{\mathbb{W}})}   \ \d{\epsilon}\\
	&\leq C \int_0^{ (C_{d} c^{as_\ast -\gamma_\ast-a\bm{1}_{s=d}})^{1/2} } \sqrt{\log (c^{C_1(a,d)}/\epsilon^{2d})}\ \d{\epsilon}\\
	&\leq C_{a,d}  \sqrt{c^{as_\ast -\gamma_\ast-a\bm{1}_{s=d}}\log c},
	\end{align*}
	as desired.
\end{proof}

\begin{lemma}\label{lem:product_difference}
	For any $c_k \geq 1, k=\kappa_\ast,\ldots,d$, let $g: \prod_{k=\kappa_\ast}^d [0,c_k] \times \prod_{k=\kappa_\ast}^d [0,c_k] \to \R$ be defined as $
	g(h_1,h_2)\equiv \prod_{k=\kappa_\ast}^d( h_1+h_2)_k$. Then
	\begin{align*}
	\abs{g(h_1,h_2)-g( h_1', h_2')}\leq \prod_{k=\kappa_\ast}^d (2c_k)\cdot \sqrt{d_\ast}\big(\pnorm{h_1-h_1'}{}+\pnorm{h_2-h_2'}{}\big).
	\end{align*}
	Here $\pnorm{\cdot}{}$ is the canonical Euclidean norm on $\R^{d_\ast}$.
\end{lemma}
\begin{proof}
	Note that the first order partial derivatives of $g(\cdot,\cdot)$ are all bounded by $\prod_{k=\kappa_\ast}^d (2c_k)$, using Taylor expansion yields that
	\begin{align*}
	\abs{g(h_1,h_2)-g( h_1', h_2')}
	&\leq  \prod_{k=\kappa_\ast}^d (2c_k) \sum_{k=\kappa_\ast}^d \bigabs{(h_1)_k-(h_1')_k}+\bigabs{(h_2)_k-(h_2')_k}\\
	&\leq  \prod_{k=\kappa_\ast}^d (2c_k)\cdot \sqrt{d_\ast}\cdot  \big(\pnorm{h_1-h_1'}{}+\pnorm{h_2-h_2'}{}\big),
	\end{align*}
	completing the proof.
\end{proof}

\begin{lemma}\label{lem:small_deviation_min_max}
	Let $\G$ be defined in Theorem \ref{thm:limit_distribution_pointwise}. Then 
	\begin{align*}
	\Prob\bigg(\min\nolimits_{\substack{0\leq (h_2)_k \leq \bm{1}_{s+1\leq k\leq d}\\ (h_2)_d = 0} } \max\nolimits_{0\leq (h_1)_k\leq \bm{1}_{\kappa_\ast\leq k\leq d}} \G(h_1,h_2)\leq 0\bigg)= 0.
	\end{align*}
\end{lemma}
In particular, for any $\epsilon>0$, there exists some $\rho_{\epsilon}>0$ such that
\begin{align*}
\Prob\bigg(\min\nolimits_{\substack{0\leq (h_2)_k \leq \bm{1}_{s+1\leq k\leq d}\\ (h_2)_d = 0} } \max\nolimits_{0\leq (h_1)_k\leq \bm{1}_{\kappa_\ast\leq k\leq d}} \G(h_1,h_2)>\rho_{\epsilon}\bigg)\geq 1-\epsilon.
\end{align*}
\begin{proof}
Without loss of generality, we may assume that $\kappa_\ast = s+1$, since otherwise we may consider the weaker statement:
\begin{align*}
\Prob\bigg(\min\nolimits_{\substack{0\leq (h_2)_k \leq \bm{1}_{s+1\leq k\leq d}\\ (h_2)_d = 0} } \max\nolimits_{\substack{\bm{1}_{\kappa_\ast \leq k\leq s}\leq (h_1)_k\leq \bm{1}_{\kappa_\ast\leq k\leq d}} } \G(h_1,h_2)\leq 0\bigg)= 0.
\end{align*}
If $\kappa_\ast = d$, the claim follows from the standard reflection principle for Brownian motion, so we work with $\kappa_\ast \leq d-1$. Let $\mathcal{I} \equiv \{S_0,S_1,\ldots,S_J\}$, where $0=S_0<S_1<\ldots<S_J=1$ will be determined later on. For $\mathscr{I} \subset \{\kappa_\ast,\ldots,d\}$, let $\mathbb{W}_{\mathscr{I}}$ be the standard white noise living on the coordinates $\mathscr{I}$, and write $\mathbb{W}\equiv \mathbb{W}_{\{\kappa_\ast,\ldots,d\}}$ for notational simplicity. For any $K>0$,
\begin{align*}
&\Prob\bigg(\min\nolimits_{\substack{0\leq (h_2)_k \leq \bm{1}_{s+1\leq k\leq d}\\ (h_2)_d = 0} } \max\nolimits_{0\leq (h_1)_k\leq \bm{1}_{\kappa_\ast\leq k\leq d}} \G(h_1,h_2)\leq 0\bigg)\\
&\leq 	\Prob\bigg(\forall (h_1)_d \in \mathcal{I}, \min\nolimits_{\substack{0\leq (h_2)_k \leq \bm{1}_{s+1\leq k\leq d}\\ (h_2)_d = 0} } \max\nolimits_{0\leq (h_1)_k\leq \bm{1}_{\kappa_\ast\leq k\leq d-1}} \mathbb{W}([-h_1,h_2])\leq 0\bigg)\\
&\leq \Prob\bigg(\forall 1\leq j\leq J, \min\nolimits_{\substack{0\leq (h_2)_k \leq \bm{1}_{s+1\leq k\leq d}\\ (h_2)_d = 0} } \max\nolimits_{0\leq (h_1)_k\leq \bm{1}_{\kappa_\ast\leq k\leq d-1}}\\
&\qquad\qquad \mathbb{W}\bigg(\big[\big(-(h_1)_{\kappa_\ast},\ldots,-S_{j}\big),\big((h_2)_{\kappa_\ast},\ldots, -S_{j-1}\big)\big]\bigg)\leq K\sqrt{ S_{j-1}}\bigg)\\
&\qquad +\sum_{j=2}^J \Prob\bigg(\max\nolimits_{\substack{0\leq (h_2)_k \leq \bm{1}_{s+1\leq k\leq d}\\ (h_2)_d = 0} } \max\nolimits_{0\leq (h_1)_k\leq \bm{1}_{\kappa_\ast\leq k\leq d-1}}\\
&\qquad\qquad  \bigg\lvert \mathbb{W}\bigg(\big[\big(-(h_1)_{\kappa_\ast},\ldots,-S_{j}\big),\big((h_2)_{\kappa_\ast},\ldots, 0\big)\big]\\
&\qquad\qquad\qquad -\mathbb{W}\bigg(\big[\big(-(h_1)_{\kappa_\ast},\ldots,-S_{j}\big),\big((h_2)_{\kappa_\ast},\ldots, -S_{j-1}\big)\big]\bigg)\bigg\rvert> K\sqrt{ S_{j-1}}\bigg)\\
&\equiv (I)+(II).
\end{align*}
First consider $(II)$. By computing the covariance structure,  $(II)$ can be further reduced as follows: there exists some constant $C_d>0$, such that for $K$ large enough,
\begin{align*}
(II)&  \leq J \cdot \Prob\bigg(C_d' \sup_{v \in [0,1]^{d_\ast-1}}  \mathbb{W}_{\{\kappa_\ast,\ldots,d-1\}}([0,v])>K\bigg)\leq C_d\cdot J \exp(-K^2/C_d).
\end{align*}
Here in the last step we used standard estimates for Gaussian processes (via e.g. Dudley's entropy integral in Lemma \ref{lem:dudley_entropy_integral}) and Gaussian concentration (cf. Lemma \ref{lem:Gaussian_concentration}). 

Next we consider $(I)$. Since the white noise processes involved in $(I)$ are independent for different $j$'s,
\begin{align*}
(I)&\leq \prod_{j=1}^J \Prob\bigg( Z\leq K\sqrt{\frac{S_{j-1}}{S_j-S_{j-1}} }\bigg).
\end{align*}
Here $Z \equiv \min_{0\leq u\leq 1}\max_{0\leq v\leq 1} \mathbb{W}_{\{\kappa_\ast,\ldots,d-1\}}([-u,v])$. Hence, for any $\mathcal{I} =\{S_0,S_1,\ldots,S_J\}$, the probability in question is bounded by
\begin{align*}
C_d\cdot J\exp(-K^2/C_d)+\prod_{j=1}^J \Prob\bigg( Z\leq K\sqrt{\frac{S_{j-1}}{S_j-S_{j-1}} }\bigg).
\end{align*}
Fix $\delta>0,\epsilon>0$. Let $S_j\equiv 2^{2^{j}}\delta$ for $j\geq 1$. Then $J=\log_2\log_2(1/\delta)$. We choose $\delta>0$ such that $J$ is an even integer. For $j\geq 2$, we have
\begin{align*}
\frac{S_{j-1}}{S_j-S_{j-1}} =\frac{2^{2^{j-1}}}{ 2^{2^{j}}-2^{2^{j-1}} } \leq \frac{2}{2^{2^{j-1}}}. 
\end{align*}
The above inequality apparently holds for $j=1$ as $S_0=0$ by definition. Let $K\equiv \sqrt{C_d \log (C_d J/\epsilon)}$.  Then for $J$ large, the probability in question can be further bounded by
\begin{align*}
&\epsilon+ \prod_{j=1}^J \Prob\bigg(Z\leq \frac{\sqrt{2C_d \log (C_d J/\epsilon) }}{2^{2^{j-2}}}\bigg)\leq \epsilon+ \prod_{j=J/2}^J \Prob\bigg(Z\leq \frac{\sqrt{2C_d \log (C_d J/\epsilon) }}{2^{2^{J/2-2}}}\bigg).
\end{align*}
If there exist $\epsilon_0,\tau_0>0$ such that $\Prob(Z\leq \tau_0)\leq 1-\epsilon_0<1$, the above probability vanishes as $J \to \infty$ (equivalently $\delta \to 0$) followed by $\epsilon \to 0$, proving the lemma. Now we prove this claim. We focus on the case $\mathbb{W}_2$ to illustrate the main proof idea; the other cases follow from minor modifications but with substantial complications in notation. Note that for any $\tau>0$,
\begin{align*}
&\Prob\bigg(\min_{0\leq u\leq 1}\max_{0\leq v\leq 1} \mathbb{W}_{2}([-u,v])\leq \tau\bigg)\\
&\leq \Prob\bigg(\min_{0\leq u\leq 1}\max_{0\leq v_1\leq 1, v_2= 1} \mathbb{W}_{2}([-u,v])\leq \tau\bigg)\\
&\leq \Prob\bigg(-\sup_{0\leq u\leq 1}\abs{\mathbb{W}_2([-u,0])} - \sup_{0\leq u_1\leq 1}\abs{\mathbb{W}_2([(-u_1,0),(0,1)])}\\
&\qquad - \sup_{\substack{0\leq u_2\leq 1 \\0\leq v_1\leq 1}}\abs{\mathbb{W}_2([(0,-u_2),(v_1,0)])}+\sup_{0\leq v_1\leq 1}\abs{\mathbb{W}_2([(0,0),(v_1,1)])}\leq \tau\bigg)\\
&\leq \Prob\big(\abs{N}\leq Y_1+Y_2+Y_3+\tau\big).
\end{align*}
Here $N,Y_1,Y_2,Y_3$ are independent random variables, where $N$ is distributed as standard normal, and $Y_i$'s are i.i.d. copies of $\sup_{0\leq v\leq 1} \abs{\mathbb{W}_2([0,v])}$. The above probability is strictly less than $1$ as $\tau \to 0$, proving the claim. The proof is complete.
\end{proof}

\section*{Acknowledgments}
The authors would like to thank four referees for offering a large number of suggestions that significantly improve the presentation of the article.

\bibliographystyle{amsalpha}
\bibliography{mybib.bib}

\end{document}